\documentclass[11pt, final, english]{article}

\usepackage[english]{babel}    
\usepackage[utf8]{inputenc}    
\usepackage[T1]{fontenc}       

\usepackage{amsmath}           
\let\hbarorig\hbar
\usepackage{amsfonts}          
\let\hbar\hbarorig
\usepackage[amsmath,thmmarks,hyperref]{ntheorem} 

\usepackage[sort]{cite}        
\usepackage{setspace}          
\usepackage{tikz-cd}           
\usepackage{exscale}           
\usepackage{xspace}            
\usepackage{bbm}               
\usepackage{mathtools}         
\usepackage{mathrsfs}          
\usepackage{stmaryrd}          

\usepackage[expansion=false    
           ]{microtype}        
\usepackage[nottoc]{tocbibind} 
\usepackage[final=true,        
           pdfpagelabels       
           ]{hyperref}         
\usepackage[a4paper]{geometry} 

\newcommand{\refitem}[1] {\textit{\ref{#1}.)}}

\numberwithin{equation}{section}

\allowdisplaybreaks

\newcommand{\NN}{\mathbbm{N}}
\newcommand{\ZZ}{\mathbbm{Z}}
\newcommand{\RR}{\mathbbm{R}}

\newcommand{\CC}{\mathbbm{C}}
\newcommand{\PP}{\mathbbm{P}}
\newcommand{\DD}{\mathbbm{D}}
\newcommand{\Unit}{\mathbbm 1}
\newcommand{\CCx}{\CC^*}

\newcommand{\Smooth}{\mathscr{C}^\infty}
\newcommand{\Polynomials}{\mathscr{P}}
\newcommand{\Holomorphic}{\mathscr{O}}
\newcommand{\Analytic}{\mathscr{A}}

\newcommand{\SmoothSections}{\Gamma^\infty}
\newcommand{\Tangent}{\mathrm{T}}
\newcommand{\CoTangent}{\mathrm{T}^*}

\newcommand{\E}{\mathrm{e}}
\newcommand{\I}{\mathrm{i}}
\newcommand{\D}{\mathrm{d}}
\newcommand{\bigproj}{\Theta}
\newcommand{\StarDomain}{\Omega}

\newcommand{\momentmap}[1][]{\mathcal{J}^{#1}}
\newcommand{\momentmapExt}[1][]{{\hat{\mathcal{J}}^{#1}}}

\newcommand{\startilde}{\mathbin{\tilde{\star}}}
\newcommand{\starred}[1][]{\mathbin{\star_{\red #1}}}
\newcommand{\starreds}[2]{\mathbin{\star_{\red #1}^{#2}}}

\newcommand{\Mred}[1][]{M_{\red}^{#1}}
\newcommand{\Levelset}[1][]{Z^{#1}}
\newcommand{\MredExt}[1][]{{{\hat M}_{\red}}^{#1}}
\newcommand{\LevelsetExt}[1][]{{\hat Z}^{#1}}
\newcommand{\CCres}[1][1+n]{\CC_+^{#1}}

\newcommand{\Diag}{\Delta}
\newcommand{\DiagZ}{\Delta_Z}
\newcommand{\DiagM}{\Delta_{\red}}
\DeclareMathOperator{\pr}{pr}
\DeclareMathOperator{\bigpr}{Pr}

\DeclareMathOperator{\IM}{Im}

\DeclareMathOperator{\sgn}{sgn}
\DeclareMathOperator{\Deg}{Deg}
\DeclareMathOperator{\id}{id}

\newcommand{\Ten}[1][\bullet]{\mathcal{T}^{#1}}
\newcommand{\SymOp}[1][\bullet]{\mathrm{Sym}^{#1}}
\newcommand{\AsymOp}[1][\bullet]{\mathrm{Asym}^{#1}}
\newcommand{\SymSec}{\mathscr{S}}
\newcommand{\ASymSec}{\mathscr{A}}
\newcommand{\Symten}[1]{\mathrm{S}^{#1}\,}
\newcommand{\ASymten}[1]{\Lambda^{#1}\,}

\newcommand{\group}[1]{\mathrm{#1}}
\newcommand{\lie}[1]{\mathfrak{#1}}
\newcommand{\GL}[1][1+n]{\group{GL}(#1,\CC)}
\newcommand{\gl}[1][1+n]{\lie{gl}_{#1}(\CC)}
\newcommand{\Stab}[1][]{\group G_{\momentmap[{#1}]}}
\newcommand{\stab}{\lie g_{\momentmap}}
\newcommand{\StabExt}[1][]{\group G_{\momentmapExt[{#1}]}}

\newcommand{\acts}{\mathbin{\triangleright}}
\newcommand{\racts}{\mathbin{\triangleleft}}

\newcommand{\sym}{\mathrm{sym}}
\newcommand{\red}{\mathrm{red}}
\newcommand{\hol}{\mathrm{hol}}
\newcommand{\antihol}{\cc{\hol}}

\newcommand{\Monom}[2]{\mathrm{b}_{#1,#2}^{}}
\newcommand{\MonomRed}[3][]{\mathrm{b}_{#2,#3;\red}^{#1}}
\newcommand{\MonomRedFund}[3][]{\mathrm{c}_{#2,#3}^{#1}}

\newcommand{\expansionCoefficients}[3]{#1^{}_{#2,#3}}

\newcommand{\geomWickRotRed}[1][]{\alpha^{#1}}
\newcommand{\WickRot}[1][]{\Phi^{#1}}
\newcommand{\WickRotRed}[1][]{\Phi_\red^{#1}}

\newcommand{\komma}{\, ,}
\newcommand{\punkt}{\, .}

\newcommand{\cc}[1]{\overline{#1}}
\newcommand{\argument}{\,\cdot\,}
\newcommand{\formal}[1]{[\![#1]\!]}
\newcommand{\neu}[1]{\emph{#1}}  
\newcommand{\at}[2][]{#1|_{#2}}
\newcommand{\signature}[1][s]{(#1)}

\DeclareFontFamily{U}{FdSymbolF}{}
\DeclareFontShape{U}{FdSymbolF}{m}{n}{
	<-7.1> FdSymbolF-Book
	<7.1-> FdSymbolF-Book
}{}
\DeclareSymbolFont{delimiters}{U}{FdSymbolF}{m}{n}
\DeclareMathDelimiter{\llangle}{\mathopen}{delimiters}{"92}{delimiters}{"92}
\DeclareMathDelimiter{\rrangle}{\mathclose}{delimiters}{"98}{delimiters}{"98}

\DeclarePairedDelimiter{\ordinarySet}{\{}{\}}
\DeclarePairedDelimiter{\ordinaryIP}{\langle}{\rangle}
\DeclarePairedDelimiter{\ordinaryKom}{[}{]}
\DeclarePairedDelimiter{\ordinaryBracket}{(}{)}
\DeclarePairedDelimiter{\coolIP}{\llangle}{\rrangle}

\newcommand{\set}[3][]{\ordinarySet[#1]{\,#2 \;#1|\; #3\,}}
\newcommand{\poi}[3][]{\ordinarySet[#1]{\,#2\,,\,#3\,}}
\newcommand{\dupr}[3][]{\ordinaryIP[#1]{\,#2 \,,\, #3\,}}
\newcommand{\kom}[3][]{\ordinaryKom[#1]{\,#2\,,\,#3\,}}
\newcommand{\genHull}[2][]{\coolIP[#1]{#2}}
\newcommand{\falling}[3][]{\ordinaryBracket[#1]{#2}_{\downarrow, #3}}
\newcommand{\rising}[3][]{\ordinaryBracket[#1]{#2}_{\uparrow, #3}}
\newcommand{\abs}[2][]{#1|#2#1|}
\newcommand{\seminorm}[3][]{#1|#1|#3#1|#1|_{#2}}

\theoremheaderfont{\normalfont\bfseries}
\theorembodyfont{\itshape}
\newtheorem{lemma}{Lemma}[section]
\newtheorem{proposition}[lemma]{Proposition}
\newtheorem{theorem}[lemma]{Theorem}
\newtheorem{corollary}[lemma]{Corollary}
\newtheorem{definition}[lemma]{Definition}

\theorembodyfont{\rmfamily}
\theorembodyfont{\itshape}
\theoremnumbering{Roman}
\newtheorem{maintheorem}{Main Theorem}

\theorembodyfont{\normalfont}
\makeatletter
\def\thmhead@plain#1#2#3{%
	\thmname{#1}\thmnumber{\@ifnotempty{#1}{ }\@upn{#2}}%
	\thmnote{ {\the\thm@notefont#3}}}
\let\thmhead\thmhead@plain
\makeatother

\theoremheaderfont{\scshape}
\theorembodyfont{\normalfont}
\theoremstyle{nonumberplain}
\theoremseparator{:}
\theoremsymbol{\hbox{$\boxempty$}}
\newtheorem{proof}{Proof}
\theoremstyle{empty}
\newtheorem{Empty}{}

\author{
  \textbf{Philipp Schmitt}%
  \thanks{\href{mailto:philippschmitt@math.ku.dk}{\texttt{philippschmitt@math.ku.dk}},
    supported by the Danish National Research Foundation through 
    the Centre of Symmetry and Deformation (DNRF92).
  }
  \\[0.5cm]
  Department of Mathematical Sciences\\
  University of Copenhagen\\
  Universitetsparken 5, 
  2100 København \O{}\\ 
  Denmark\\[1cm]
  \textbf{Matthias Schötz}%
  \thanks{\href{mailto:Matthias.Schotz@ulb.ac.be}{\texttt{Matthias.Schotz@ulb.ac.be}},
    Boursier de l'ULB, 
    supported by the Fonds de la Recherche Scientifique (FNRS) and the 
    Fonds Wetenschappelijk Onderzoek - Vlaaderen (FWO) under EOS Project n$^\circ$30950721.
  }
  \\[0.5cm]
  Département de Mathématiques\\
  Université libre de Bruxelles CP 218\\
  Boulevard du Triomphe, Bruxelles 1050\\
  Belgium\\[1cm]
}

\title{Wick Rotations in Deformation Quantization\texorpdfstring{\\[1cm]}{}}

\date{January 2020}

\begin{document}
	\maketitle
	\begin{abstract}
    We study formal and non-formal deformation quantizations of a family of manifolds 
    that can be obtained by phase space reduction from $\CC^{1+n}$ with the Wick star
    product in arbitrary signature. Two special cases of such manifolds  are the complex projective
    space $\CC\PP^n$ and the complex hyperbolic disc $\DD^n$. 
    We generalize several older results to this setting:
    The construction of formal star products and their explicit description by bidifferential
    operators, the existence of a convergent subalgebra of ``polynomial'' functions, and 
    its completion to an algebra of certain analytic functions that allow an easy characterization
    via their holomorphic extensions. Moreover, we find an isomorphism between the non-formal
    deformation quantizations for different signatures, linking e.g.\ the star products on
    $\CC\PP^n$ and $\DD^n$. More precisely, we describe an isomorphism between the (polynomial or
    analytic) function algebras that is compatible with Poisson brackets and
    the convergent star products. This isomorphism is essentially given by Wick rotation,
    i.e.\ holomorphic extension of analytic functions and restriction to a new domain.
    It is not compatible with the $^*$-involution of pointwise complex conjugation.
	\end{abstract}


    \newpage
	\tableofcontents

\begin{onehalfspace}




\section{Introduction}      \label{sec:introduction}

One way to study the quantization problem arising in physics, which asks how to associate a quantum 
mechanical system to a classical mechanical one, is \emph{formal deformation quantization} as introduced in 
\cite{bayen.flato.fronsdal.lichnerowicz.sternheimer:DeformationTheoryAndQuantization}. In this approach,
the classical observable algebra is assumed to be the algebra $\Smooth(M)$ of smooth functions on a Poisson 
manifold $M$ and one tries to find a so-called \emph{formal star product} 
$\star$ that deforms the classical product. More precisely,
$\star \colon \Smooth(M)\formal{\lambda} \times \Smooth(M)\formal{\lambda} \to 
\Smooth(M)\formal{\lambda}$ is called a formal star product if it is 
$\CC\formal\lambda$-bilinear, associative, has the constant $1$-function as a unit,
and if it can be expanded as $f \star g = \sum_{r=0}^\infty \lambda^r C_r(f,g)$
with $\CC\formal\lambda$-linear extensions of bidifferential operators
$C_r \colon \Smooth(M) \times \Smooth(M) \to \Smooth(M)$ that satisfy
that $C_0(f,g) = fg$ is the usual commutative pointwise product and that 
$C_1(f,g) - C_1(g,f) = \I \poi{f}{g}$ is (up to the factor $\I$) the Poisson bracket of $f, g \in \Smooth(M)$.
Here $\formal{\lambda}$ denotes formal power series in the parameter $\lambda$.
We say that $\star$ deforms \neu{in direction of} the Poisson bracket $\poi{\argument}{\argument}$.
Such a star product is called \neu{Hermitian} if 
pointwise complex conjugation is a $^*$-involution, i.e.\ if
$\cc{f \star g} = \cc{g} \star \cc{f}$ holds
for all $f,g\in \Smooth(M)$.
In a sense, formal deformation quantization transfers the quantization problem to algebra and therefore 
allows to use powerful algebraic tools in its study. For example, existence and classification results follow 
from Kontsevich's formality theorem in the most general case of Poisson manifolds 
\cite{kontsevich:DeformationQuantizationOfPoissonManifolds}, but were already proven before in the special 
case of symplectic manifolds by various authors \cite{bertelson.cahen.gutt:EquivalenceOfStarProducts, 
nest.tsygan:AlgebraicIndexTheorem, fedosov:ASimpleGeometricConstructionOfDeformationQuantization, 
dewilde.lecomte:ExistenceOfStarProductsAndOfFormalDeformationsOfThePoissonLieAlgebraOfArbitrarySymplecticManifolds}
 and with the help of different techniques, e.g.\ the so-called Fedosov construction.

Formal deformation quantizations can also be studied in an equivariant setting. Assume $G$ is a Lie group 
acting on $M$. Then a star product is called \emph{$G$-invariant} if all the bidifferential operators $C_r$ 
are $G$-invariant. For Hamiltonian $G$-actions there is a related notion of $G$-equivariance that
considers the quantization of a momentum map as well. Existence and classification results are also available 
in this setting \cite{bertelson.bieliavsky.gutt:ParametrizingEquivalenceClassesOfInvariantStarProducts,
reichert.waldmann:ClassificationOfEquivariantStarproductsOnSymplecticManifolds,
esposito.klein.schnitzer:AProofOfTsygansFormalityConjectureForHamiltonianActions}. 
Some explicit examples of star products can easily be obtained on $\CC^{1+n}$,
namely the exponential star products like Moyal /  Weyl–Groenewold or Wick star products.
There are also explicit methods to obtain star products on more general spaces, like $\CC\PP^n$ or $\DD^n$. 
\cite{bordemann.brischle.emmrich.waldmann:PhaseSpaceReductionForStarProducts.ExplicitConstruction,
	bordemann.brischle.emmrich.waldmann:SubalgebrasWithConvergingStarProducts,
	beiser.waldmann:FrechetAlgebraicDeformationOfThePoincareDisc,
	kraus.roth.schoetz.waldmann:OnlineConvergentStarProductOnPoincareDisc}
use a construction via phase space reduction from one of the aforementioned products on $\CC^{1+n}$. 
Alternatively, one can e.g.\ use Berezin dequantization \cite{cahen.gutt.rawnsley:QuantisationOfKaehlerManifolds3},
a Lie algebraic approach 
\cite{alekseev.lachowska:InvariantStarProductsOnCoadjointOrbitsAndTheShapovalovPairing}
or an explicit solution of the recursive equations coming from the Fedosov construction 
\cite{loeffler:FedosovDifferentialsAndCartanNumbers}.

The drawback of considering formal power series is that one cannot easily replace the formal 
parameter
$\lambda$ by Planck's constant $\hbar$, as required in actual physical applications. 
Therefore \emph{strict quantization} asks to find a field of well-behaved algebras, 
usually Fr\'echet $^*$-algebras or C$^*$-algebras, see 
\cite{bieliavsky.gayral:DeformationQuantizationForActionsOfKaehlerianLieGroups, 
rieffel:DeformationQuantizationForActionsOfRd, 
natsume.nest.peter:StrictQuantizationOfSymplecticManifolds},
that depend nicely on a parameter $\hbar$ ranging over some subset of $\CC$, 
and that reproduce the usual product and Poisson bracket in the zeroth and first order as above for $\hbar 
\to 0$. Usually, strict quantizations as in \cite{rieffel:DeformationQuantizationForActionsOfRd, 
bieliavsky.gayral:DeformationQuantizationForActionsOfKaehlerianLieGroups} are constructed by analytical 
methods, involving oscillatory integrals.
If a strict quantization depends smoothly on the parameter $\hbar$, its asymptotic expansion around 
$\hbar = 0$ yields a formal deformation quantization.  
Conversely, one can ask to construct strict quantizations that have a given formal deformation quantization 
as their limit.

Some results in this direction were obtained by Waldmann and collaborators, who try to find some 
distinguished subalgebra $\Polynomials(M)$ of $\Smooth(M)$, on which a star product converges trivially because the formal 
power series are finite.
Such a choice usually comes from some extra structure, for example if $M = \CoTangent Q$ is a cotangent space 
one can try to use functions that are polynomial in the momenta. One then tries to find some 
topology with respect to which a star product on $\Polynomials(M)$ is continuous, in order to complete 
$\Polynomials(M)$ to a more interesting algebra $\Analytic(M)$, typically consisting of analytic functions.
This approach has been worked out e.g.\ for star 
products of exponential type on possibly infinite-dimensional vector spaces 
\cite{schoetz.waldmann:ConvergentStarProductsForProjectiveLimitsOfHilbertSpaces, 
waldmann:ANuclearWeylAlgebra}, for the Gutt star 
product on the dual of a Lie algebra \cite{esposito.stapor.waldmann:ConvergenceOfTheGuttStarProduct},
for the 2-sphere \cite{esposito.schmitt.waldmann:OnlineComparisonContinuityWickStarProducts}, for the 
hyperbolic disc $\DD^n$ \cite{beiser.waldmann:FrechetAlgebraicDeformationOfThePoincareDisc,
kraus.roth.schoetz.waldmann:OnlineConvergentStarProductOnPoincareDisc}, and for semisimple 
coadjoint orbits of semisimple connected Lie groups \cite{schmitt:StrictQuantizationOfCodjointOrbits}. 
In the case of the hyperbolic disc the completed algebra $\Analytic$ has a nice geometric 
interpretation as functions that allow an extension to holomorphic functions on some fixed, larger space.

In this article we generalize the approach used in 
\cite{kraus.roth.schoetz.waldmann:OnlineConvergentStarProductOnPoincareDisc}
for the hyperbolic disc to obtain formal and non-formal star products on a larger class of certain 
(pseudo-)K\"ahler manifolds.
These manifolds depend on two parameters, dimension $n$ and signature $s$, and are obtained by using 
Marsden--Weinstein reduction 
for the canonical $\group{U}(1)$-action on $\CC^{1+n}$ endowed with a metric of signature $s$. 
Focussing on treating all these examples in a uniform way, we construct $\group{U}(s,1+n-s)$-invariant, Hermitian formal star products.
Using ideas relating to K\"ahler reduction, we derive an explicit formula
in Theorem~\ref{theorem:reducedWickProductFormula}:
\newpage
\begin{maintheorem}
	For any of the reduced (pseudo-)Kähler manifolds $\Mred$ described above, the formula
	\begin{equation}
	f \starred g
	=
	\sum_{r=0}^\infty \frac 1 {r!} \frac {\lambda^r} {\left(1-\lambda\right)\left(1 -2\lambda\right) \dots \left(1 -(r-1)\lambda\right)} 
	\dupr[\big]{(D_\red^\sym)^r f \otimes (D_\red^\sym)  ^r g }{ H_\red^r }
	\label{eq:intro:formula}
	\end{equation}
	defines a formal star product. Here $f,g\in \Smooth(\Mred)$, $D_\red^\sym$ is the symmetrized covariant derivative 
	associated to the Levi-Civita connection of $\Mred$, and $H_\red$ a certain bivector field on $\Mred$.
\end{maintheorem}
This formula was already known in the special case of $\CC\PP^n$ and $\DD^n$, 
\cite{loeffler:FedosovDifferentialsAndCartanNumbers}, where it was derived from the Fedosov construction.
Our result therefore allows to compare this approach with phase space reduction without appealing to any 
abstract classification results, and generalizes it to a larger class of manifolds.

It will become clear from the construction that, at least outside of the poles appearing in
\eqref{eq:intro:formula},
the star product $\starred$ converges trivially for a class of functions 
$\Polynomials(\Mred)$ that is obtained by reducing polynomials on $\CC^{1+n}$. 
All these functions can be (uniquely) extended to holomorphic functions on a larger complex manifold $\MredExt$ that 
can be obtained by an analogous reduction procedure from $\CC^{1+n} \times \CC^{1+n}$.
We define the algebra $\Analytic(\Mred)$ of all functions that can be extended to holomorphic 
functions on $\MredExt$, thus obtaining an algebra of certain analytic functions.
Using methods from complex analytic geometry, we prove that 
$\Polynomials(\Mred)$ is dense in $\Analytic(\Mred)$ with respect to the topology of 
locally uniform convergence of the extensions to $\MredExt$. 
Then we obtain for all complex $\hbar$ outside of the poles of
\eqref{eq:intro:formula} our Theorem~\ref{theorem:strictProduct}:

\begin{maintheorem} The strict product $\starred[,\hbar]$ on $\Polynomials(\Mred)$ obtained by replacing 
the formal parameter $\lambda$ with $\hbar$ in \eqref{eq:intro:formula}, is continuous 
with respect to the topology of locally uniform convergence of the holomorphic extensions to $\MredExt$. It therefore 
extends uniquely to a continuous product on $\Analytic(\Mred)$.
\end{maintheorem}
The geometries of the manifolds $\Mred$ can be quite different (e.g.\ sometimes compact, sometimes not).
However, both the classical and quantum algebras of analytic functions cannot see this 
difference as we show in Theorem~\ref{theorem:wickRotation} and Theorem~\ref{theorem:wickstar} using
essentially a generalization of the Wick rotation:

\begin{maintheorem}
	The algebras $\Analytic(\Mred)$ (for the same dimension $n$ but different signatures $s$) with the pointwise product are all isomorphic as unital Fréchet algebras.
\end{maintheorem}
\begin{maintheorem}
	The algebras $\Analytic(\Mred)$ (for the same dimension $n$ but different signatures $s$) with the product $\starred[,\hbar]$ and fixed $\hbar$
	are all isomorphic as unital Fréchet algebras.
\end{maintheorem}
Note that these last two results
can also be proven in a more Lie algebraic context for coadjoint orbits 
\cite{schmitt:StrictQuantizationOfCodjointOrbits}.
However, the algebras $\Analytic(\Mred)$ are in general not $^*$-isomorphic 
(for real $\hbar$ and the $^*$-involution of pointwise complex conjugation),
which demonstrates the 
importance of considering $^*$-algebras in strict deformation quantization. This
can be shown by examining positive linear functionals on these $^*$-algebras,
which encode information about their $^*$-representations on pre-Hilbert spaces.

The article is structured as follows:
After introducing some notation in Section~\ref{sec:notation},
we discuss the smooth and complex manifolds occurring
at various stages of the construction in Section~\ref{sec:geometry}. The classical and quantum phase space reduction
allow to construct Poisson brackets and formal star products on a reduced manifold $\Mred$
out of a constant Poisson bracket and the Wick star product on $\CC^{1+n}$. This is achieved
essentially by first restricting to the level set $\Levelset$ of a momentum map 
$\momentmap \in \Smooth(\CC^{1+n})$ and then dividing out the action of the group
$\group{U}(1)$ to obtain $\Mred \cong \Levelset / \group{U}(1)$. Depending on 
the choice of signature, $\Mred$ can e.g.\ be $\CC\PP^n$ or $\DD^n$. In order to
be able to construct the spaces of analytic functions on which the non-formal star
products can be defined, we introduce complex manifolds $\CC^{1+n}\times \CC^{1+n}$,
$\LevelsetExt$ and $\MredExt$ into which $\CC^{1+n}$, $\Levelset$ and $\Mred$ can
be embedded ``anti-diagonally''. The complex structure on $\CC^{1+n}$ finally gives rise
to a complex structure on $\Mred$, which in the special cases of $\CC\PP^n$ and $\DD^n$
coincides with the usual one. This also allows to obtain $\Mred$ by restricting
first to an open subset $\CCres$ of $\CC^{1+n}$ and then dividing out an action of
the complexification $\CCx = \set{z\in \CC}{z\neq 0}$ of $\group{U}(1)$,
which simplifies some later considerations.

Section~\ref{sec:algebra} deals with the algebras $\Smooth(\ldots)$, 
$\Analytic(\ldots)$ and $\Polynomials(\ldots)$ of smooth, certain analytic, and polynomial
functions on $\CC^{1+n}$, $\Levelset$ and $\Mred$. It is also discussed under which 
conditions and how additional structures given by bidifferential operators
on $\CC^{1+n}$ can be reduced to $\Mred$. This is then applied in Section~\ref{sec:physics}
to the Poisson bracket and Wick star product on $\CC^{1+n}$. We obtain the usual
Fubini--Study structures as well as explicit formulae for the reduced star products both
by means of bidifferential operators and by structure constants.

As the constructions for $\CC\PP^n$, $\DD^n$ and the other examples only differ
by the choice of certain signs, it is not surprising that they yield closely
related results: In Section~\ref{sec:wick} we construct isomorphisms between
various function spaces on the reduced manifolds,
which are compatible with both the Poisson brackets and the convergent star products, i.e.\ with
the classical and quantum structures. 

Finally, in Appendix~\ref{sec:derivatives} we discuss some details
concerning the symmetrized covariant derivatives used for the explicit description of
bidifferential operators in Section~\ref{sec:physics}.




\subsubsection*{Acknowledgements}
The authors would like to thank Simone Gutt, Ryszard Nest and Stefan Waldmann for some helpful discussions.




\section{Notation and Conventions} \label{sec:notation}

There are some conventions that will be used throughout the whole article: We fix two
natural numbers $n \in \NN$, $s\in\{1,\dots,1+n\}$. These will be the complex dimension $n$ of
the reduced manifold $\Mred$ and the choice of signature $s$. Nearly all objects will depend
on this signature, but in order to keep the notation clean this dependence will usually not be made
explicit. Only when it is necessary (especially when discussing the Wick rotation in Section~\ref{sec:wick})
the choice of $s$ will be indicated by a superscript in brackets.

For a smooth manifold $M$, we denote by $\Smooth(M)$ the unital $^*$-algebra of complex-valued
smooth functions on $M$ with the pointwise operations.
Similarly, $\Tangent M$ and $\CoTangent M$ are the complexified tangent and cotangent spaces.
If $M$ is even a complex manifold with complex structure $I$, then $\Tangent^{(1,0)} M$ and
$\Tangent^{(0,1)} M$ denote the linear subbundles of $+\I$ and $-\I$ eigenvectors of $I$, respectively,
and $\Tangent^{*,(1,0)} M$, $\Tangent^{*,(0,1)} M$ their duals.
The space of smooth sections of a complex vector bundle $E \to M$ over a smooth manifold $M$
is denoted by $\SmoothSections(E)$ and is a $\Smooth(M)$-module. Tensor products between such 
spaces of sections are always tensor products over the ring $\Smooth(M)$.
If $M$ is endowed with the action of a group $G$, then $\Smooth(M)^G \subseteq \Smooth(M)$
denotes the $G$-invariant smooth functions on $M$. This notation is also applied to subspaces
of $\Smooth(M)$. A multilinear map $\Phi \colon \Smooth(M) \times \dots \times \Smooth(M) \to \Smooth(M)$
is called $G$-invariant if $\Phi(f_\alpha\racts g, \dots, f_\omega\racts g) \racts g^{-1} = \Phi(f_\alpha,\dots,f_\omega)$
holds for all $f_\alpha, \dots,f_\omega \in \Smooth(M)$ and all $g\in G$.

The tensor algebra over a vector space $V$ is denoted by $\Ten[\bullet] V \coloneqq \bigoplus_{k=0}^\infty \Ten[k] V$ with
$\Ten[k] V$ the linear subspace of homogeneous tensors of degree $k \in \NN_0$.
The symmetric and antisymmetric
tensor algebra are identified with the linear subspaces $\Symten{\bullet}V$ and 
$\ASymten{\bullet} V$ of $\Ten[\bullet]V$ consisting of
symmetric and antisymmetric tensors, respectively, with symmetric and
antisymmetric tensor product $X\vee Y = \SymOp[\bullet](X\otimes Y)$ for all $X,Y\in \Symten{\bullet}V$
and 
$X\wedge Y = \AsymOp[\bullet](X\otimes Y)$ for all $X,Y\in \ASymten{\bullet}V$.
Here 
$\SymOp[\bullet], \AsymOp[\bullet] \colon \Ten[\bullet] V \to \Ten[\bullet] V$,
the operators of symmetrization and antisymmetrization, are defined as the
homogeneous projections onto $\Symten{\bullet}V$ and $\ASymten{\bullet}V$ fulfilling
\begin{align}
      \SymOp[k]\big( v_1 \otimes \dots \otimes v_k \big) 
      &= \frac{1}{k!} \sum_\sigma v_{\sigma(1)} \otimes \dots \otimes v_{\sigma(k)}
      \shortintertext{and}
      \AsymOp[k]\big( v_1 \otimes \dots \otimes v_k \big) 
      &= \frac{1}{k!} \sum_\sigma \sgn(\sigma) v_{\sigma(1)} \otimes \dots \otimes v_{\sigma(k)}
\end{align}
for $k\in \NN_0$ and $v_1,\dots,v_k \in V$, where the sum is over all 
permutations $\sigma$ of $\{1,\dots,k\}$.
So especially $v\vee w = \frac{1}{2}(v\otimes w + w\otimes v)$ and
$v\wedge w = \frac{1}{2}(v\otimes w - w\otimes v)$ for all $v,w\in V$.
Vector bundles and their sections are treated analogously.
    
By $\dupr{\argument}{\argument} \colon V^* \times V \to \CC$ we denote
the dual pairing between a complex vector space $V$ and its algebraic dual $V^*$,
$\dupr{\omega}{\alpha} \coloneqq \omega(\alpha)$ for all $\omega \in V^*$, $\alpha \in V$.
This pairing is extended to higher tensor powers by demanding that
\begin{align}
   \dupr[\big]{\omega_1\otimes \dots \otimes \omega_k}{\alpha_1\otimes \dots \otimes \alpha_k}
   =
   \dupr{\omega_1}{\alpha_1} \dots \dupr{\omega_k}{\alpha_k}
\end{align}
for all $k\in \NN_0$ and $\omega_1,\dots,\omega_k \in V^*$, $\alpha_1,\dots,\alpha_k \in V$.
Especially for symmetric tensor products this yields
\begin{align}
  \dupr[\big]{\omega_1\vee \dots \vee \omega_k}{\alpha_1\vee \dots \vee \alpha_k}
  =
  \frac{1}{k!}\sum_{\sigma} \dupr{\omega_1}{\alpha_{\sigma(1)}} \dots \dupr{\omega_k}{\alpha_{\sigma(k)}}
\end{align}
where again the sum is over all permutations $\sigma$ of $\{1,\dots,k\}$. If $\iota_\beta$
denotes the insertion derivation with a vector $\beta\in V$, i.e.\ the derivation of degree $-1$
of the symmetric tensor algebra over $V^*$ that fulfils $\iota_\beta \omega = \dupr{\omega}{\beta}$
for all $\omega \in V^*$, then by the above conventions,
\begin{align}
    \frac{1}{k} \dupr[\big]{\iota_\beta (\omega_1 \vee \dots \vee \omega_{k})}{\alpha_1 \vee \dots \vee \alpha_{k-1}}
  =
  \dupr[\big]{\omega_1 \vee \dots \vee \omega_{k}}{\beta \vee \alpha_1 \vee \dots \vee \alpha_{k-1}}
  \label{eq:insertionpairing}
\end{align}
holds for all $k\in \NN$, $\omega_1,\dots,\omega_k \in V^*$ and $\alpha_1,\dots,\alpha_{k-1} \in V$.
Like before, vector bundles and their sections are treated analogously.

We will also make use of multiindices $P,Q \in \NN_0^{1+n}$ and define as usual 
$P! \coloneqq \prod_{k=0}^n P_k!$ and
\begin{equation}
  \binom{P}{Q} \coloneqq \frac{P!}{(P-Q)!Q!}
\end{equation}
for $Q\le P$ (the order is the elementwise one). Moreover, the elementwise minimum is
\begin{equation}
  \min\{P,Q\} \coloneqq \big( \min\{P_0,Q_0\},\dots,\min\{P_n,Q_n\} \big)
  \punkt
\end{equation}




\section{Geometric Background}    \label{sec:geometry}

In this section we will in detail explain the following commutative diagram, that describes the reduction 
procedures to obtain $\Mred$ and $\MredExt$:
\begin{equation} \label{diagram:reduction}
	\begin{tikzcd}[column sep=huge, row sep=large, nodes={inner xsep=12pt, inner ysep=7pt}]
	\CC^{1+n}\times\CC^{1+n}
	&
	\LevelsetExt
	\arrow[swap]{l}{ {\hat{\iota}} }
	\arrow{r}{ {\hat{\pr}} }
	&
	\MredExt
	\\ 
	\CC^{1+n}
	\arrow[swap]{u}{ \Diag }
	&
	\arrow[swap]{l}{ \iota }
	\Levelset
	\arrow[swap]{u}{ \DiagZ }
	\arrow{r}{ \pr }
	\arrow{d}{  }
	&
	\Mred
	\arrow[swap]{u}{ \DiagM }
	\\
	&
	\CCres
	\arrow{lu}{}
	\arrow[swap]{ru}{\bigpr}
	&
	\end{tikzcd}
\end{equation}
Note the similarity to the diagram considered in
\cite{kraus.roth.schoetz.waldmann:OnlineConvergentStarProductOnPoincareDisc}.




\subsection*{Middle Row}

The middle row is a typical example of Marsden--Weinstein reduction, even though we will not 
yet discuss symplectic structures in this section.  
It consists of (at least) smooth manifolds endowed with an action of the real Lie group 
$\Stab$, which is defined below, and of $\Stab$-equivariant smooth maps.
 
On $\CC^{1+n}$, let $z^0,\dots,z^n$ be the standard coordinates,
i.e.\ $z^k(\rho) = \rho^k$ for all $k\in \{0,\dots,n\}$ and $\rho\in \CC^{1+n}$.
We define
\begin{equation}
\momentmap
\coloneqq \sum_{k=0}^{n} \nu_{k} z^k \cc{z}^k
= \sum_{k=0}^{s-1} z^k \cc z^k - \sum_{k=s}^{n} z^k \cc z^k
\komma
\end{equation}
where the coefficients $\nu_{k}$ are
$+1$ if $k \in \{0,\dots,s-1\}$ and 
$-1$ if $k \in \{s,\dots,n\}$. 
Note that we drop the dependence of $\momentmap$ and $\nu_k$ on $s$ from our notation as explained in
the convention at the beginning of Section~\ref{sec:notation}.
The Lie group $\GL$ of invertible complex $(1+n)\times(1+n)$\,-\,matrices acts 
from the left on $\CC^{1+n}$ as usual via
$A\acts \rho \coloneqq A\rho$ for all $A\in\GL$ and $\rho \in \CC^{1+n}$.
This left action $\argument\acts\argument$ on $\CC^{1+n}$ induces a right action 
$\argument\racts\argument$ on smooth functions and tensor fields by pullback.
Especially for the coordinate functions, this yields
$z^k \racts A = \sum_{\ell=0}^{n} {A^k}_\ell\, z^\ell$.

The stabilizer of $\momentmap$, i.e.\ the set of all $A\in\GL$ fulfilling
$\momentmap \racts A = \momentmap$, is
\begin{equation}
\Stab 
\coloneqq
\group{U}(s,1+n-s)
=
\set[\Big]{
	A\in\GL
}{
	\sum\nolimits_{k=0}^n \nu_{k} {A^k}_\ell \cc{{A^k}_m} = \delta_{\ell,m} \,\nu_{m} 
	\textup{ for all }\ell,m\in \{0,\dots,n\}
}
\end{equation}
with $\delta_{\ell,m}$ the usual Kronecker-$\delta$.
Note that $\Stab$ is a real Lie group and a subgroup of $\GL$.
Its Lie algebra is
\begin{equation}
\stab 
\coloneqq
\lie{u}(s,1+n-s)
=
\set[\Big]{
 	A\in\gl
}{
 	\nu_{\ell} \cc{{A^\ell}_m}\, + \nu_{m} {A^m}_\ell = 0
 	\textup{ for all }\ell,m\in \{0,\dots,n\}
}
\komma
\end{equation}
which is a real form of $\gl = \CC^{(1+n)\times(1+n)}$.
 
We define 
$\Levelset
\coloneqq {\momentmap}^{-1}(\{1\}) 
= \set[\big]{\rho\in \CC^{1+n}}{1+\sum_{k=s}^{n} \abs{\rho^k}^2 = \sum_{k=0}^{s-1} \abs{\rho^k}^2}$,
the $1$-level set of $\momentmap$, and $\iota \colon \Levelset \to \CC^{1+n}$ as
the canonical inclusion. Then the $\Stab$-action on $\CC^{1+n}$
restricts to $\Levelset$ and $\iota$ is $\Stab$-equivariant.
 
The second step is to divide out the orbits of the action of the $\group{U}(1)$-subgroup
$\set{\E^{\I\phi}\Unit_{1+n}}{\phi\in\RR}$ of $\Stab$, with $\Unit_{1+n}$ the identity matrix, which yields
\begin{equation}
\Mred \coloneqq \Levelset \big/\,\group{U}(1) \punkt
\end{equation}
As the 
$\group{U}(1)$-subgroup of $\Stab$ is central,
the $\Stab$-action remains well-defined on $\Mred$ and the
canonical projection $\pr \colon \Levelset \to \Mred$ is $\Stab$-equivariant.

We note that, by mapping the $\group{U}(1)$-equivalence class $[\rho] \in \Mred$
of some $\rho\in \Levelset$ to its $\CCx$-equivalence class $[\rho] \in \CC\PP^n$, 
the manifold $\Mred$ can be identified with the well-defined open complex submanifold
$\set{[\rho] \in \CC\PP^n}{ \momentmap(\rho) > 0}$ of $\CC\PP^n$.
Then $w^1,\dots,w^n \colon \set{[\rho] \in \Mred}{\rho^0 \neq 0} \to \CC$,
\begin{equation} \label{eq:coordinates:Mred}
w^k\big([\rho] \big) \coloneqq \frac{\rho^k}{\rho^0}
\end{equation}
with $k\in \{1,\dots,n\}$ define the usual (complex) projective coordinates on 
$\set{[\rho] \in \Mred}{\rho^0 \neq 0} \subseteq \Mred$ and it is easy to 
obtain an atlas by considering similar coordinates on 
$\set{[\rho] \in \Mred}{\rho^\ell \neq 0}$ for $1 \leq \ell \leq n$.
We will later see how the complex structure that $\Mred$ inherits
from $\CC\PP^n$ can also be obtained in a more natural way.

In the special case of the signature $s=1+n$, this construction yields $\Mred[{\signature[1+n]}] \cong \CC\PP^n$
with the usual action of $\group{U}(1+n)$ on it. For $s=1$, one obtains the disc
$\Mred[{\signature[1]}] \cong \DD^n = \set{\xi \in \CC^n}{ \sum_{k=1}^n \abs{\xi^k}^2 < 1}$
with the action of $\group{U}(1,n)$ by Möbius transformations. The holomorphic
isomorphism from $\Mred[{\signature[1]}]$ to the disc is simply given by the coordinates $w^1,\dots,w^n$,
which are global coordinates if $s = 1$.

Note that in general, these projective coordinates $w^1 , \dots, w^n$ describe a chart for $\Mred$
with dense domain of definition. Because of this, it is essentially
sufficient to use only these coordinates for the explicit description of some
tensors later on, but it is important to keep in mind that they describe $\Mred$
only up to a meagre subset.




\subsection*{Top Row}

The top row consists of complex manifolds carrying an action of a complex Lie 
group $\StabExt$,
and of $\StabExt$-equivariant holomorphic maps. These complex manifolds will later
be helpful for defining certain algebras of analytic functions on $\CC^{1+n}$ and $\Mred$.

On $\CC^{1+n}\times \CC^{1+n}$,
the standard complex coordinate functions are denoted by $x^0,\dots,x^n,y^0,\dots,y^n$,
and given by
$x^k(\xi,\eta) \coloneqq \xi^k$ as well as $y^k(\xi,\eta) \coloneqq \eta^k$ for all $k\in 
\{0,\dots,n\}$ and $\xi,\eta \in \CC^{1+n}$.
Define the holomorphic polynomial
\begin{equation}
\momentmapExt \coloneqq \sum_{k=0}^{n} \nu_{k} x^k y^k = \sum_{k=0}^{s-1} x^k y^k - 
\sum_{k=s}^{n} x^k y^k \punkt
\end{equation}
Note that the polynomial $\momentmap$ considered before is just the restriction 
of $\momentmapExt$ to the antidiagonal. More precisely, if $\Diag \colon \CC^{1+n} \to \CC^{1+n} \times \CC^{1+n}$,
\begin{equation}
 \rho \mapsto \Diag(\rho) \coloneqq (\rho, \cc \rho)
\end{equation}
denotes the embedding along the antidiagonal,
then $\momentmap = \momentmapExt \circ \Diag = \Diag^* (\momentmapExt)$. Similarly, $\Diag^*(x^k) = z^k$ and
$\Diag^*(y^k) = \cc{z}^k$ for all $k\in \{0,\dots,n\}$.

The complex Lie group $\GL\times \GL$
acts holomorphically from the left on $\CC^{1+n}\times\CC^{1+n}$ as usual via 
$(A,B) \acts (\xi,\eta) \coloneqq (A\xi,B\eta)$ for all $A,B\in \GL$ and $\xi,\eta\in \CC^{1+n}$,
which induces a right action $\argument \racts \argument$ by pullback on the spaces 
of holomorphic functions or holomorphic tensor fields. 
Especially for the coordinate functions, 
this yields
$x^k \racts (A,B) = \sum_{\ell=0}^{n} {A^k}_\ell\, x^\ell$ and 
$y^k \racts (A,B) = \sum_{\ell=0}^{n} {B^k}_\ell\, y^\ell$.

The stabilizer $\StabExt$ of $\momentmapExt$, i.e.\ the set of $(A,B) \in \GL\times \GL$
fulfilling $\momentmapExt \racts (A,B) = \momentmapExt$, is explicitly given by
\begin{equation}
\StabExt
=
\set[\bigg]{ (A,B)\in \GL\times\GL }{
	\sum_{k=0}^{n} \nu_{k} {A^k}_\ell {B^k}_m = \delta_{\ell,m}\, \nu_{m}
	\textup{ for all }\ell,m\in \{0,\dots,n\}
} \punkt
\end{equation}
Note that for all $A\in \GL$
there exists a unique $B\in \GL$ such that $(A,B)\in \StabExt$, namely
${B^k}_m = \nu_{k} \nu_{m} {(A^{-1})^m}_k$, so $\StabExt$ is a 
complex Lie group and isomorphic to $\GL$.

Similar to the definition of $\Levelset$ we define $\LevelsetExt$ as the $1$-level set of $\momentmapExt$
in $\CC^{1+n}\times \CC^{1+n}$, i.e.
\begin{equation}
\LevelsetExt 
\coloneqq 
\momentmapExt{}^{-1}(\{1\}) 
= 
\set[\bigg]{(\xi,\eta)\in \CC^{1+n}\times\CC^{1+n}}{1+\sum_{k=s}^{n} \xi^k\,\eta^k = 
\sum_{k=0}^{s-1} \xi^k\,\eta^k}\,\punkt
\end{equation}
Then $\LevelsetExt$ is a complex submanifold of $\CC^{1+n}\times\CC^{1+n}$.
The canonical inclusion of $\LevelsetExt$ into $\CC^{1+n}\times \CC^{1+n}$ is denoted by ${\hat{\iota}}$.
As $\momentmapExt$ is invariant under the action of $\StabExt$, this action can be restricted to
$\LevelsetExt$ and $\hat{\iota}$ then is clearly $\StabExt$-invariant.
Moreover the inclusion $\Diag$ restricts to an inclusion $\DiagZ \colon \Levelset \to \LevelsetExt$, which makes 
the upper left square in \eqref{diagram:reduction} commute.

The second step is to divide out the orbits of the Lie group $\CCx \coloneqq \CC\setminus\{0\}$,
more precisely of the subgroup $\set{(\alpha \Unit_{1+n}, \alpha^{-1}\Unit_{1+n})}{\alpha\in\CCx}$
of $\StabExt$. So define
\begin{equation}
\MredExt 
\coloneqq 
\LevelsetExt \big/ \,\CCx \komma
\end{equation}
then $\MredExt$ can be identified with
$\set{([\xi],[\eta])\in \CC\PP^n\times\CC\PP^n}{ \momentmapExt(\xi,\eta) \neq 0}$,
a well-defined open and dense complex submanifold of $\CC\PP^n\times\CC\PP^n$,
via $\MredExt \ni [(\xi,\eta)] \mapsto ([\xi],[\eta]) \in \CC\PP^n\times\CC\PP^n$.
As the $\CCx$-subgroup of $\StabExt$ is central,
the $\StabExt$-action remains well-defined on $\MredExt$.
The canonical projection from $\LevelsetExt$ onto the quotient $\MredExt$
will be denoted by $\hat{\pr}$ and is again $\StabExt$-equivariant by construction.
Finally, one can check that $\DiagM \colon \Mred \to \MredExt$,
\begin{equation}
[\rho] \mapsto \DiagM \big([\rho]\big) \coloneqq \big[ \DiagZ(\rho) \big] = [(\rho,\cc{\rho})]
\end{equation}
is well-defined and makes the upper right rectangle of \eqref{diagram:reduction} commute.

On $\MredExt$, we are going to use the usual projective
coordinates coming from $\CC\PP^n\times \CC\PP^n$, denoted by
$u^1,\dots,u^n \colon \set{[(\xi,\eta)] \in \MredExt }{\xi^0 \neq 0} \to \CC$
and
$v^1,\dots,v^n \colon \set{[(\xi,\eta)] \in \MredExt }{\eta^0 \neq 0} \to \CC$,
and given by
\begin{equation} \label{eq:coordinates:extended}
u^k\big([(\xi,\eta)] \big) \coloneqq \frac{\xi^k}{\xi^0}
\quad\quad\text{as well as}\quad\quad
v^k\big([(\xi,\eta)] \big) \coloneqq \frac{\eta^k}{\eta^0}
\end{equation}
for all $k\in \{1,\dots,n\}$.
Note that it is again easy to obtain an atlas by considering similarly defined coordinates on
$\set{[(\xi,\eta)] \in \MredExt }{\xi^i \neq 0}$ and
$\set{[(\xi,\eta)] \in \MredExt }{\eta^j \neq 0}$
and that the relations $(\DiagM)^* ( u^k) = w^k$ and $(\DiagM)^* ( v^k) = \cc{w}^k$ 
hold for all $k\in\{1,\dots,n\}$. As before, one should also keep in mind that
these coordinates form a chart with dense domain of definition.




\subsection*{Bottom Node}

It turns out that the complex structure on $\CC^{1+n}$ can be used to simplify the 
Marsden--Weinstein reduction in the middle row of \eqref{diagram:reduction}.
First, we define a complex structure on $\Mred$ that is 
compatible with the complex coordinates defined before. A
more general treatment of this procedure can be found in 
\cite{stienon.xu:reductionOfGeneralizedComplexStructures}.
Then we find a holomorphic projection map $\bigpr : \CCres \to \Mred$ from the open subset 
\begin{align}
\CCres \coloneqq \set[\big]{ z \in \CC^{1+n} }{ \momentmap(z) > 0 }
\end{align}
of $\CC^{1+n}$ to $\Mred$ making the bottom right triangle in \eqref{diagram:reduction} commute.
Since restriction to an open subset is easy for almost any geometric structure,
one can therefore avoid the restriction to a hypersurface 
that is needed in the Marsden--Weinstein reduction.

Denote the standard complex structure of $\CC^{1+n}$ by $I$. For $A \in \gl$, let $X_A$ be the 
vector field on $\CC^{1+n}$ obtained by differentiating the right action of $\GL$ on 
$\Smooth(\CC^{1+n})$ in the direction of $A$, i.e.\ $X_A(f) = \frac{\D}{\D t}\at[\big]{t=0} f\racts \exp(tA)$.
In particular,
\begin{align}
X_\I
\coloneqq
X_{\I\Unit_{1+n}}
=
\sum_{k=0}^n \bigg(
	\I z^k \frac{\partial}{\partial z^k} 
	- \I \cc{z}^k \frac{\partial}{\partial \cc{z}^k}
\bigg)
\end{align}
is the generator of the (diagonal) $\group{U}(1)$-action and
\begin{equation}
  X_\Unit 
  \coloneqq
  X_{\Unit_{1+n}}
  = 
  \sum_{k=0}^n \bigg(
    z^k \frac{\partial}{\partial z^k} + \cc{z}^k \frac{\partial}{\partial \cc{z}^k}
  \bigg)
  = -I X_\I
  \punkt
\end{equation}
Let $\genHull{X_\I}$ and $\genHull{X_\Unit}$ be the complex $1$-dimensional vector subbundles of $\Tangent \CCres$
spanned by $X_\I\at{\CCres}$ and $X_\Unit\at{\CCres}$, respectively.
Moreover, for $\rho \in \CCres$ define 
\begin{equation} \label{eq:def:Xi}
  \Xi_\rho 
  \coloneqq 
  \set[\big]{
    \alpha_\rho \in \Tangent_\rho \CCres
  }{
    \alpha_\rho(\momentmap) = 0
    \text{ and }
    \big(I \alpha_\rho\big)(\momentmap) = 0
  }
  \quad\text{and}\quad
  \Xi \coloneqq \bigcup_{\rho \in \CCres} \Xi_\rho
  \komma
\end{equation}
then one can check that $\Xi$ is a complex $2n$-dimensional vector subbundle of $\Tangent \CCres$,
and we get:
\begin{proposition} \label{proposition:decomposition:CC}
  The tangent bundle of $\CCres$ can be decomposed as the direct sum
  \begin{equation} \label{eq:decomposition:CC}
    \Tangent \CCres = \genHull{X_\Unit} \oplus \genHull{X_\I} \oplus \Xi 
    \punkt
  \end{equation}
  Moreover, for all $\rho \in \Levelset$, the map
  $\Tangent_\rho \pr \circ \operatorname{(\Tangent_\rho \iota)}^{-1} \colon \Xi_\rho \to \Tangent_{[\rho]} \Mred$
  is a linear isomorphism.
\end{proposition}
\begin{proof}
  For all $\rho \in \CCres$, the linear subspace
  $S_\rho \coloneqq \set[\big]{
    \alpha_\rho \in \Tangent_\rho \CCres
  }{
    \alpha_\rho(\momentmap) = 0
  }$ of $\Tangent_\rho \CCres$ has complex codimension $1$, and $\genHull{X_\Unit}\at{\rho}$ is a complement
  of $S_\rho$ in $\Tangent_\rho \CCres$ because $X_\Unit(\momentmap) = 2 \momentmap$. So 
  $\Tangent_\rho \CCres = \genHull{X_\Unit}\at{\rho} \oplus S_\rho$. Moreover, $\group{U}(1)$-invariance of $\momentmap$
  implies that $X_\I(\momentmap) = 0$, so $\genHull{X_\I}\at{\rho} \subseteq S_\rho$,
  and $\Xi_\rho \subseteq S_\rho$ is clear. But as $(I X_\I)(\momentmap) = - X_\Unit(\momentmap) = -2 \momentmap$,
  the sum of $\genHull{X_\I}\at{\rho}$ and $\Xi_\rho$ is direct, and therefore $S_\rho = \genHull{X_\I}\at{\rho} \oplus \Xi_\rho$
  by counting dimensions.
  This proves the decomposition \eqref{eq:decomposition:CC}.
  
  If $\rho \in \Levelset$, then $S_\rho = \genHull{X_\I}\at{\rho}\oplus \Xi_\rho$
  coincides with the image of $\Tangent_\rho \Levelset$ under $\Tangent_\rho \iota$.
  Because of this, the
  map $\Tangent_\rho \pr \circ (\Tangent_\rho \iota)^{-1}$
  is well-defined as a map from $\genHull{X_\I}\at{\rho}\oplus \Xi_\rho$ to $\Tangent_{[\rho]} \Mred$
  and is clearly surjective. The kernel of this map is $\genHull{X_\I}\at{\rho}$, so its restriction
  to $\Xi_\rho$ is an isomorphism.
\end{proof}
Note that $X_\I$, $X_\Unit$ and $\Xi$ do not depend on any choices but arise naturally from the $\group{U}(1)$-action,
the map $\momentmap$ and the complex structure $I$ that $\CCres$ inherits from $\CC^{1+n}$. By
definition of $\Xi$, this complex structure restricts to $\Xi$. As it is also $\group{U}(1)$-invariant,
it gives rise to a well-defined (almost) complex structure $I_\red$ on $\Mred$:

\begin{definition}
  Define the vector bundle endomorphism $I_\red \colon \Tangent \Mred \to \Tangent \Mred$,
  that maps any $\beta_{[\rho]} \in \Tangent_{[\rho]} \Mred$ with $[\rho] \in \Mred$ to 
  $I_\red(\beta_{[\rho]}) \coloneqq \big(\Tangent_\rho \pr \circ \operatorname{(\Tangent_\rho \iota)}^{-1} \circ I\at{\rho} \circ 
	(\Tangent_\rho \pr \circ \operatorname{(\Tangent_\rho \iota)}^{-1})^{-1}\big)(\beta_{[\rho]})$.
\end{definition}
It is clear that $I_\red$ squares to $-\id_{\Tangent \Mred}$ and hence is an almost complex
structure. In order to see that it is also integrable, we check that $I_\red$
coincides with the complex structure that $\Mred$ inherits from $\CC\PP^n$. For a more general
discussion, see \cite{stienon.xu:reductionOfGeneralizedComplexStructures}:

\begin{definition}
  On $\CCres \setminus \set{\rho \in \CCres}{ z^0(\rho) = 0}$ we define the vector fields
  \begin{align}
  W_k \coloneqq z^0\bigg( \frac{\partial}{\partial z^k} - \frac{\nu_k\cc{z}^k}{\momentmap} \sum_{\ell=0}^n 
  z^\ell
  \frac{\partial}{\partial z^\ell}\bigg) \at[\bigg]{\CCres \setminus \set{\rho \in \CCres}{ z^0(\rho) = 0}}
  \end{align}
  for all $k\in \{1,\dots,n\}$.
\end{definition}
Note that, analogously to the projective coordinates $w^1,\dots,w^n$ on $\Mred$, 
the vector fields $W_1,\dots,W_n$ are only defined on a dense subset of $\CCres$.
However, this will be completely sufficient for our purposes.

As $I W_k = \I W_k$ and $\dupr{\D \momentmap}{W_k} = 0$ for all $k\in \{1,\dots,n\}$
on the domain of definition of $W_k$,
these vector fields $W_k$, as well as their complex conjugates $\cc{W}_k$ with $k\in \{1,\dots,n\}$,
are actually (local, densely defined) sections of $\Xi$. A short calculation shows that
\begin{equation}
  W_k\big(z^\ell / z^0\big) = \delta^{\ell}_k
\end{equation}
for all $k,\ell \in \{1,\dots,n\}$, so the sections $W_k$ are pointwise linearly independent
and, by counting dimensions, they form a (local, densely defined) frame of $\Xi$.
Moreover, this immediately shows: 

\begin{proposition} \label{proposition:wredused}
  If $\rho \in \Levelset$, $z^0(\rho) \neq 0$, then
  \begin{align}
  \big(\Tangent_\rho \pr \circ \operatorname{(\Tangent_\rho \iota)}^{-1}\big)(W_k\at{\rho}) 
  =
  \frac{\partial}{\partial w^k} \at[\bigg]{[\rho]}
  \quad\quad\text{and}\quad\quad
  \big(\Tangent_\rho \pr \circ \operatorname{(\Tangent_\rho \iota)}^{-1}\big)(\cc{W}_k\at{\rho})
  =  
  \frac{\partial}{\partial \cc{w}^k} \at[\bigg]{[\rho]}
  \end{align}
  for all $k\in \{1,\dots,n\}$.
\end{proposition}
%
As an immediate consequence we obtain:
\begin{corollary}
  The reduced complex structure fulfils
  $I_\red(\frac{\partial}{\partial w^k}) = \I \frac{\partial}{\partial w^k}$ and
  $I_\red(\frac{\partial}{\partial \cc{w}^k}) = -\I \frac{\partial}{\partial \cc{w}^k}$
  for all $k\in \{1,\dots,n\}$,
  so $I_\red$ is indeed the standard complex structure of $\Mred$ interpreted as an open subset 
  of $\CC\PP^n$. In particular, $I_\red$ is integrable and really a complex structure.
\end{corollary}
It thus makes sense to talk about holomorphic maps from $\CC^{1+n}$ or $\CCres$ to $\Mred$:

\begin{lemma}
  If a holomorphic complex-valued map $\phi$ on a connected and open 
  subset $S \subseteq \CC^{1+n}$ with $S \cap \Levelset \neq \emptyset$
  vanishes on $S\cap \Levelset$, then it already vanishes on all of $S$.
\end{lemma}
\begin{proof}
  Indeed, as
  $\Tangent_\rho \CC^{1+n} 
  =
  \genHull{ X_\Unit} \at{\rho} \oplus (\Tangent_\rho \iota)(\Tangent_\rho \Levelset)$
  for all $\rho \in \Levelset$,
  as $\alpha_\rho(\phi) = 0$ for all $\alpha_\rho \in (\Tangent_\rho \iota)(\Tangent_\rho \Levelset)$
  by assumption and as also
  $X_\Unit\at{\rho}(\phi)
  = 
  X_\I\at{\rho}(-\I \phi) 
  =
  0$
  because $\phi$ is holomorphic and $X_\Unit\at{\rho} \in (\Tangent_\rho \iota)(\Tangent_\rho \Levelset)$,
  all first order partial derivatives of $\phi$ 
  vanish on $S\cap \Levelset$. This now extends to all arbitrarily high partial derivatives
  by using the same argument and thus the holomorphic $\phi$ vanishes on whole $S$.
\end{proof}
As a consequence, there is at most one holomorphic map $\bigpr \colon \CCres \to \Mred$
whose restriction to $\Levelset$ coincides with $\pr$. In the special case treated here it
is not hard to guess this map:

\begin{proposition}
  There exists a (necessarily unique) holomorphic map
  $\bigpr \colon \CCres \to \Mred$
  whose restriction to $\Levelset$ coincides with $\pr$.
  It is explicitly given by
  \begin{equation} \label{eq:defPr}
    \rho \mapsto \bigpr(\rho) = [\rho / \sqrt{\momentmap(\rho)}]\,.
  \end{equation}
  In coordinates, $w^k \circ \bigpr = z^k / z^0$.
\end{proposition}
\begin{proof}
It is not hard to check the expression of \eqref{eq:defPr} in coordinates, which also
shows that $\bigpr$ is holomorphic. Its restriction to $\Levelset$ clearly coincides with $\pr$.
\end{proof}
We also note that the domain $\CCres$ of $\bigpr$, which was chosen rather arbitrarily, 
is naturally determined from the $\group{U}(1)$-action on $\CC^{1+n}$ and the complex structure $I$:
The action of the corresponding Lie algebra $\lie{u}(1)\cong \RR$ is given
by its fundamental vector field $X_\I$, and the complex structure $I$ allows to extend this
to an action of the complexified Lie algebra $\lie{u}(1) \otimes \CC \cong \CC$ via the
fundamental vector fields $X_\I$ and $X_\Unit$. This action even integrates to a unique holomorphic
action of the corresponding complex Lie group $\CCx$ on $\CC^{1+n}$, which is just given
by multiplication with scalars. The orbit of $\Levelset$ under the action of $\CCx$ is easily
seen to be $\CCres$, and $\bigpr\colon \CCres \to \Mred$ is the quotient map that identifies
$\CCres / \CCx$ with $\Mred$ as complex manifolds. From this point of view, the complex
structure on $\CC^{1+n}$ allows to replace the two steps of Marsden--Weinstein reduction
(restriction to the level set $\Levelset$ and taking $\group{U}(1)$-equivalence classes)
by restriction to the open complex submanifold $\CCres$ and taking equivalence classes with respect
to the action of the complexification $\CCx$ of $\group{U}(1)$.

For future use it will be helpful to be able to express the standard coordinate vectors
$\frac{\partial}{\partial z^k}$ with $k\in \{0,\dots,n\}$ in terms of the holomorphic Euler vector
field
\begin{equation}
  E \coloneqq \frac{1}{2} (X_\Unit - \I X_\I) \at[\big]{\CCres} = \sum_{k = 0}^n z^k \frac{\partial}{\partial 
  z^k}\at[\bigg]{\CCres} 
\end{equation}
and the $W_k$, $k \in \{1,\dots, n\}$. On their domain of definition, one gets
(using that always $\nu_0 = 1$)
\begin{equation}\label{eq:zAsLinearCombOfW}
  \frac{\partial}{\partial z^0}  = \frac{\cc{z}^0}{\momentmap} E - \sum_{\ell=1}^n 
  \frac{z^\ell}{(z^0)^2} W_\ell
  \quad\quad\text{and}\quad\quad
  \frac{\partial}{\partial z^k} =  \frac{\nu_k \cc{z}^k}{\momentmap} E + \frac{1}{z^0} W_k
\end{equation}
for all $k\in \{1,\dots, n\}$ and $(E, W_1, \dots, W_k)$ is a local, densely defined frame
for $\Tangent^{(1,0)} \CCres$. Together with its
complex conjugates $(\cc{E},\cc{W}_1,\dots,\cc{W}_n)$ we obtain a densely defined frame for 
the whole tangent space $\Tangent \CCres$.
The dual frames are denoted by $(E^*, W_1^*, \dots, W_n^*)$ and $(\cc{E}^*, \cc{W}_1^*, \dots, \cc{W}_n^*)$,
and (again only on the domain of definition of the vector fields $W_k$) we have
\begin{align}
E^* &= \frac{1}{\momentmap}\sum_{k=0}^n \nu_k \cc z^k \D z^k
&&\text{and}& 
W_k^* &= - \frac{z^k}{(z^0)^2} \D z^0 + \frac{1}{z^0} \D z^k = \bigpr^*(\D w^k) \komma 
\label{eq:EWAsLinearCombOfDz}
\\
\D z^0 &= z^0 E^* - \frac{(z^0)^2}{\momentmap}\sum_{k=1}^n \nu_k \cc z^k W_k^*
&&\text{and}&
\D z^k &= z^k E^* + z^0 \bigg(W_k^* - \frac{z^k}{\momentmap} \sum_{\ell=1}^n v_\ell \cc{z}^\ell W_\ell^*\bigg)
\punkt
\label{eq:DzAsLinearCombOfWStar}
\end{align}
Note that $E$ and $\cc{E}$ are obtained from the $\group{U}(1)$-action and complex 
structure of $\CCres$. Similarly,
also $E^*$ and $\cc{E}^*$ can be obtained naturally as the $(1,0)$ and $(0,1)$-parts of $\D \momentmap / \momentmap$.
Only the vector fields $W_1,\dots,W_n$ as well as their conjugates and duals depend on a choice of coordinates.




\section{Algebraic Point of View}   \label{sec:algebra}

The general reduction procedure from $\CC^{1+n}$ to $\Mred$ by first restricting
to the level set $\Levelset$ and then dividing out the action of $\group{U}(1)$
has a dual version that connects various function algebras on $\CC^{1+n}$ and 
$\Mred$:
First, one divides out the ideal of functions vanishing on $\Levelset$ and then
restricts to $\group{U}(1)$-invariant equivalence classes. However, as
every $\group{U}(1)$-invariant equivalence class of functions also contains at 
least one $\group{U}(1)$-invariant function, which can be obtained by averaging over the
compact group $\group{U}(1)$, a simplified procedure yields the
same results: First, one restricts to $\group{U}(1)$-invariant functions and 
then divides out the ideal of functions vanishing on $\Levelset$. We will use
this second approach throughout.

It is well-known that this way one can also construct algebraic structures on 
$\Mred$ out of such structures on $\CC^{1+n}$, especially Poisson brackets and
star products.
In the following we will consider three types of function algebras: 
All smooth functions, polynomial functions and certain analytic functions. 
While formal star products are defined on all smooth functions, their non-formal 
versions can only be defined on polynomial or some analytic functions.
All these function algebras on $\CC^{1+n}$ will also be endowed with the right-action of the
stabilizer group $\Stab$.




\subsection{Smooth Functions}   \label{subsec:smoothfunctions}

The reduction procedure for smooth functions is well-known. In order to fix notation,
it is helpful to shortly discuss some details again:
Recall that $\Smooth(\CC^{1+n})^{\group{U}(1)}$ is the unital subalgebra of 
$\Smooth(\CC^{1+n})$ whose elements are the $\group{U}(1)$-invariant
functions. It is easy to see that the following is well-defined:
\begin{definition} \label{definition:reductionmap}
  Let $S$ be an open and $\group{U}(1)$-invariant subset of $\CC^{1+n}$ such that $S \supseteq \Levelset$.
  The \neu{(classical) reduction map} is $\argument_\red \colon \Smooth(S)^{\group{U}(1)} \to \Smooth(\Mred)$, $f \mapsto f_\red$,
  where
  \begin{equation}\label{eq:reductionmap}
    f_\red([\rho]) \coloneqq f(\rho)
  \end{equation}
for all $\rho \in \Levelset$.
\end{definition}
We will especially be interested in the two cases $S=\CC^{1+n}$ and $S=\CCres$.
Note that $f_\red$ is the unique smooth function on $\Mred$ that fulfils $\pr^*(f_\red) = \iota^*(f)$.
From the algebraic point of view, smooth functions on
$\CC^{1+n}$ and $\Mred$ can be related as follows:

\begin{lemma} \label{lemma:extension}
  For every $g\in \Smooth(\Mred)$ there exists an $f\in \Smooth(\CC^{1+n})^{\group{U}(1)}$
  such that $f_\red = g$, and $f$ can even be chosen in such a way that the following locality condition is fulfilled: 
  Whenever $U$ is an open subset of $\Mred$ such that the restriction of $g$ to $U$ vanishes, 
  then there exists an open subset $V$ of $\CC^{1+n}$ such that $V \supseteq \pr^{-1}(U)$ 
  and such that the restriction of $f$ to $V$ vanishes.
\end{lemma}
\begin{proof}
  This is well-known to be true in more generality, but in the present case it 
  is also easy to construct such an $f\in \Smooth(\CC^{1+n})^{\group{U}(1)}$
  for every $g\in \Smooth(\Mred)$: Indeed, one can define $f(\rho) \coloneqq 0$ for all $\rho \in \CC^{1+n} 
  \setminus \CCres$ and $f(\rho) \coloneqq g(\bigpr(\rho)) \chi(\momentmap(\rho))$
  for all $\rho \in \CCres$, where $\chi \colon {]0,\infty[} \to {[0,1]}$ is a smooth
  function with compact support that fulfils $\chi(1) = 1$.
\end{proof}
This lemma has the following consequence:
\begin{proposition}
    For every open and $\group{U}(1)$-invariant subset $S \subseteq \CC^{1+n}$ that contains $\Levelset$,
	the reduction map $\argument_\red \colon \Smooth(S)^{\group{U}(1)} \to \Smooth(\Mred)$
	descends to an isomorphism between the unital $^*$-algebras
	$\Smooth(S)^{\group{U}(1)} / \set{v \in \Smooth(S)^{\group{U}(1)} }{\iota^*(v) = 0}$
	and $\Smooth(\Mred)$. 
\end{proposition}
We can now also construct algebraic structures on $\Smooth(\Mred)$ out of such structures on 
$\Smooth(\CC^{1+n})$ or $\Smooth(\CCres)$:

\begin{proposition} \label{proposition:restriction}
  Let $S$ be an open and $\group{U}(1)$-invariant subset of $\CC^{1+n}$ such that 
  $S \supseteq \Levelset$, and let $C\colon \Smooth(S) \times \Smooth(S) \to \Smooth(S)$
  be a $\group{U}(1)$-invariant bilinear map, then the following is equivalent:
  \begin{itemize}
    \item There exists a bilinear map 
    $C_\red \colon \Smooth(\Mred) \times \Smooth(\Mred) \to \Smooth(\Mred)$
    such that
    \begin{align*}
     \big(C(f,g)\big)_\red = C_\red\big(f_\red,g_\red\big)
    \end{align*}
    holds for all $f,g\in \Smooth(S)^{\group{U}(1)}$.
    \item $C(f,v)\at{\Levelset} = 0 = C(v,f)\at{\Levelset}$ holds for all 
    $f,v \in  \Smooth(S)^{\group{U}(1)}$ with $\iota^*(v) = 0$.
  \end{itemize}
  If one, hence both of these two conditions are fulfilled, then the bilinear map 
  $C_\red$ from the first point is uniquely determined.
\end{proposition}
\begin{proof}
  Using the existence of preimages under $\argument_\red$ from Lemma~\ref{lemma:extension},
  the equivalence of the two points and the uniqueness of $C_\red$ are standard results.
\end{proof}

\begin{definition} \label{definition:reducible}
  Let $S$ be an open and $\group{U}(1)$-invariant subset of $\CC^{1+n}$ such that 
  $S \supseteq \Levelset$, and let $C\colon \Smooth(S) \times \Smooth(S) \to \Smooth(S)$
  be a $\group{U}(1)$-invariant bilinear map, then $C$ is called \neu{reducible}
  if one, hence both of the equivalent properties from the previous Proposition~\ref{proposition:restriction}
  are fulfilled. In this case, we also define the \neu{reduced} map $C_\red$ like
  in the first point there.
\end{definition}
One example is of course the multiplication: Let $C$
be the pointwise multiplication of smooth functions on $\CC^{1+n}$,
then $C_\red$ is the pointwise multiplication of smooth 
functions on $\Mred$. 
For more interesting examples, however,
the second point in Proposition~\ref{proposition:restriction} can still be
hard to check. Luckily, there are some simplifications for bidifferential operators.
Note that in the following it is no loss of generality to consider the special case of a
$\group{U}(1)$-invariant bidifferential operator 
$C\colon \Smooth(\CCres) \times \Smooth(\CCres) \to \Smooth(\CCres)$:
A bidifferential operator on a different domain of definition can always be restricted
and extended (in a not necessarily unique way) to a bidifferential operator
on $\CCres$ which coincides with the original one in a neighbourhood of $\Levelset$
and thus yields the same reduced map.

\begin{proposition} \label{proposition:reduction:sufficientCondition}
  Let $C\colon \Smooth(\CCres) \times \Smooth(\CCres) \to \Smooth(\CCres)$
  be a $\group{U}(1)$-invariant bidifferential operator. If
  $C\big((\momentmap-1)\at{\CCres}f,f'\big) = 0 = C\big(f,(\momentmap-1)\at{\CCres}f'\big)$
  holds for all $f,f' \in \Smooth(\CCres)^{\group{U}(1)}$, then
  $C$ is reducible and
  \begin{align}
    \big( C( \bigpr^*(g),  \bigpr^*(g') ) \big)_\red = C_\red(g,g')
    \label{eq:localreduction}
  \end{align}
  holds for all $g,g'\in\Smooth(\Mred)$.
\end{proposition}
\begin{proof}
  In order to show that $C$ is reducible, let $f,v \in \Smooth(\CCres)^{\group{U}(1)}$
  with $\iota^*(v) = 0$ be given. For every $\epsilon \in {]0,1[}$ and
  using a bump function $\chi \in \Smooth({]0,\infty[})$ with support in $[1-\epsilon,1+\epsilon]$
  fulfilling
  $\chi(r) = 1$ for all $r \in [1-\epsilon/2,1+\epsilon/2]$,
  one can express $v$ as the sum $v = v\cdot (\chi\circ \momentmap\at{\CCres}) + (\momentmap-1)\at{\CCres} \tilde{v}$ of a function
  $v\cdot (\chi\circ \momentmap\at{\CCres}) \in \Smooth(\CC^{1+n})^{\group{U}(1)}$ with 
  support in $\set{\rho \in \CCres}{-\epsilon \le \momentmap(\rho)-1 \le \epsilon}$
  and the product of $(\momentmap-1)\at{\CCres}$ with a function $\tilde{v} \in \Smooth(\CCres)^{\group{U}(1)}$.
  Then $C(f,v) = C\big(f, v\cdot (\chi\circ \momentmap\at{\CCres}) \big)$ and 
  $C(v,f) = C\big(v\cdot (\chi\circ \momentmap\at{\CCres}), f \big)$
  have support in $\set{\rho \in \CCres}{-\epsilon \le \momentmap(\rho)-1 \le \epsilon}$.
  As $\epsilon \in {]0,1[}$ was arbitrary, even $C(v,f) = 0 = C(f,v)$ holds and $C$ is reducible.
  For Equation~\eqref{eq:localreduction}
  we just note that $(\bigpr^*(g))_\red = g$ for all $g\in \Smooth(\Mred)$.
\end{proof}




\subsection{Polynomial Functions}   \label{subsec:polynomialfunctions}

On polynomial functions it will be possible to construct non-formal star products 
in Section~\ref{sec:physics}. Here we only discuss the basic definitions and the 
reduction procedure:

\begin{definition}
  We write $\Polynomials(\CC^{1+n})$ for the unital $^*$-subalgebra of 
  $\Smooth(\CC^{1+n})$ that consists of all polynomial functions in $z^0,\dots,z^n,\cc{z}^0,\dots,\cc{z}^n$.
  We denote the image of $\Polynomials(\CC^{1+n})^{\group{U}(1)}$ under $\argument_\red$ by
  $\Polynomials(\Mred)$ and call its elements \neu{polynomials on $\Mred$}.
\end{definition}
One can check that $\Polynomials(\Mred)$ is a unital $^*$-subalgebra
of $\Smooth(\Mred)$ and so the reduction map restricts
to a surjective unital $^*$-homomorphism from $\Polynomials(\CC^{1+n})^{\group{U}(1)}$
to $\Polynomials(\Mred)$. 
Its kernel are all $\group{U}(1)$-invariant polynomial functions on $\CC^{1+n}$ which vanish on $\Levelset$. 
We see that the unital $^*$-algebra $\Polynomials(\Mred)$ is
isomorphic to the quotient 
$\Polynomials(\CC^{1+n})^{\group{U}(1)} / \set{
    v \in \Polynomials(\CC^{1+n})^{\group{U}(1)} 
}{
    \iota^*(v) = 0
}$ like in the smooth case.
A basis of $\Polynomials(\CC^{1+n})^{\group{U}(1)}$ yields a generating
subset of $\Polynomials(\Mred)$, a subset of which is a basis of 
$\Polynomials(\Mred)$. We essentially follow 
\cite{beiser.waldmann:FrechetAlgebraicDeformationOfThePoincareDisc,
kraus.roth.schoetz.waldmann:OnlineConvergentStarProductOnPoincareDisc}
and just check that the definitions and results there, which were made for 
the special case $s=1$, actually work for all signatures:

\begin{definition} \label{definition:basis:polynomials}
  For every pair of multiindices $P,Q \in \NN_0^{1+n}$ we define the monomial on $\CC^{1+n}$
  \begin{align}
    \Monom{P}{Q} 
    \coloneqq 
    z^P \cc{z}^Q 
    \coloneqq 
    (z^0)^{P_0} \dots (z^n)^{P_n} (\cc{z}^0)^{Q_0} \dots (\cc{z}^n)^{Q_n}
    \punkt
  \end{align}
\end{definition}
The monomials $\Monom{P}{Q}$ with $P,Q\in \NN_0^{1+n}$ are a basis of $\Polynomials(\CC^{1+n})$,
and those monomials with $\abs{P}=\abs{Q}$ are a basis of $\Polynomials(\CC^{1+n})^{\group{U}(1)}$.
The resulting \neu{reduced monomials} $\MonomRed{P}{Q} \in \Polynomials(\Mred)$ are, in the projective
coordinates defined in \eqref{eq:coordinates:Mred} (and restricted to the dense domain of definition of these 
coordinates),
\begin{align}
  \MonomRed{P}{Q} 
  = 
  \frac{w^{P'} \cc{w}^{Q'}}{(1 + \sum_{k=1}^n \nu_k w^k \cc{w}^k)^{\abs{P}}}
  \coloneqq
  \frac{(w^1)^{P_1} \dots (w^n)^{P_n}(\cc{w}^1)^{Q_1} \dots (\cc{w}^n)^{Q_n}}{(1 + \sum_{k=1}^n \nu_k w^k \cc{w}^k)^{\abs{P}}}
\end{align}
for all $P,Q\in \NN_0^{1+n}$ with $\abs{P}=\abs{Q}$ and with $P' \coloneqq (P_1,\dots,P_n) \in \NN_0^n$, 
analogously for $Q$. To check this, note that the pullback with $\bigpr$ of the right-hand side
coincides with $\Monom{P}{Q} / \momentmap[\abs{P}]$ on $\CCres$, hence with $\Monom{P}{Q}$ on $\Levelset$.
Even though the monomials $\Monom{P}{Q}$ on $\CC^{1+n}$ are linearly independent,
this does no longer hold for their counterparts $\MonomRed{P}{Q}$ on $\Mred$.
Because of this we introduce:

\begin{definition} \label{definition:fundamentalMonomials}
  For all multiindices $P,Q\in \NN_0^n$ we define the 
  \neu{fundamental monomial on $\Mred$}
  \begin{align} \label{eq:definition:fundamentalMonomials}
    \MonomRedFund{P}{Q} \coloneqq 
    \begin{cases}
      \MonomRed{(\abs{Q}-\abs{P},P_1,\dots,P_n)}{(0,Q_1,\dots,Q_n)}
      &\text{if }\abs{P}\le\abs{Q} \komma \\
      \MonomRed{(0,P_1,\dots,P_n)}{(\abs{P}-\abs{Q},Q_1,\dots,Q_n)}
      &\text{if }\abs{P}\ge\abs{Q}    \punkt
    \end{cases}
  \end{align}
\end{definition}
Note that the fundamental monomials on $\Mred$ -- unlike the monomials on $\CC^{1+n}$ --
are determined by $2n$ indices, not $2n+2$. Using projective coordinates on $\Mred$, they
can be expressed as
\begin{align}
  \MonomRedFund{P}{Q}
  =
  \frac{ 
    w^P \cc{w}^Q
  }{
    (1 + \sum_{k=1}^n \nu_k w^k \cc{w}^k)^{\max\{\abs{P},\abs{Q}\}}
  }
\end{align}
for all $P,Q\in \NN_0^n$. While the usual easy multiplication rules
for monomials still hold for the $\MonomRed{P}{Q}$, i.e.\ 
$\MonomRed{P}{Q}\MonomRed{R}{S} = \MonomRed{P+R}{Q+S}$ for all $P,Q,R,S\in \NN_0^{1+n}$
with $\abs{P} = \abs{Q}$ and $\abs{R}=\abs{S}$, this is no longer true for the fundamental
monomials on $\Mred$. Their product can be obtained by rewriting them
in terms of the reduced monomials, which can easily be multiplied, and 
by applying the following:

\begin{lemma} \label{lemma:redexpansioninredfund}
  For all $P,Q\in \NN_0^{1+n}$ with $\abs{P}=\abs{Q}$, the identity
  \begin{align}
    \MonomRed{P}{Q} 
    = 
    \sum_{\substack{T\in\NN_0^n \\ \abs{T}\le\min\{P_0,Q_0\}}}
    (-1)^{\abs{T}}
    \sgn(T)
    \binom{\min\{P_0,Q_0\}}{\abs{T}}
    \frac{\abs{T}!}{T!}
    \MonomRedFund{P'+T}{Q'+T}
    \label{eq:expansioninfundamentals}
  \end{align}
  holds, where $P'\coloneqq (P_1,\dots,P_n) \in \NN_0^n$,
  $Q'\coloneqq (Q_1,\dots,Q_n) \in \NN_0^n$ and
  $\sgn(T) \coloneqq \prod_{k=1}^n \nu_k^{T_k}$.
\end{lemma}
\begin{proof}
  For $k\in \{0,\dots,n\}$, let $E_k \coloneqq (0,\dots,0,1,0,\dots,0) \in \NN_0^{1+n}$ be the tuple with $1$ at position $k$.
  From $\MonomRed{E_0}{E_0} = 1 - \sum_{k=1}^n \nu_k \MonomRed{E_k}{E_k}$ and the multinomial theorem it follows that
  \begin{align*}
    \big(\MonomRed{E_0}{E_0}\big)^{\min\{P_0,Q_0\}}
    &= 
    \sum_{\substack{T\in\NN_0^n \\ \abs{T}\le \min\{P_0,Q_0\}}}
    (-1)^{\abs{T}} \sgn(T)
    \binom{\min\{P_0,Q_0\}}{\abs{T}}
    \frac{\abs{T}!}{T!}
    \MonomRedFund{T}{T}
    \punkt
  \end{align*}
  Combining this with 
  $\MonomRed{P}{Q} = ( \MonomRed{E_0}{E_0} )^{\min\{P_0,Q_0\}} \MonomRedFund{P'}{Q'}$
  yields the desired result.
\end{proof}
Analogous to
\cite{beiser.waldmann:FrechetAlgebraicDeformationOfThePoincareDisc,
kraus.roth.schoetz.waldmann:OnlineConvergentStarProductOnPoincareDisc},
one can show that these fundamental monomials $\MonomRedFund{P}{Q}$ with 
$P,Q \in \NN_0^{n}$ are a Hamel basis of $\Polynomials(\Mred)$. We will
come back to this problem later in Section~\ref{sec:wick}.




\subsection{Analytic Functions}   \label{subsec:analyticfunctions}

The polynomial algebras discussed in the previous Subsection~\ref{subsec:polynomialfunctions} 
can be completed to algebras of certain analytic functions. More precisely,
we are interested in the pullbacks with $\Diag \colon \CC^{1+n} \to \CC^{1+n}\times\CC^{1+n}$
and $\DiagM \colon \Mred \to \MredExt$ of holomorphic functions:    

\begin{definition}
  By $\Holomorphic(M)$ we denote the unital complex algebra of holomorphic
  functions on a complex manifold $M$. Moreover, we define the following
  subsets of $\Smooth(\CC^{1+n})$ and $\Smooth(\Mred)$, respectively:
  \begin{align}
    \Analytic(\CC^{1+n}) 
    &\coloneqq 
    \set[\big]{\Diag^*(\hat{f})}{\hat{f}\in\Holomorphic(\CC^{1+n}\times\CC^{1+n})}
  \shortintertext{and}
    \Analytic(\Mred) 
    &\coloneqq 
    \set[\big]{\DiagM^*(\hat{g})}{\hat{g}\in\Holomorphic(\MredExt)} 
    \punkt
  \end{align}
\end{definition}
It is not hard to check that $\Analytic(\CC^{1+n})$ and $\Analytic(\Mred)$ are unital $^*$-subalgebras of
$\Smooth(\CC^{1+n})$ and $\Smooth(\Mred)$, respectively. Especially for the $^*$-involution one finds:
Given $\hat{f}\in\Holomorphic(\CC^{1+n}\times\CC^{1+n})$ or $\hat{g}\in\Holomorphic(\MredExt)$,
then one can define $\hat{f}{}'\in\Holomorphic(\CC^{1+n}\times\CC^{1+n})$ or
$\hat{g}{}'\in\Holomorphic(\MredExt)$ as the functions 
$\CC^{1+n}\times\CC^{1+n} 
\ni 
(\xi,\eta) 
\mapsto 
\hat{f}{}'(\xi,\eta) 
\coloneqq 
\cc{\hat{f}(\cc{\eta},\cc{\xi})} 
\in 
\CC$
or
$\MredExt
\ni 
([\xi,\eta]) 
\mapsto 
\hat{g}{}'([\xi,\eta])
\coloneqq 
\cc{\hat{g}([\cc{\eta},\cc{\xi}])}
\in 
\CC$, so that 
$\Diag ^*(\hat{f}{}') = \cc{\Diag ^*(\hat{f})}$ and
$\DiagM^*(\hat{g}{}') = \cc{\DiagM^*(\hat{g})}$, respectively.
As algebras, $\Holomorphic(\CC^{1+n}\times\CC^{1+n})$ and $\Analytic(\CC^{1+n})$ as well as
$\Holomorphic(\MredExt)$ and $\Analytic(\Mred)$ are isomorphic:

\begin{proposition} \label{proposition:analyticholomorphiciso}
  The pullbacks $\Diag^* \colon \Holomorphic(\CC^{1+n}\times\CC^{1+n}) \to \Analytic(\CC^{1+n})$ and
  $\DiagM^* \colon \Holomorphic(\MredExt) \to \Analytic(\Mred)$ are isomorphisms
  of algebras.
\end{proposition}
\begin{proof}
  It is easy to check that $\Diag^*$ and $\DiagM^*$ are homomorphisms of algebras, 
  and they are surjective by definition of $\Analytic(\CC^{1+n})$ and $\Analytic(\Mred)$,
  so only injectivity remains: Given $\hat{f} \in \Holomorphic(\CC^{1+n}\times\CC^{1+n})$ with
  $\Diag^*(\hat{f}) = 0$ or $\hat{g} \in \Holomorphic(\MredExt)$ with $\DiagM^*(\hat{g}) = 0$,
  then, in the coordinates introduced in Section~\ref{sec:geometry},
  \begin{align*}
    \frac{\partial \hat{f}}{\partial x^k} \at[\bigg]{(\rho,\cc{\rho})}
    &=
    \big(\Tangent_\rho \Diag\big)\bigg( \frac{\partial}{\partial z^k} \at[\bigg]{\rho} \bigg)\big(\hat{f}\big)
    =
    \frac{\partial}{\partial z^k} \at[\bigg]{\rho} \Diag^*\big(\hat{f}\big)
    =
    0
    &&\text{and}&
    \frac{\partial \hat{f}}{\partial y^k} \at[\bigg]{(\rho,\cc{\rho})}
    &=
    \frac{\partial}{\partial \cc{z}^k} \at[\bigg]{\rho} \Diag^*\big(\hat{f}\big)
    =
    0
  \intertext{or}
    \frac{\partial \hat{g}}{\partial u^\ell} \at[\bigg]{[(\rho,\cc{\rho})]}
    &=
    \big(\Tangent_{[\rho]} \DiagM\big)\bigg( \frac{\partial}{\partial w^\ell} \at[\bigg]{[\rho]} \bigg)\big(\hat{g}\big)
    =
    \frac{\partial}{\partial w^\ell} \at[\bigg]{[\rho]} \DiagM^*\big(\hat{g}\big)
    =
    0
    &&\text{and}&
    \frac{\partial \hat{g}}{\partial v^\ell} \at[\bigg]{[(\rho,\cc{\rho})]}
    &=
    \frac{\partial}{\partial \cc{w}^\ell} \at[\bigg]{[\rho]} \DiagM^*\big(\hat{g}\big)
    =
    0
  \end{align*}
  hold for all $\rho \in \Levelset$ with $z^0(\rho) \neq 0$ and all $k\in\{0,\dots,n\}$,
  $\ell\in\{1,\dots,n\}$, respectively. By iteration one finds that also
  all higher derivatives of $\hat{f}$ or $\hat{g}$ vanish, so that $\hat{f} = 0$ or
  $\hat{g} = 0$.
\end{proof}
It is well-known that the holomorphic functions $\Holomorphic(M)$ on a complex manifold $M$ with the
pointwise operations become a Fréchet algebra with the topology of locally uniform convergence
(i.e.\ $\Holomorphic(M)$ is complete and the multiplication continuous with respect to this
metrizable locally convex topology). This locally convex topology can be described by all the
seminorms $\seminorm{K}{\argument} \colon \Holomorphic(M) \to {[0,\infty[}$,
\begin{align}
  \hat{f} \mapsto \seminorm{K}{\hat{f}} \coloneqq \max_{z\in K} \abs{\hat{f}(z)}
  \label{eq:maxcompactseminorms}
\end{align}
with $K$ a compact subset of $M$. From this we see immediately that $\Analytic(\CC^{1+n})$ and
$\Analytic(\Mred)$ with the topology coming from $\Holomorphic(\CC^{1+n}\times\CC^{1+n})$ and $\Holomorphic(\MredExt)$,
respectively, are Fréchet $^*$-algebras (Fréchet algebras endowed with a continuous $^*$-involution).
It is a consequence of the Cauchy integral formula on $\CC^{1+n}\times\CC^{1+n}$
that every $f\in \Analytic(\CC^{1+n})$ can be expressed in a unique way as an absolutely convergent series
\begin{align}
 f = \sum_{P,Q\in \NN_0^{1+n}} \expansionCoefficients{f}{P}{Q} \Monom{P}{Q}
 \label{eq:analyticexpansion}
\end{align}
with complex coefficients $\expansionCoefficients{f}{P}{Q}$ fulfilling
\begin{equation} \label{eq:normoben}
   \seminorm{r}{f} 
   \coloneqq 
   \sum_{P,Q \in \NN_0^{1+n}} \abs{\expansionCoefficients{f}{P}{Q}} r^{\abs{P}+\abs{Q}}
   <
   \infty
\end{equation}
for all $r\in {[1,\infty[}$, and that the topology of $\Analytic(\CC^{1+n})$ can 
equivalently be described by these seminorms
$\seminorm{r}{\argument} \colon \Analytic(\CC^{1+n}) \to {[0,\infty[}$.
See e.g.\ \cite[Proposition 3.5]{schmitt:StrictQuantizationOfCodjointOrbits} for details.
We will later in Proposition~\ref{proposition:wickapplication_basis} obtain an analogous result also for $\Analytic(\Mred)$.
Like for polynomials one also finds that the $\group{U}(1)$-invariant analytic
functions $f$ are precisely those which fulfil $\expansionCoefficients{f}{P}{Q} = 0$
for all $P,Q\in \NN_0^{1+n}$ with $\abs{P}\neq \abs{Q}$, e.g.\ by explicitly calculating
the coefficients with the help of the Cauchy integral formula.
Note that due to the completeness of $\Analytic(\CC^{1+n})$, averaging over
the $\group{U}(1)$-action on $\Analytic(\CC^{1+n})$ is possible and yields for
every $f\in \Analytic(\CC^{1+n})$ an $f_{\mathrm{av}} \in \Analytic(\CC^{1+n})^{\group{U}(1)}$.

We observe that the reduction map $\argument_\red$ can be defined
analogously as before also for holomorphic functions:
\begin{lemma} \label{lemma:analyticred}
  Let $\hat{f} \in \Holomorphic(\CC^{1+n}\times\CC^{1+n})$ be $\CCx$-invariant
  in the sense that $\hat{f}\racts (\alpha\Unit_{1+n},\alpha^{-1}\Unit_{1+n}) = \hat{f}$
  holds for all $\alpha \in \CCx$, then there exists a unique
  $\hat{f}_{\hat{\red}} \in \Holomorphic(\MredExt)$ for which
  \begin{align*}
    \hat{\iota}{}^*\big(\hat{f}\big) = \hat{\pr}{}^*\big(\hat{f}_{\hat{\red}}\big)
  \end{align*}
  holds.
\end{lemma}
\begin{proof}
  As $\hat{\iota}{}^*(\hat{f})$ is $\CCx$-invariant, it descends to a well-defined
  function $\hat{f}_{\hat{\red}}$ on $\MredExt = \LevelsetExt / \CCx$,
  which is automatically holomorphic. Uniqueness of $\hat{f}_{\hat{\red}}$ is clear.
\end{proof}

\begin{proposition} \label{proposition:analyticred}
  The reduction map $\argument_\red$ restricts to a map from $\Analytic(\CC^{1+n})^{\group U(1)}$ to $\Analytic(\Mred)$.
  More precisely, given $f\in \Analytic(\CC^{1+n})^{\group U(1)}$ and $\hat{f} \in \Holomorphic(\CC^{1+n}\times\CC^{1+n})$
  such that $\Diag^*(\hat{f}) = f$, then $\hat{f}$ is $\CCx$-invariant in the sense of
  the previous Lemma~\ref{lemma:analyticred} and $f_\red = \DiagM^*(\hat{f}_{\hat{\red}}) \in \Analytic(\Mred)$.
\end{proposition}
\begin{proof}
  Given such $f$ and $\hat{f}$, then
  \begin{align*}
    \Diag^*\big(\hat{f} \racts (\E^{\I\phi}\Unit_{1+n},\E^{-\I\phi}\Unit_{1+n})\big)
    = 
    \Diag^*(\hat{f}) \racts \E^{\I\phi}\Unit_{1+n}
    =
    f\racts \E^{\I\phi}\Unit_{1+n}
    =
    f
    =
    \Diag^*(\hat{f})
  \end{align*}
  holds for all $\phi \in \RR$, so $\hat{f}$ is $\group{U}(1)$-invariant because $\Diag^*$
  is an isomorphism between $\Holomorphic(\CC^{1+n}\times\CC^{1+n})$ and $\Analytic(\CC^{1+n})$.
  But since the action of the complex Lie group $\CCx$ on 
  $\CC^{1+n} \times \CC^{1+n}$ is holomorphic, $\hat f$ is 
  even $\CCx$-invariant. Using the commutativity of
  the diagram in Section~\ref{sec:geometry}, one can now check that
  \begin{align*}
    \pr^*\big(\DiagM^*(\hat{f}_{\hat{\red}})\big)
    =
    \DiagZ^*\big(\hat{\pr}{}^*(\hat{f}_{\hat{\red}})\big)
    =
    \DiagZ^*\big(\hat{\iota}{}^*(\hat{f})\big)
    =
    \iota^*\big(\Diag^*(\hat{f})\big)
    =
    \iota^*(f)
  \end{align*}
  holds, hence $f_\red = \DiagM^*(\hat{f}_{\hat{\red}}) \in \Analytic(\Mred)$.
\end{proof}
Using some deep results from complex analysis, the analytic functions on $\Mred$ and on 
$\CC^{1+n}$ can be related in the same way as smooth or polynomial functions:

\begin{lemma} \label{lemma:analyticLift}
  For every $g\in \Analytic(\Mred)$ there exists an $f\in\Analytic(\CC^{1+n})^{\group{U}(1)}$
  such that $f_\red = g$.
\end{lemma}
\begin{proof}
  Given $g\in \Analytic(\Mred)$ and corresponding $\hat{g} \in \Holomorphic(\MredExt)$ such that
  $\DiagM^*(\hat{g}) = g$, then $\hat{\pr}^*(\hat{g})$ is a holomorphic function on $\LevelsetExt$.
  Now note that $\CC^{1+n}\times\CC^{1+n}$ is a Stein manifold by 
  \cite[Sec.~5.1]{hoermander:ComplexAnalysisInSeveralVariables} and that $\LevelsetExt$ 
  is -- in the language of \cite[Def.~6.5.1]{hoermander:ComplexAnalysisInSeveralVariables} -- an 
  analytic submanifold 
  thereof because it
  is the set of zeros of a holomorphic function on $\CC^{1+n}\times\CC^{1+n}$. So 
  \cite[Thm.~7.4.8]{hoermander:ComplexAnalysisInSeveralVariables} applies
  and shows that there exists an extension $\hat{f} \in \Holomorphic(\CC^{1+n}\times\CC^{1+n})$
  of $\hat{\pr}^*(\hat{g})$, i.e.\ $\hat{\iota}^*(\hat{f}) = \hat{\pr}^*(\hat{g})$.
  Therefore $f \coloneqq \Diag^*(\hat{f})$ fulfils $\iota^*(f) = \pr^*(g)$
  due to the commutativity of the diagram in Section~\ref{sec:geometry}.
  By averaging over the $\group{U}(1)$-action on $\Analytic(\CC^{1+n})$ we can even arrange that
  $f$ is $\group{U}(1)$-invariant.
\end{proof}
For an alternative proof one can also generalize the more constructive results obtained in
\cite[Sec.~3.2]{kraus.roth.schoetz.waldmann:OnlineConvergentStarProductOnPoincareDisc} for the case
of signature $s=1$, or use these results and the Wick rotation as discussed later in Section~\ref{sec:wick}.

Clearly, $\set{f \in \Analytic(\CC^{1+n})^{\group{U}(1)}}{ \iota^*(f) = 0 }$ is the kernel of $\cdot_\red$ 
restricted to $\Analytic(\CC^{1+n})$ and 
therefore a closed $^*$-ideal of $\Analytic(\CC^{1+n})^{\group{U}(1)}$. Similarly to the case of smooth 
or polynomial functions we get:

\begin{proposition} \label{proposition:topologies:quotientTopology}
	The reduction map $\argument_\red$ descends to a homeomorphic $^*$-isomorphism between the Fréchet $^*$-algebras  
	$\Analytic(\CC^{1+n})^{\group{U}(1)} / \set{f \in \Analytic(\CC^{1+n})}{ \iota^*(f) = 0 }$
	and $\Analytic(\Mred)$.
\end{proposition}
\begin{proof}
	Using Lemma~\ref{lemma:analyticred} it is clear that that $\argument_\red$ induces a $^*$-isomorphism.
	As $\seminorm{K}{\hat{f}_{\hat{\red}}} = \seminorm{B \cap \pr^*(K)}{\hat{f}}$
	holds for every $\hat{f} \in \Holomorphic(\CC^{1+n}\times \CC^{1+n})^{\CCx}$
	with $B \subseteq \CC^{1+n}\times\CC^{1+n}$ a sufficiently large
	closed ball, the map 
	$\argument_{\hat{\red}} \colon \Holomorphic(\CC^{1+n}\times \CC^{1+n})^{\CCx} \to \Holomorphic(\MredExt)$
	from Lemma~\ref{lemma:analyticred} is continuous with respect to the 
	topologies of locally uniform convergence, thus
	$\argument_{\red} \colon \Analytic(\CC^{1+n})^{\group{U}(1)} \to \Analytic(\Mred)$
	is continuous as well. It follows from the open mapping theorem 
	that it is a homeomorphism.
\end{proof}
As the $\group{U}(1)$-invariant polynomials $\Polynomials(\CC^{1+n})^{\group{U}(1)}$
are dense in $\Analytic(\CC^{1+n})^{\group{U}(1)}$, this immediately yields:

\begin{corollary} \label{corollary:polynomialsdense}
  The polynomials $\Polynomials(\Mred)$ are dense in $\Analytic(\Mred)$.
\end{corollary}




\section{Poisson Brackets and Star Products}   \label{sec:physics}

In this section we introduce a Poisson bracket and star product on $\CC^{1+n}$ and discuss their 
reduction to $\Mred$. First we consider formal star products, which make 
sense for formal power series of smooth functions. We present a method for reducing the (pseudo-)Wick 
product on $\CC^{1+n}$ to $\Mred$ in Subsection~\ref{subsec:physics:smooth} and derive more explicit formulas in 
Subsection~\ref{subsec:physics:explicit}. The other two sections deal with strict star products. 
In order to make the formal power series convergent, we restrict ourselves to polynomials in 
Subsection~\ref{subsec:physics:polynomial} and extend these results to analytic functions in 
Subsection~\ref{subsec:physics:analytic}.




\subsection{The Smooth Case}   \label{subsec:physics:smooth}

We will now introduce the Wick star product on $\CC^{1+n}$.
The antisymmetrization of its first order gives rise to a Poisson structure on $\CC^{1+n}$.
Let $\nabla$ be the Euclidean covariant derivative of $\CC^{1+n}$, $D$ its exterior covariant derivative
and $D^\sym$ the corresponding symmetrized covariant derivative, 
see Appendix~\ref{sec:derivatives} for details. We define
\begin{equation}\label{eq:def:H}
H 
\coloneqq 
\sum_{k=0}^n \nu_k \frac{\partial}{\partial\cc z^k} \otimes \frac{\partial}{\partial z^k}
\in \SmoothSections\big( \Tangent^{(0,1)} \CC^{1+n} \otimes \Tangent^{(1,0)} \CC^{1+n} \big)
\punkt
\end{equation}
It is easy to see that $H$ is $\group{U}(1)$-invariant, so that
$H_\red \in \SmoothSections( \Tangent^{(0,1)} \Mred \otimes \Tangent^{(1,0)} \Mred )$
can be defined as
\begin{equation}\label{eq:def:Hred}
H_\red \at{[\rho]}
\coloneqq 
\big(\Tangent_\rho \bigpr\big)^{\otimes 2} (H \at{\rho})
\end{equation}
for all $[\rho] \in \Mred$ with representative $\rho \in \Levelset$.
An explicit formula for $H_\red$ in projective coordinates will be given later in
Lemma~\ref{lemma:H:writtenWithW}.
Using $H$ and symmetrized covariant derivatives, we can now define the
well-known Wick star product:

\begin{definition} \label{definition:wick}
	The product 
	$\star \colon \Smooth(\CC^{1+n} )\formal{\lambda} \times 
	\Smooth(\CC^{1+n})\formal{\lambda} \to \Smooth(\CC^{1+n} )\formal{\lambda}$,
	\begin{equation} \label{eq:starProduct:Wick}
	(f,g) 
	\mapsto 
	f \star g 
	\coloneqq 
	\sum_{r=0}^\infty \frac {\lambda^r} {r!} 
	\dupr[\big]{(D^\sym)^r(f) \otimes (D^\sym)^r(g)}{H^r}
	\end{equation}
	is the \emph{(pseudo-)Wick star product} on $\CC^{1+n}$. Here $H^r$ denotes the $r$-th
	power of $H$ as an element of degree $(1,1)$ in the algebra
	$\SymSec^\bullet(\CC^{1+n}) \otimes \SymSec^\bullet(\CC^{1+n})$ with
	$\SymSec^\bullet(\CC^{1+n}) \coloneqq \bigoplus_{k=0}^\infty \SmoothSections( \Symten{k} \Tangent 
	\CC^{1+n} )$
	the algebra of symmetric multivector fields.
\end{definition}
Note that one can check that $\star$ is actually a $\Stab$-invariant Hermitian 
formal star product constructed out of the bidifferential operators
\begin{align}
 C_r(f,g) = \frac {1} {r!} \dupr[\big]{(D^\sym)^r(f) \otimes (D^\sym)^r(g)}{H^r} \punkt
\end{align}
It deforms in direction of the standard Poisson bracket with signature $s$
\begin{equation} \label{eq:definition:poissonBracket}
\frac{1}{\I}\big(C_1(f,g) - C_1(g,f)\big)
=
\frac 1 \I \sum_{k=0}^n \nu_k \left(
\frac{\partial f}{\partial\cc z^k} \frac{\partial g}{\partial z^k}
- \frac{\partial g}{\partial\cc z^k} \frac{\partial f}{\partial z^k}
\right)
\mathbin{\text{\reflectbox{$\coloneqq$}}}
\poi{f}{g}
\end{equation}
on $\CC^{1+n}$ with Poisson tensor
\begin{equation}\label{eq:poissontensor:cn}
\pi 
=
-2\I \sum_{k=0}^n \nu_k \frac{\partial}{\partial \cc z^k} \wedge \frac{\partial}{\partial z^k}
=
2\IM(H) \komma
\end{equation}
where, as usual, 
$\poi f g 
=
\dupr{\D f \otimes \D g}{\pi} 
=
\dupr{D^{\sym} f \otimes D^{\sym} g}{\pi}$.
Note that \eqref{eq:poissontensor:cn} implies that $\pi$ is a real tensor.

\begin{lemma}
	The Poisson bracket \eqref{eq:definition:poissonBracket} fulfils the condition for reducibility of 
	Proposition~\ref{proposition:reduction:sufficientCondition}.
\end{lemma}
\begin{proof}
	First, $\poi{\argument}{\argument}$ is bidifferential, hence can be restricted
	to $\CCres$. As $H$ is $\group{U}(1)$-invariant, the Poisson bracket is $\group{U}(1)$-invariant, too.
	One also finds that $\poi{f}{\momentmap} = X_\I(f)$ for all $f\in \Smooth(\CC^{1+n})$,
	with $X_\I$ the generator of the $\group U(1)$-action as before. 
	So if $f, g \in \Smooth(\CC^{1+n})^{\group U(1)}$
	are $\group U(1)$-invariant, then
	$\poi{f(\momentmap-1)}{g} = \poi{f}{g} (\momentmap-1) - f X_\I(g) = \poi{f}{g} (\momentmap-1)$
	vanishes on $\Levelset$, and similarly 
	$\poi{f}{g(\momentmap-1)} \at \Levelset = 0$.
\end{proof}
Thus we can construct a reduced Poisson bracket on $\Mred$ by application
of Definition~\ref{definition:reducible} and get:

\begin{proposition} \label{proposition:poired}
  The reduced Poisson bracket 
  $\poi\argument\argument_\red : \Smooth(\Mred) \times \Smooth(\Mred) \to \Smooth(\Mred)$
  is given for all $f ,g \in \Smooth(\Mred)$ and $\rho \in \Levelset$ by
  \begin{align}
    \poi{f}{g}_{\mathrm{red}}([\rho]) 
    = 
    \poi[\big]{\bigpr^*(f)}{\bigpr^*(g)}(\rho) 
    = 
    \dupr[\big]{\D f \otimes \D g \at{[\rho]}}{(\Tangent_\rho \bigpr)^{\otimes 2}\pi\at{\rho}}
    \label{eq:poired1}
  \end{align}
  and the corresponding Poisson tensor $\pi_\red$ on $\Mred$ is simply
  \begin{align}
    \pi_\red = 2 \IM(H_\red)
    \punkt
    \label{eq:poired2}
  \end{align}
\end{proposition}
\begin{proof}
  Equation \eqref{eq:poired1} is clear and \eqref{eq:poired2} then follows from
  \eqref{eq:def:Hred} and \eqref{eq:poissontensor:cn}.
\end{proof}
This is just the Poisson-algebraic analog of the Marsden--Weinstein reduction scheme.
However, the situation is a bit more difficult if one tries to reduce the bidifferential operators 
$C_r$ defining the Wick star product.
One immediately sees that Proposition~\ref{proposition:restriction} cannot be applied directly:
For example, $C_1 (\momentmap, \momentmap-1) = \momentmap \neq 0$.
Following \cite{bordemann.brischle.emmrich.waldmann:PhaseSpaceReductionForStarProducts.ExplicitConstruction},
this problem can be overcome by restricting to $\CCres$ and performing an equivalence transformation
$S = \id + \sum_{k=1}^\infty \lambda^k S_k$,
with differential operators $S_k \colon \Smooth(\CCres) \to 
\Smooth(\CCres)$
that vanish on constant functions,
from $\star$ to a suitable new star product $\startilde$, i.e.\ $f\startilde f' \coloneqq 
S(S^{-1}(f)\star S^{-1}(f'))$,
in such a way that $\startilde$ is reducible to a star product $\starred$ on $\Mred$ by application of 
Proposition~\ref{proposition:reduction:sufficientCondition}.
If this can be achieved, then $\pr^*(g \starred g') = ( \bigpr^*(g) \startilde \bigpr^*(g') )\at{\Levelset}$
for all $g,g'\in \Smooth(\Mred)$. For this we require the following:
\begin{enumerate}
	\item $S$ should commute with $\cc{\argument}$, since then $\startilde$ is again a Hermitian star 
	product. 
	\label{item:prop:Hermitian}
	\item $S$ should be $\Stab$-invariant, 
	since then $\startilde$ is again $\Stab$-invariant. \label{item:prop:U1equi}
	\item Moreover, $\startilde$ should fulfil $\momentmap \startilde f = \momentmap f$ for all 
	$f\in \Smooth(\CC^{1+n})^{\group{U}(1)}$,
	hence also 
	$f \startilde \momentmap = \cc{\momentmap \startilde \cc{f}} = \cc{\momentmap \cc{f}} = f \momentmap$
	for all $f\in \Smooth(\CC^{1+n})^{\group{U}(1)}$.
	As a consequence, Proposition~\ref{proposition:reduction:sufficientCondition} can be applied to 
	the bidifferential operator defining the $r$-th order of $\startilde$ for any $r$,
	so that $\starred$ as described above is indeed well-defined.  \label{item:prop:ptw}
	\item Finally, it would be helpful if $S$ (hence also $S^{-1}$) acts as the
	identity on $\CCx$-invariant functions, because this has the consequence that
	the formula for $\starred$ simplifies to
	\begin{align}
	\pr^*(g \starred g') = \big( \bigpr^*(g) \startilde \bigpr^*(g') \big)\at[\big]{\Levelset} = 
	\big(S \big(\bigpr^*(g) \star \bigpr^*(g') \big) \big)\at[\big]{\Levelset}
	\end{align}
	for all $g,g'\in \Smooth(\Mred)$.  \label{item:prop:partialh}
\end{enumerate}
Let us define the rescaled vector field
\begin{align}
  \frac{\partial}{\partial\momentmap}
  \coloneqq
  \frac 1 {2\momentmap} X_\Unit
  \in
  \SmoothSections(\Tangent \CCres)
\end{align}
on $\CCres$, which satisfies $\frac{\partial}{\partial\momentmap} \momentmap = 1$. Then 
Properties \refitem{item:prop:Hermitian}, \refitem{item:prop:U1equi} and 
\refitem{item:prop:partialh}
are fulfilled if all the differential operators $S_k$ with $k\in \NN$
are of the form 
$S_k = \sum_{\ell = 1}^\infty (S_{k,\ell} \circ \momentmap) 
\big(\frac{\partial}{\partial \momentmap}\big)^\ell$
with smooth functions $S_{k,\ell} \colon {]0,\infty[} \to \RR$, such that for every fixed $k\in \NN$
there are only finitely many $\ell \in \NN$ with $S_{k,\ell} \neq 0$.

We are interested in the inverse equivalence transformation $T = S^{-1}$, which then also 
contains only derivatives $\frac{\partial}{\partial \momentmap}$
and coefficient functions dependent on $\momentmap$,
i.e.\ $T = \id+\sum_{k,\ell = 1}^\infty \lambda^k (T_{k,\ell} \circ \momentmap) 
\big(\frac{\partial}{\partial \momentmap}\big)^\ell$.
Therefore recall the usual definition of the \emph{falling} and \emph{rising factorial} as
\begin{equation}
  \falling \xi r \coloneqq \prod_{k=0}^{r-1} (\xi-k)
  \quad\quad\text{and}\quad\quad
  \rising{\xi}{r} \coloneqq \prod_{k=0}^{r-1} (\xi+k)
  \komma
\end{equation}
respectively, for all elements $\xi$ of a ring with unit and all $r\in \NN_0$.
Here the empty product is of course $\falling{\xi}{0} = 1 = \rising{\xi}{0}$. 
For formal Laurent series in $\lambda$ over the smooth functions $\Smooth(M)$ of a manifold
$M$ and a pointwise invertible $f\in \Smooth(M)$ we see that
\begin{equation}
  \frac{1}{\falling{f/\lambda}{r}}
  =
  \frac{\lambda^r}{ \prod_{k=0}^{r-1} (f-k\lambda)}
  \quad\quad\text{and}\quad\quad
  \frac{1}{\rising{f/\lambda}{r}}
  =
  \frac{\lambda^r}{ \prod_{k=0}^{r-1} (f+k\lambda)}
\end{equation}
are actually formal power series because
$f \pm k \lambda \in \Smooth(M)\formal{\lambda}$ are invertible.

\begin{proposition} \label{proposition:Trafo1}
	Let $T$ be a $\group{U}(1)$-invariant equivalence transformation on $\CCres$ 
	from a new star product ${}\startilde{}$ to ${}\star{}$, then the following is equivalent:
	\begin{itemize}
		\item $T(\momentmap)=\momentmap$ and $\momentmap\startilde f = \momentmap f$ for all
		$f\in \Smooth(\CCres)^{\group{U}(1)}$.
		\item $\kom{T}{\momentmap}(f) = \lambda \momentmap \frac{\partial}{\partial \momentmap} 
		T(f)$ for all $f\in \Smooth(\CCres)^{\group{U}(1)}$, where $\kom{\argument}{\argument}$
		denotes the commutator.
	\end{itemize}
	If $T$ fulfils one, hence both of these conditions, then
	\begin{align}
	  T\big( \lambda^r \falling{\momentmap/\lambda}{r} \big) = {\momentmap}^r
	  \quad\quad\text{and}\quad\quad
	  T\bigg(\frac{\momentmap}{\lambda^{r+1} \rising{\momentmap/\lambda}{r+1}} \bigg) 
	  = 
	  {\momentmap}^{-r}
	\end{align}
  for all $r \in \NN_0$.
\end{proposition} 
 
\begin{proof}
 	Assume $T(\momentmap)=\momentmap$ and $\momentmap\startilde f = \momentmap f$
 	for some $f\in \Smooth(\CCres)^{\group{U}(1)}$, then
 	\begin{align*}
 	T(\momentmap f)
 	=
 	T(\momentmap \startilde f)
 	=
 	T(\momentmap) \star T(f)
 	=
 	\momentmap \star T(f)
 	=
 	\big( \momentmap + \lambda E \big) T(f)
  =
 	\bigg(\momentmap + \lambda \momentmap\frac{\partial}{\partial \momentmap} \bigg) T(f)
 	\end{align*}
 	and so $\kom{T}{\momentmap}(f) = \lambda \momentmap \frac{\partial}{\partial \momentmap} 
 	T(f)$. 
 	Conversely,	if $\kom{T}{\momentmap}(f) = \lambda \momentmap \frac{\partial}{\partial \momentmap} 
 	T(f)$
 	for all $f\in \Smooth(\CCres)^{\group{U}(1)}$, then especially for $f=1$ one gets
 	$T(\momentmap)-\momentmap T(1) = \lambda \momentmap \frac{\partial}{\partial \momentmap} T(1)$,
 	i.e.\ $T(\momentmap)-\momentmap = 0$ because $T(1) = 1$ for the equivalence transformation $T$.
 	Then one also checks that 
 	\begin{align*}
 	  \momentmap\startilde f 
 	  = 
 	  T^{-1}\bigg(\bigg(\momentmap + \lambda \momentmap\frac{\partial}{\partial \momentmap} \bigg) T(f)\bigg)
 	  = 
 	  T^{-1}\big( \momentmap T(f) + \kom{T}{\momentmap}(f)\big)
 	  =
 	  \momentmap f
      \punkt
 	\end{align*}
 	Moreover, by induction one finds that indeed 
 	$T( \lambda^r \falling{\momentmap/\lambda}{r} ) = {\momentmap}^r$ for all $r\in 
 	\NN_0$:
 	For $r=0$ this is just $T(1) = 1$, and if it holds for one $r\in \NN_0$, then
 	\begin{align*}
 	T\big( \lambda^{r+1} \falling{\momentmap/\lambda}{r+1} \big) 
 	&= 
 	\lambda T\big( \lambda^{r} 
 	\falling{\momentmap/\lambda}{r}(\momentmap/\lambda-r) \big) 
 	\\
 	&=
 	\kom{T}{\momentmap}\big( \lambda^{r} \falling{\momentmap/\lambda}{r} \big) 
 	+
 	\momentmap T\big( \lambda^{r} \falling{\momentmap/\lambda}{r} \big) 
 	-
 	\lambda r T\big( \lambda^{r} \falling{\momentmap/\lambda}{r} \big)
 	\\
 	&=
 	\bigg(\lambda \momentmap \frac{\partial}{\partial \momentmap} + \momentmap - \lambda 
 	r\bigg) T\big( 
 	\lambda^{r} 
 	\falling{\momentmap/\lambda}{r} \big)
 	\\
 	&=
 	\bigg(\lambda \momentmap \frac{\partial}{\partial \momentmap} + \momentmap - \lambda 
 	r\bigg) 
 	{\momentmap}^{r}
 	\\
 	&=
 	{\momentmap}^{r+1} \punkt
 	\end{align*}
 	In order to check the formula for ${\momentmap}^{-r}$, we note first that
 	\begin{align*}
 	(\momentmap+r \lambda) T\big( \momentmap (\lambda^{r+1} 
 	\rising{\momentmap/\lambda}{r+1})^{-1} \big)
 	&=
 	[\momentmap+r \lambda, T] \big( \momentmap (\lambda^{r+1} 
 	\rising{\momentmap/\lambda}{r+1})^{-1} \big)
 	+ T\big( \momentmap (\lambda^{r} \rising{\momentmap/\lambda}{r})^{-1} \big)
 	\\
 	&=
 	- \lambda \momentmap \frac \partial {\partial \momentmap} T \big( \momentmap 
 	(\lambda^{r+1} 
 	\rising{\momentmap/\lambda}{r+1})^{-1} \big)
 	+ T\big( \momentmap (\lambda^{r} \rising{\momentmap/\lambda}{r})^{-1} \big) \komma
 	\end{align*}
 	so
 	\begin{equation*}
 	\bigg(\momentmap+r \lambda+\lambda \momentmap \frac \partial {\partial 
 	\momentmap}\bigg) T
 	\big(
 		\momentmap (\lambda^{r+1} \rising{\momentmap/\lambda}{r+1})^{-1}
 	\big)
 	=
 	T\big( \momentmap (\lambda^{r} \rising{\momentmap/\lambda}{r})^{-1} \big) \punkt
 	\end{equation*}
 	Since $\momentmap$ is an invertible function on $\CCres$ it follows that 
 	$\momentmap+r \lambda+\lambda \momentmap \frac \partial {\partial \momentmap}$
 	is invertible on $\Smooth(\CCres)\formal{\lambda}$. Since
 	\begin{equation*}
 	\bigg(\momentmap+r \lambda+\lambda \momentmap \frac \partial {\partial 
 	\momentmap}\bigg) 
 	{\momentmap}^{-r}
 	=
 	{\momentmap}^{-r+1} + r \lambda {\momentmap}^{-r} -r \lambda {\momentmap}^{-r}
 	= {\momentmap}^{-r+1}
 	\end{equation*}
 	for all $r\in \NN$, we obtain
 	$\big(\momentmap+r\lambda+\lambda\momentmap\frac\partial{\partial \momentmap}\big)^{-1} 
 	({\momentmap}^{-r+1}) 
 	= {\momentmap}^{-r}$.
 	The statement now follows by induction because the base case $r=0$ reduces to $T(1) = 1$ and is 
 	therefore fulfilled.
\end{proof}
 
\begin{proposition} \label{proposition:Trafo2}
 	There exists a unique equivalence transformation $T$ on $\CCres$ of the form
 	\begin{equation}
 	T = \id + \sum_{k=1}^\infty \sum_{\ell = 1}^{2k} \lambda^k (T_{k,\ell} \circ \momentmap) \bigg( 
 	\frac{\partial}{\partial \momentmap}\bigg)^\ell
 	\label{eq:eqtrafoTexplizit}
 	\end{equation}
 	with $T_{k,\ell} \in \Smooth( {]0,\infty[})$ that has the properties
 	from the previous Proposition~\ref{proposition:Trafo1}.
 	Its inverse $S = T^{-1}$ thus has all the properties \refitem{item:prop:Hermitian} to \refitem{item:prop:partialh}
 	discussed above and additionally fulfils $S(\momentmap) = \momentmap$.
\end{proposition}

\begin{proof}
 	By collecting terms in $\lambda^k$ and $\big(\frac{\partial}{\partial \momentmap}\big)^\ell$,
 	the identity 
 	$\kom{T}{\momentmap} 
 	=
 	\lambda \momentmap \frac{\partial}{\partial \momentmap} T$
 	with $T$ like in \eqref{eq:eqtrafoTexplizit} is equivalent to
 	\begin{align*}
 	T_{k+1,\ell+1} \circ \momentmap = \frac{\momentmap}{\ell+1} \big( (T_{k,\ell}' + T_{k,\ell-1}) \circ \momentmap \big)
 	\end{align*}
 	for all $k\in \NN_0$, $\ell \in \NN_0$ with initial conditions
 	$T_{0,0} = 1$ and $T_{0,\ell} = 0 = T_{k,0}$ for all $k\in \NN$, $\ell \in \NN$,
 	where $T_{k,\ell}' \in \Smooth( {]0,\infty[})$ is the derivative of $T_{k,\ell}$
 	and where $T_{k,-1} \coloneqq 0$ for all $k\in \NN_0$.
\end{proof}
So the equivalence transformation $S$ exists and is uniquely determined if we add to the
four requirements \refitem{item:prop:Hermitian} to \refitem{item:prop:partialh}
above the fifth requirement that $S(\momentmap) = \momentmap$, which is just
a convenience.
We can now construct the reduced star product on $\Mred$:

\begin{definition} \label{definition:reducedStar}
  The \neu{transformed star product} $\startilde$ on $\CCres$ is the one obtained
  from $\star$ by application of the equivalence transformation $S = T^{-1}$
  with $T$ like in Proposition~\ref{proposition:Trafo2}. Explicitly,
  \begin{align}
    f \startilde g = S\big( T(f) \star T(g) \big)
  \end{align}
  for all $f,g\in \Smooth(\CCres)\formal{\lambda}$.
  Moreover, the \neu{reduced star product} $\starred$ on $\Mred$ is defined as
  \begin{align}
    f \starred g \coloneqq \sum_{r=0}^\infty \lambda^r \tilde C_{r,\red}(f,g)
  \end{align}
  for all $f,g\in \Smooth(\Mred)$ and extended to formal power series in $\lambda$,
  where $\tilde C_{r,\red}$ on $\Mred$ are the reductions
  like in Definition~\ref{definition:reducible} of the bidifferential operators
  $\tilde C_r$ on $\CCres$ that describe the transformed star product $\startilde$ on $\CCres$.
\end{definition}
Using the defining properties of the reduced bilinear maps $\tilde C_{r,\red}$
it is easy to check that $\starred$ is again associative and it is clear
that the constant $1$-function is the neutral element. It also follows from
the construction that $\tilde C_{r,\red}$ are bidifferential operators on $\Mred$,
but we will also shows this by giving an explicit formula in the next subsection.   
As $T$ and thus also $S$ commute with the
pointwise complex conjugation and the action of $\Stab$, both
$\startilde$ and $\starred$ are Hermitian and $\Stab$-invariant.
Note also that $\startilde$ still deforms in direction of the original Poisson bracket
$\poi{\argument}{\argument}$ (or rather, its restriction to $\CCres$),
so that it is easy to check that $\starred$ deforms in direction of the reduced
Poisson bracket $\poi{\argument}{\argument}_\red$ on $\Mred$.




\subsection{Explicit Formulae}   \label{subsec:physics:explicit}

We want to find an explicit expression for the reduced Poisson bracket
$\poi{\argument}{\argument}_\red$ and star product $\starred$ in terms of bidifferential
operators on $\Mred$.

\begin{lemma} \label{lemma:H:writtenWithW}
	The restriction to $\CCres$ of the tensor $H$ can be expressed as
	\begin{align} \label{eq:H:writtenWithW}
	H \at{\CCres} 
	&=
	\frac{1}{\momentmap} \cc{E} \otimes E 
	+
	H_\Xi
	\end{align}
	with some $H_\Xi \in \SmoothSections(\Xi \otimes \Xi)$.
	Explicitly,
	\begin{align} \label{eq:H:XiPart}
	H_\Xi 
	&=
	\frac{1}{\cc{z}^0z^0} \bigg(
      \sum_{k,\ell=1}^n \frac{\cc{z}^k z^\ell}{\cc{z}^0z^0 } \cc{W}_k \otimes W_\ell
      + 
      \sum_{k=1}^n \nu_k \cc{W}_k \otimes W_k 
	\bigg)
	\end{align}
	on the domain of definition of the vector fields $W_1, \dots, W_n$ and consequently
	\begin{equation} \label{eq:H:reduced}
	H_\red 
	= 
	\bigg(1 + \sum_{k=1}^n \nu_k \cc{w}^k w^k\bigg)
	\bigg(
      \sum_{k,\ell=1}^n \cc{w}^k w^\ell \frac{\partial}{\partial \cc{w}^k} \otimes \frac{\partial}{\partial w^\ell}
      + 
      \sum_{k=1}^n \nu_k \frac{\partial}{\partial \cc{w}^k} \otimes \frac{\partial}{\partial w^k}
	\bigg)
	\end{equation}
	in projective coordinates on $\Mred$.
\end{lemma} 
\begin{proof} The first part is an easy computation using \eqref{eq:zAsLinearCombOfW}.
	The formula for $H_\red$ then follows since 
	$(\cc{z}^0z^0)^{-1} \at{\Levelset} = ( \momentmap / \abs{z^0}^2 ) \at{\Levelset} = \bigpr^*(1+\sum_{k=1}^n \nu_k \cc{w}^k w^k ) \at{\Levelset}$
	and $(\Tangent_\rho \bigpr) (W_k \at{\rho}) = (\Tangent_\rho \pr) (W_k\at{\rho}) = \frac{\partial}{\partial w^k} \at{[\rho]}$
	by Proposition~\ref{proposition:wredused} and because $(\Tangent_\rho \bigpr) (E \at{\rho}) = 0$.
\end{proof}
As an immediate consequence we get from \eqref{eq:poired2}:

\begin{proposition}
  The reduced Poisson tensor $\pi_\red$ that determines $\poi{\argument}{\argument}_\red$ is
	\begin{equation} \label{eq:pi:reduced}
	\pi_\red 
	= 
	- 2 \I 
	\bigg(1 + \sum_{k=1}^n \nu_k \cc{w}^kw^k \bigg)
	\bigg(
      \sum_{k,\ell=1}^n \cc{w}^k w^\ell \frac{\partial}{\partial \cc{w}^k} \wedge \frac{\partial}{\partial w^\ell}
      + 
      \sum_{k=1}^n \nu_k \frac{\partial}{\partial \cc{w}^k} \wedge \frac{\partial}{\partial w^k}
	\bigg)
	\end{equation}  
	in projective coordinates.
\end{proposition}
For the signature $s=1+n$, this is the usual Poisson tensor associated to the symplectic Fubini--Study form 
on $\Mred[{\signature[1+n]}] \cong \CC\PP^n$. If $s=1$, then one obtains (up to a sign) the Poisson tensor 
associated to the symplectic Fubini--Study form on the hyperbolic disc $\Mred[{\signature[1]}] \cong \DD^n$.

Similarly to the Wick star product from Definition~\ref{definition:wick},
the bidifferential operators defining the reduced star product
should be expressed using symmetrized covariant derivatives. In order to
define reduced symmetrized covariant derivatives we need the following:

\begin{definition}
  We write $\bigproj_\Xi \colon \SmoothSections(\Tangent \CCres) \to \SmoothSections(\Tangent \CCres)$
  for the projection on the subbundle $\Xi$ of $\Tangent \CCres$ associated to the decomposition
  $\Tangent \CCres = \genHull{X_\Unit} \oplus \genHull{X_\I} \oplus \Xi$
  from Proposition~\ref{proposition:decomposition:CC}. Moreover, its dual will be denoted by
  $\bigproj_\Xi^* \colon \SmoothSections(\CoTangent \CCres) \to \SmoothSections(\CoTangent \CCres)$.
\end{definition}
Note that $\bigproj_\Xi$ commutes with the complex structure $I$ of $\CCres$.
Like in Propositions~\ref{proposition:covDer:reduction} and \ref{proposition:covsymDer:reduction}
we can construct a reduced exterior covariant derivative and a reduced symmetrized covariant derivative
on $\Mred$ out of $D$ and $D^\sym$ on $\CCres$, because $D$ and $D^\sym$ are $\CCx$-invariant 
(even invariant under arbitrary linear automorphisms of $\CC^{1+n}$):

\begin{definition}
  By
  $D_\red \colon (\ASymSec\otimes \SymSec)^{\bullet,\bullet}(\Mred) \to (\ASymSec\otimes \SymSec)^{\bullet+1,\bullet}(\Mred)$
  we denote the
  \neu{reduced exterior covariant derivative} on $\Mred$, which is the one that fulfils
  \begin{align}
    \bigpr^*\big(D_\red \Omega\big) = (\bigproj^*_\Xi)^{\otimes (k+1+\ell)} D \bigpr^*(\Omega)
  \end{align}
  for all $\Omega \in (\ASymSec \otimes \SymSec)^{k,\ell}(\Mred)$, $k,\ell\in \NN_0$, and
  analogously, the \neu{reduced symmetrized covariant derivative}
  $D^\sym_\red \colon \SymSec^\bullet(\Mred) \to \SymSec^{\bullet+1}(\Mred)$ on $\Mred$ is determined
  by
\begin{align}
  \bigpr^*\big(D^\sym_\red \omega\big) = (\bigproj^*_\Xi)^{\otimes (k+1)} D^\sym \bigpr^*(\omega)
  \label{eq:definition:Dsymred}
\end{align}
for all $\omega \in \SymSec^k(\Mred)$, $k\in \NN_0$.
\end{definition}
We will give a more explicit characterization of the corresponding covariant derivative
on $\Mred$ later in Proposition~\ref{proposition:Dredchar}. Note that $D$ is compatible
with the complex structure on $\CC^{1+n}$ in the sense of Definition~\ref{definition:compatibleWithComplexStructure},
and thus splits into $(1,0)$ and $(0,1)$-components $D = D_\hol + D_{\cc{\hol}}$
like in Definition~\ref{definition:holantiholofD}, analogously $D^\sym = D^\sym_\hol + D^\sym_{\cc{\hol}}$.
This carries over to the reduced derivatives:
\begin{proposition} \label{proposition:Dsymredhol}
  The reduced exterior covariant derivative $D_\red$ is compatible with the
  complex structure and
  \begin{equation}
    \bigpr^*(D^\sym_{\red,\hol} \omega) = (\bigproj_\Xi^*)^{\otimes(k+1)} D^\sym_\hol \bigpr^*(\omega)
    \label{eq:definition:Dsymredhol}
  \end{equation}
  holds for all $\omega \in \SymSec^k(\Mred)$.
\end{proposition}
\begin{proof}
  As a consequence of Proposition~\ref{proposition:decomposition:CC} the projection
  $\bigproj^*_\Xi$ commutes with the complex structure $I$ on $\CCres$ and therefore 
  $(\bigproj^*_\Xi)^{\otimes(p+q)}$ commutes with the projection onto 
  symmetric tensors of degree $(p, q)$. The projection onto such tensors also 
  commutes with $\bigpr^*$ since $\bigpr$ is holomorphic. Therefore $D_\red$ is
  compatible with the complex structure and 
  \eqref{eq:definition:Dsymredhol} follows immediately from \eqref{eq:definition:Dsymred}.
\end{proof}
We can now formulate the main theorem of this section:

\begin{theorem}\label{theorem:reducedWickProductFormula}
 	The reduced Wick star product is
 	\begin{equation}
 	f \starred g
 	=
 	\sum_{r=0}^\infty \frac 1 {r!} \frac {1} {\falling{1/\lambda} {r}} 
 	\dupr[\big]{(D_\red^\sym)^r f \otimes (D_\red^\sym)^r g }{ H_\red^r }
 	\end{equation}
 	for all $f,g\in \Smooth(\Mred)$,
 	where $H_\red\at{[\rho]} = (\Tangent_\rho \bigpr)^{\otimes 2} H \at \rho$ was computed
 	in Lemma~\ref{lemma:H:writtenWithW}. Moreover, if even $g\in \Polynomials(\Mred)$,
 	then the series in $r$ in the product $f \starred g$ has only finitely many non-zero terms.
\end{theorem}
Note that for complex projective spaces and hyperbolic discs this formula
coincides (up to rescaling the formal parameter) with the formula derived in 
\cite[Thm.~3.2.4]{loeffler:FedosovDifferentialsAndCartanNumbers}
for a Fedosov star product with form $\Omega = 0$.
For the proof of Theorem~\ref{theorem:reducedWickProductFormula} we have to collect some intermediate results:

\begin{lemma} \label{lemma:transformedWickRewritten}
 	On $\CCx$-invariant functions $f, g \in \Smooth(\CCres)^{\CCx}$ the transformed Wick star product can be 
 	expressed as
 	\begin{equation} \label{eq:starProduct:transformedWick}
 	f \startilde g
 	=
 	S (f \star g)
 	=
 	\sum_{r=0}^\infty \frac{1}{ r!} \frac {\momentmap/\lambda} 
 	{\rising{\momentmap/\lambda}{r+1}} 
 	\dupr[\big]{(D^\sym_{\antihol})^r f \otimes (D^\sym_\hol)^r g}{ H^r {\momentmap}^r } \punkt
 	\end{equation}
\end{lemma}
\begin{proof}
The first equality in \eqref{eq:starProduct:transformedWick} follows from requirement 
\refitem{item:prop:partialh} for the equivalence transformation.
For the second one we use that we can express $f\star g$ as
\begin{align*}
  f \star g 
  =
  \sum_{r=0}^\infty \frac{\lambda^r}{ {\momentmap}^r r!} 
  \dupr[\big]{(D^\sym)^r f \otimes (D^\sym)^r g}{ H^r {\momentmap}^r } \punkt
\end{align*}
Note that $\dupr{(D^\sym)^r f \otimes (D^\sym)^r g}{ H^r {\momentmap}^r }$ is $\CCx$-invariant, 
so that all its derivatives $\frac{\partial}{\partial \momentmap}$ vanish.
Now 
$\dupr{(D^\sym)^r f \otimes (D^\sym)^r g}{ H^r {\momentmap}^r } = 
\dupr{(D^\sym_{\antihol})^r f \otimes (D^\sym_\hol)^r g}{ H^r {\momentmap}^r }$
because of Proposition~\ref{proposition:holahol} and since the first tensor factor of $H$ lies in 
$\Tangent^{(0,1)}\CC^{1+n}$ and the second one in
$\Tangent^{(1,0)}\CC^{1+n}$.
Then it only remains to apply the formula for
$S({\momentmap}^{-r})$ from Proposition~\ref{proposition:Trafo1}. 
\end{proof}
If we restrict \eqref{eq:starProduct:transformedWick} to $\Levelset$, we can substitute
$\momentmap$ by $1$. In order to express 
$\dupr{(D^\sym_{\antihol})^r f \otimes (D^\sym_\hol)^r g}{ H^r }\at{\Levelset}$
with $\CCx$-invariant functions $f$ and $g$ by differential operators on $\Mred$,
we use formula \eqref{eq:H:writtenWithW} for $H$ and explicitly calculate the
contribution of the vertical directions $E$ and $\cc{E}$:

\begin{lemma} \label{lemma:kommutatoriotaEDhol}
	For $\CCx$-invariant $\sigma \in \SmoothSections(\Symten k \Tangent^{*,(1,0)} \CCres)^{\CCx}$,
	$k\in \NN_0$, we get $\kom{\iota_E}{ D_\hol^\sym }(\sigma) = -2 k \sigma$.
\end{lemma}

\begin{proof}
    Let $k\in \NN_0$ be given. For a multiindex $P \in \NN_0^{1+n}$ we write 
    $(\D z)^P \coloneqq (\D z^0)^{P_0} \vee \dots \vee (\D z^n)^{P_n}$, then a
    general element $\sigma \in \SmoothSections(\Symten k \Tangent^{*,(1,0)} \CCres)^{\CCx}$ can be expanded as
    $\sigma = \sum_{P\in \NN_0^{1+n}, \abs{P}=k} \sigma_P (\D z)^P$ with coefficient
    functions $\sigma_P \in \Smooth(\CCres)$.
    Note that $\CCx$-invariance of $\sigma$ implies that its Lie derivative with $E$
    vanishes, $\mathscr L_E \sigma = 0$. As
    $\mathscr L_E \D z^\ell = \D z^\ell$ for all $\ell\in \{0,\dots,n\}$,
    hence $\mathscr L_E (\D z)^P = k (\D z)^P$, it follows that
    $\mathscr L_E (\sigma_P) = -k \sigma_P$ and thus 
    $\kom{\iota_E}{ D_\hol^\sym } (\sigma_P) = \iota_E \D (\sigma_P) = \mathscr L_E (\sigma_P) = -k \sigma_P$.
    Moreover, $\kom{\iota_E}{ D_\hol^\sym } \D z^\ell = - D_\hol^\sym \iota_E \D z^\ell = - \D z^\ell$ for all $\ell\in \{0,\dots,n\}$.
    As both $\iota_E$ and $D_\hol^\sym$ are derivations, their commutator $\kom{\iota_E}{ D_\hol^\sym }$
    is also a derivation and we obtain $\kom{\iota_E}{ D_\hol^\sym }\sigma = -2k \sigma$.
\end{proof}

\begin{lemma} \label{lemma:insertEInSymCovDer}
	For $f,g \in \Smooth(\CCres)^{\CCx}$ and $r\in \NN$ we have
	\begin{align}
	\iota_{\cc E} (D_{\antihol}^\sym)^r f = - r (r-1) (D_{\antihol}^\sym)^{r-1} f
	\quad\text{and}\quad 
	\iota_E       (D_{  \hol  }^\sym)^r g = - r (r-1) (D_{  \hol  }^\sym)^{r-1} g
	\label{eq:insertEInSymCovDer:1}
	\end{align}
	as well as
	\begin{align}
	  \dupr[\big]{(D^\sym_{\antihol})^r f \otimes (D^\sym_\hol)^r g}{ H^r }\at[\big]{\Levelset}
	  =
	  \sum_{k=1}^r\frac{r!(r-k)!}{k!} \binom{r-1}{k-1}^2
	  \dupr[\big]{(D^\sym_{\antihol})^k f \otimes (D^\sym_\hol)^k g}{ (H_\Xi)^k }\at[\big]{\Levelset}
	  \komma
	\end{align}
	where $H_\Xi$ is the component of $H$ in $\Xi\otimes \Xi$, as defined in \eqref{eq:H:XiPart}.
\end{lemma}
\begin{proof}
	For \eqref{eq:insertEInSymCovDer:1} it suffices to prove the second statement 
	since the first one then follows by taking complex conjugates.
	Note that $(D^\sym_\hol)^k g$ is $\CCx$-invariant, so the previous 
	Lemma~\ref{lemma:kommutatoriotaEDhol} yields
	\begin{equation*}
	\iota_E (D_\hol^\sym)^r g
	= \sum_{k=0}^{r-1} (D_\hol^\sym)^{r-k-1} \kom[\big]{\iota_E}{D_\hol^\sym} (D_\hol^\sym)^k g
	= \sum_{k=0}^{r-1} (-2 k) (D_{  \hol  }^\sym)^{r-1} g
	= - r (r-1) (D_{  \hol  }^\sym)^{r-1} g
	\end{equation*}
	for all $r\in \NN_0$.
	With this and $H \at{\Levelset} = \cc{E} \otimes E \at{\Levelset} + H_\Xi \at{\Levelset}$
	from Lemma~\ref{lemma:H:writtenWithW} we can now calculate
	\begin{align*}
	  \dupr[\big]{(&D^\sym_{\antihol})^rf \otimes (D^\sym_\hol)^r g}{ H^r }\at[\big]{\Levelset}
	  =\\  
	  &=
	  \sum_{k=0}^r \binom{r}{k} 
	  \dupr[\big]{
        (D^\sym_{\antihol})^r f \otimes (D^\sym_\hol)^r g
      }{ 
        (\cc{E} \otimes E)^{r-k} (H_\Xi)^k
      }\at[\big]{\Levelset}
      \\
      &\mathrel{\smash{\overset{\mathclap{\strut{(1)}}}=}}
      \sum_{k=0}^r \binom{r}{k} \bigg( \frac{k!}{r!} \bigg)^2
      \dupr[\big]{
        (\iota_{\cc{E}})^{r-k}(D^\sym_{\antihol})^r f \otimes (\iota_{E})^{r-k}(D^\sym_\hol)^r g
      }{ 
        (H_\Xi)^k
      }\at[\big]{\Levelset}
      \\
      &\mathrel{\smash{\overset{\mathclap{\strut{(2)}}}=}}
      \sum_{k=1}^r \binom{r}{k} \bigg( \frac{(r-1)!}{(k-1)!}\bigg)^2
      \dupr[\big]{
        (D^\sym_{\antihol})^k f \otimes (D^\sym_\hol)^k g
      }{ 
        (H_\Xi)^k
      }\at[\big]{\Levelset}
      \\
      &=
      \sum_{k=1}^r \frac{r!(r-k)!}{k!} \binom{r-1}{k-1}^2
      \dupr[\big]{
        (D^\sym_{\antihol})^k f \otimes (D^\sym_\hol)^k g
      }{ 
        (H_\Xi)^k
      }\at[\big]{\Levelset}
      \punkt
	\end{align*}
	The factors appearing in step (1) are due to our conventions for the symmetric product,
	the dual pairing and the insertion derivation, see Equation \eqref{eq:insertionpairing}.
	In (2) we used 
	\begin{align*}
	(\iota_E)^{r-k} (D^\sym_\hol)^r g 
	= 
	(-1)^{r-k} \frac{r!}{k!} \frac{(r-1)!}{(k-1)!} g
	\end{align*}
	and its complex conjugate,
	which can be obtained by applying \eqref{eq:insertEInSymCovDer:1} several times.
	In the special case $k=0$, \eqref{eq:insertEInSymCovDer:1} yields $(\iota_E)^{r} (D^\sym_\hol)^r g = 0$.
\end{proof}
The combinatorial factors in \eqref{eq:starProduct:transformedWick} can be simplified using:

\begin{lemma} \label{lemma:relationBetweenRisingAndFallingFactorials}
  For all $k \in \NN$ the identity
  \begin{equation}
    \sum_{s=0}^\infty 
    \frac{1/\lambda}{\rising{1/\lambda}{k+s+1}} \binom{k+s-1}{k-1}^2 s!
    =
    \frac 1 {\falling{1/\lambda}{k}}     
    \label{eq:fallingrisingid}
  \end{equation}
  holds for a formal parameter $\lambda$.
\end{lemma}
\begin{proof}
  By multiplication with $\lambda^{-k}$ we see that \eqref{eq:fallingrisingid} is
  equivalent to the identity
  \begin{equation}
    \sum_{s=0}^\infty 
    \frac{\lambda^s s!}{\prod_{\ell = 1}^{k+s}(1+\lambda \ell)} \binom{k+s-1}{k-1}^2
    =
    \frac 1 {\prod_{\ell=1}^{k-1}(1-\lambda \ell)}
    \label{eq:fallingrisingidEquiv}
    \tag{\text{$*$}}
  \end{equation}
  of formal power series, which can be proven by induction over $k$: First note that
  \begin{equation*}
    \sum_{s=0}^\infty 
    \frac{\lambda^s s!}{\prod_{\ell = 1}^{1+s}(1+\lambda \ell)}
    =
    \sum_{s=0}^\infty 
    \frac{\lambda^s s!}{\prod_{\ell = 1}^{s}(1+\lambda \ell)} \bigg(1 - \frac{\lambda(s+1)}{1+\lambda(s+1)} \bigg)
    =
    \sum_{s=0}^\infty 
    \frac{\lambda^s s!}{\prod_{\ell = 1}^{s}(1+\lambda \ell)}
    -
    \sum_{s=0}^\infty 
    \frac{\lambda^{s+1} (s+1)!}{\prod_{\ell = 1}^{s+1}(1+\lambda \ell)}
    \komma
  \end{equation*}
  which is a telescope sum that gives the result $1$. So 
  \eqref{eq:fallingrisingidEquiv} holds for $k=1$. Now assume that 
  \eqref{eq:fallingrisingidEquiv} holds for some $k\in \NN$, then we get for $k+1$:
  \begin{align*}
    &\quad\quad\sum_{s=0}^\infty
    \frac{\lambda^s s!}{\prod_{\ell = 1}^{k+s+1}(1+\lambda \ell)} \binom{k+s}{k}^2
    =
    \\
    &=
    \frac{1}{1-\lambda k} \sum_{s=0}^\infty \bigg(
      \frac{\lambda^s s!}{\prod_{\ell = 1}^{k+s+1}(1+\lambda \ell)} 
      -
      \frac{\lambda^{s+1} s! k }{\prod_{\ell = 1}^{k+s+1}(1+\lambda \ell)}
    \bigg) \binom{k+s}{k}^2
    \\
    &=
    \frac{1}{1-\lambda k} \sum_{s=0}^\infty \bigg(
      \frac{\lambda^s s!}{\prod_{\ell = 1}^{k+s}(1+\lambda \ell)} \bigg( 1 - \frac{\lambda (k+s+1)}{1+\lambda(k+s+1)} \bigg)
      -
      \frac{\lambda^{s+1} s! k}{\prod_{\ell = 1}^{k+s+1}(1+\lambda \ell)}
    \bigg) \binom{k+s}{k}^2
    \\
    &=
    \frac{1}{1-\lambda k} \sum_{s=0}^\infty \bigg(
      \frac{\lambda^s s!}{\prod_{\ell = 1}^{k+s}(1+\lambda \ell)}
      -
      \frac{\lambda^{s+1} (s+1)!}{\prod_{\ell = 1}^{k+s+1}(1+\lambda \ell)} \frac{k+s+1}{s+1}
      -
      \frac{\lambda^{s+1} (s+1)!}{\prod_{\ell = 1}^{k+s+1}(1+\lambda \ell)} \frac{k}{s+1}
    \bigg) \binom{k+s}{k}^2
    \\
    &\mathrel{\smash{\overset{\mathclap{\strut{(0)}}}=}}
    \frac{1}{1-\lambda k} \Bigg(
      \frac{1}{\prod_{\ell = 1}^{k}(1+\lambda \ell)}
      +
      \sum_{s=0}^\infty
      \frac{\lambda^{s+1} (s+1)!}{\prod_{\ell = 1}^{k+s+1}(1+\lambda \ell)} \bigg(
        \frac{(k+s+1)^2}{(s+1)^2}
        -
        \frac{k+s+1}{s+1}
        -
        \frac{k}{s+1}
      \bigg) \binom{k+s}{k}^2
    \Bigg)
    \\
    &=
    \frac{1}{1-\lambda k} \Bigg(
      \frac{1}{\prod_{\ell = 1}^{k}(1+\lambda \ell)}
      +
      \sum_{s=0}^\infty
      \frac{\lambda^{s+1} (s+1)!}{\prod_{\ell = 1}^{k+s+1}(1+\lambda \ell)} \binom{k+s}{k-1}^2
    \Bigg)
    \\
    &=
    \frac{1}{1-\lambda k} \Bigg(
      \frac{1}{\prod_{\ell = 1}^{k}(1+\lambda \ell)}
      +
      \sum_{s=1}^\infty
      \frac{\lambda^s s!}{\prod_{\ell = 1}^{k+s}(1+\lambda \ell)} \binom{k+s-1}{k-1}^2
    \Bigg)
    \\
    &=
    \frac{1}{1-\lambda k}
    \sum_{s=0}^\infty
    \frac{\lambda^s s!}{\prod_{\ell = 1}^{k+s}(1+\lambda \ell)} \binom{k+s-1}{k-1}^2
    \\
    &=
    \frac 1 {(1-\lambda k)\prod_{\ell=1}^{k-1}(1-\lambda \ell)}
    \\
    &=
    \frac 1 {\prod_{\ell=1}^{k}(1-\lambda \ell)}
    \punkt
  \end{align*}
  At (0) we have used that
  \begin{align*}
    \sum_{s=0}^\infty \frac{\lambda^s s!}{\prod_{\ell = 1}^{k+s}(1+\lambda \ell)} \binom{k+s}{k}^2
    &=
    \frac{1}{\prod_{\ell = 1}^{k}(1+\lambda \ell)}
    +
    \sum_{s=1}^\infty \frac{\lambda^s s!}{\prod_{\ell = 1}^{k+s}(1+\lambda \ell)} \frac{(k+s)^2}{s^2}\binom{k+s-1}{k}^2
    \\
    &=
    \frac{1}{\prod_{\ell = 1}^{k}(1+\lambda \ell)}
    +
    \sum_{s=0}^\infty \frac{\lambda^{s+1} (s+1)!}{\prod_{\ell = 1}^{k+s+1}(1+\lambda \ell)} \frac{(k+s+1)^2}{(s+1)^2} \binom{k+s}{k}^2    
    \punkt
  \end{align*}
\end{proof}
The last, crucial step is the following observation:

\begin{lemma} \label{lemma:DofEStar}
	We have
	\begin{equation} \label{eq:dummy}
	D^\sym_{\antihol} \cc E^* = -(\cc E^*)^2
	\quad\quad\text{as well as}\quad\quad
	D^\sym_\hol E^* = -(E^*)^2 
	\end{equation}
	and consequently
	\begin{align}
	(\bigproj^*_\Xi)^{\otimes (k+1)} D^\sym_{\antihol} (\bigproj^*_\Xi)^{\otimes k} \cc{\omega}
	&=
	(\bigproj^*_\Xi)^{\otimes (k+1)} D^\sym_{\antihol} \cc{\omega}
	\shortintertext{as well as}
	(\bigproj^*_\Xi)^{\otimes (k+1)} D^\sym_{  \hol  } (\bigproj^*_\Xi)^{\otimes k}     \omega
	&= 
	(\bigproj^*_\Xi)^{\otimes (k+1)} D^\sym_{  \hol  }     \omega
	\label{eq:DofEStar:2}
	\end{align}
	for all $\omega \in \SmoothSections(\Symten{k} \Tangent^{*,(1,0)} \CCres)$
	with $k\in \NN_0$.
\end{lemma}
\begin{proof}
	Again, it suffices to prove the second equalities since the first ones then follow by taking complex 
	conjugates. Using \eqref{eq:EWAsLinearCombOfDz} and \eqref{eq:DzAsLinearCombOfWStar}, an easy 
	computation shows
	\begin{align*}
	D^\sym_\hol E^* 
	= D^\sym_\hol \bigg(\frac 1 {\momentmap} \sum_{k=0}^n \nu_k \cc z^k \D z^k \bigg)
	=
	\big( D^\sym_\hol \momentmap{}^{-1} \big) \sum_{k=0}^n \nu_k \cc{z}^k \D z^k
	=
	- \momentmap{}^{-2} \bigg( \sum_{k=0}^n \nu_k \cc{z}^k \D z^k \bigg)^2
	= 
	-(E^*)^2 \punkt
	\end{align*}
	For \eqref{eq:DofEStar:2} it is sufficient to consider the case $k=1$,
	the general case then follows from the algebraic properties of 
	$\bigproj^*_\Xi$ and $D^\sym_\hol$ (i.e.\ being a projection and a derivation).
	If $k=1$, then there is an $f \in \Smooth(\CCres)$ such that
	$\bigproj^*_\Xi \omega - \omega = f E^*$, and thus
	\begin{align*}
	D^\sym_{\hol} \bigproj^*_\Xi \omega - D^\sym_{\hol} \omega
	=
	D^\sym_{\hol} ( \bigproj^*_\Xi \omega - \omega )
	=
	D^\sym_{\hol} f E^*
	=
	\D f \vee E^*
	-
	f (E^*)^2
	\end{align*}
    is in the kernel of $(\bigproj^*_\Xi)^{\otimes 2}$.
\end{proof}

\begin{Empty}[Proof of Theorem~\ref{theorem:reducedWickProductFormula}:]
    The reduced star product
    on $\Mred$ fulfils
    \begin{align*}
    \pr^*(f \starred g) 
    = 
    \big(S \big(\bigpr^*(f) \star \bigpr^*(g) \big) \big)\at[\big]{\Levelset}
    \end{align*}
    for all $f,g\in\Smooth(\Mred)$.
    Application of first Lemma~\ref{lemma:transformedWickRewritten} and then
    Lemma~\ref{lemma:insertEInSymCovDer} now yields
    \begin{align*}
      \pr^*&(f \starred g) =
      \\
      &= 
      \sum_{r=0}^\infty \frac{1}{ r!} \frac {1/\lambda} 
      {\rising{1/\lambda}{r+1}}
      \dupr[\big]{(D^\sym_{\antihol})^r \bigpr^*(f) \otimes (D^\sym_\hol)^r \bigpr^*(g)}{ H^r }
      \at[\big]{\Levelset}
      \\
      &=
      fg\at[\big]{\Levelset}
      +
      \sum_{r=1}^\infty
      \sum_{k=1}^r
      \frac {1/\lambda} {\rising{1/\lambda}{r+1}}
      \frac{(r-k)!}{k!} \binom{r-1}{k-1}^2
      \dupr[\big]{(D^\sym_{\antihol})^k \bigpr^*(f) \otimes (D^\sym_\hol)^k \bigpr^*(g)}{ (H_\Xi)^k }\at[\big]{\Levelset}
    \end{align*}
    for all $f,g\in\Smooth(\Mred)$. By collecting the $k$-th derivatives and using
    Lemma~\ref{lemma:relationBetweenRisingAndFallingFactorials} this leads to
    \begin{align*}
      \pr^*&(f \starred g) =
      \\
      &=
      fg\at[\big]{\Levelset}
      +
      \sum_{k=1}^\infty
      \sum_{r=k}^\infty
      \frac {1/\lambda} {\rising{1/\lambda}{r+1}}
      \frac{(r-k)!}{k!} \binom{r-1}{k-1}^2
      \dupr[\big]{(D^\sym_{\antihol})^k \bigpr^*(f) \otimes (D^\sym_\hol)^k \bigpr^*(g)}{ (H_\Xi)^k }\at[\big]{\Levelset}
      \\
      &=
      fg\at[\big]{\Levelset}
      +
      \sum_{k=1}^\infty
      \sum_{s=0}^\infty
      \frac {1/\lambda} {\rising{1/\lambda}{k+s+1}}
      \frac{s!}{k!} \binom{k+s-1}{k-1}^2
      \dupr[\big]{(D^\sym_{\antihol})^k \bigpr^*(f) \otimes (D^\sym_\hol)^k \bigpr^*(g)}{ (H_\Xi)^k }\at[\big]{\Levelset}
      \\
      &=
      fg\at[\big]{\Levelset}
      +
      \sum_{k=1}^\infty
      \frac{1}{k!}
      \frac 1 {\falling{1/\lambda}{k}} 
      \dupr[\big]{(D^\sym_{\antihol})^k \bigpr^*(f) \otimes (D^\sym_\hol)^k \bigpr^*(g)}{ (H_\Xi)^k }\at[\big]{\Levelset}
      \\
      &=
      \sum_{k=0}^\infty
      \frac{1}{k!}
      \frac 1 {\falling{1/\lambda}{k}} 
      \dupr[\big]{(D^\sym_{\antihol})^k \bigpr^*(f) \otimes (D^\sym_\hol)^k \bigpr^*(g)}{ (H_\Xi)^k }\at[\big]{\Levelset}
      \punkt
    \end{align*}
    As $(H_\Xi)\at{\rho} \in \Xi_\rho \otimes \Xi_\rho$ for all $\rho \in \CCres$, we may insert 
    projections $\bigproj^*_{\Xi}$
    and get
    \begin{align*}
      \dupr[\big]{
        (D^\sym_{\antihol})^k \bigpr^*(f) \otimes &(D^\sym_\hol)^k \bigpr^*(g)
      }{ 
        (H_\Xi)^k 
      }
      =
      \\
      &=
      \dupr[\big]{
        (\bigproj^*_{\Xi})^{\otimes k} (D^\sym_{\antihol})^k \bigpr^*(f)
        \otimes
        (\bigproj^*_{\Xi})^{\otimes k} (D^\sym_\hol)^k \bigpr^*(g)
      }{
        (H_\Xi)^k 
      }
      \punkt
    \end{align*}
    Using Lemma~\ref{lemma:DofEStar} and Proposition~\ref{proposition:Dsymredhol} we obtain
    \begin{align*}
     (\bigproj^*_{\Xi})^{\otimes k} (D^\sym_\hol)^k \bigpr^*(g)
     =
     (\bigproj^*_{\Xi})^{\otimes k} D^\sym_\hol (\bigproj^*_{\Xi})^{\otimes (k-1)} D^\sym_\hol \dots \bigproj^*_{\Xi} 
     D^\sym_\hol \bigpr^*(g)
     =
     \bigpr^* \big( (D^\sym_{\red,\hol})^k g \big)
    \end{align*}
    and analogously for $f$, so that
    \begin{align*}
      \pr^*(f \starred g) 
      &=
      \sum_{k=0}^\infty
      \frac{1}{k!}
      \frac 1 {\falling{1/\lambda}{k}} 
      \dupr[\big]{(D^\sym_{\antihol})^k \bigpr^*(f) \otimes (D^\sym_\hol)^k \bigpr^*(g)}{ (H_\Xi)^k }\at[\big]{\Levelset}
      \\
      &=
      \sum_{k=0}^\infty
      \frac{1}{k!}
      \frac 1 {\falling{1/\lambda}{k}} 
      \dupr[\big]{\bigpr^*\big((D^\sym_{\red,\antihol})^k f\big) \otimes \bigpr^*\big((D^\sym_{\red,\hol})^k g\big)}{ (H_\Xi)^k }\at[\big]{\Levelset}
      \\
      &=
      \sum_{k=0}^\infty
      \frac{1}{k!}
      \frac 1 {\falling{1/\lambda}{k}} 
      \pr^*\big(\dupr[\big]{(D^\sym_{\red,\antihol})^k f \otimes (D^\sym_{\red,\hol})^k g}{ H_\red^k }\big)
      \\
      &=
      \sum_{k=0}^\infty
      \frac{1}{k!}
      \frac 1 {\falling{1/\lambda}{k}} 
      \pr^*\big(\dupr[\big]{(D^\sym_{\red})^k f \otimes (D^\sym_{\red})^k g}{ H_\red^k }\big)
      \punkt
    \end{align*}
    In the last step we used Proposition~\ref{proposition:holahol} and that the first tensor factor of $H_\red$ lies in 
    $\Tangent^{(0,1)} \Mred$ 
    whereas the second lies in $\Tangent^{(1,0)} \Mred$.
    
    It remains to check that only finitely many derivatives of $g$ contribute to $f \starred g$
    if $g \in \Polynomials(\Mred)$:
    Indeed, if $g = \MonomRed{P}{Q}$ with $P,Q \in \NN_0^{1+n}$ and $\abs{P} = \abs{Q}$, then
    $\bigpr^*(\MonomRed{P}{Q}) = {\momentmap}^{-\abs{P}} \Monom{P}{Q}$, and as
    $D^\sym_\hol \momentmap = \momentmap E^*$ is in the kernel of $\bigproj^*_{\Xi}$,
    we get
    \begin{align*}
      \bigpr^* \big( (D^\sym_{\red,\hol})^k \MonomRed{P}{Q} \big)
      &=
      (\bigproj^*_{\Xi})^{\otimes k} D^\sym_\hol (\bigproj^*_{\Xi})^{\otimes (k-1)} D^\sym_\hol \dots \bigproj^*_{\Xi} 
      D^\sym_\hol {\momentmap}^{-\abs{P}} \Monom{P}{Q} 
      \\
      &=
      {\momentmap}^{-\abs{P}} 
      (\bigproj^*_{\Xi})^{\otimes k} D^\sym_\hol (\bigproj^*_{\Xi})^{\otimes (k-1)} D^\sym_\hol \dots \bigproj^*_{\Xi} 
      D^\sym_\hol \Monom{P}{Q} 
      \\
      &=
      {\momentmap}^{-\abs{P}}  (\bigproj^*_{\Xi})^{\otimes k} (D^\sym_\hol)^k \Monom{P}{Q}
      \komma
    \end{align*}
    which vanishes for $k > \abs{P}$.
\end{Empty}
Finally, we can also characterize the reduced covariant derivative as follows:

\begin{proposition} \label{proposition:Dredchar}
  The reduced exterior covariant derivative $D_\red$ on $\Mred$ is the one
  for the Levi-Civita connection associated to the (not necessarily definite)
  reduced metric $g_\red \in \SymSec^2(\Mred)$, which is defined by
  \begin{align}
    \bigpr^*(g_\red) = (\bigproj_\Xi^*)^{\otimes 2} \bigg(\sum_{k=0}^n \frac{\nu_k \D\cc{z}^k \vee \D z^k}{ \momentmap } \at[\bigg]{\CCres}\bigg)
    \punkt
  \end{align}
\end{proposition}
\begin{proof}
  As $\sum_{k=0}^n \nu_k \D\cc{z}^k \vee \D z^k / \momentmap$ is $\CCx$-invariant,
  $g_\red$ is indeed well-defined. As $D$ is torsion-free, $D_\red$ is torsion free
  as well (see Proposition~\ref{proposition:covDer:reduction}).
  Now we calculate
  \begin{align*}
    \bigpr^*\big( D_\red ( g_\red ) \big)
    =
    (\bigproj^*_\Xi)^{\otimes 3} D \bigpr^*( g_\red )
    =
    (\bigproj^*_\Xi)^{\otimes 3} D (\bigproj^*_\Xi)^{\otimes 2} 
    \bigg( \sum_{k=0}^n \frac{\nu_k \D\cc{z}^k \vee \D z^k}{ \momentmap } \at[\bigg]{\CCres} \bigg)
    \punkt
  \end{align*}
  Using \eqref{eq:DzAsLinearCombOfWStar} one can check that
  $\sum_{k=0}^n \nu_k (\bigproj_\Xi^* \D\cc{z}^k) \vee (\bigproj_\Xi^*\D z^k)
  =
  \sum_{k=0}^n \nu_k \D\cc{z}^k \vee \D z^k - \momentmap \cc{E}^* \vee E^*$
  so that
  \begin{align*}
    \bigpr^*\big( D_\red ( g_\red ) \big)
    =
    (\bigproj^*_\Xi)^{\otimes 3} D
    \bigg( \sum_{k=0}^n \frac{\nu_k \D\cc{z}^k \vee \D z^k}{ \momentmap } \at[\bigg]{\CCres} \bigg)
    -
    (\bigproj^*_\Xi)^{\otimes 3} D \big( \cc{E}^* \vee E^* \big)
    =
    0
  \end{align*}
  because $\bigproj^*_\Xi \D (\momentmap[-1]) = - \momentmap[-2]\bigproj^*_\Xi \D \momentmap = 0$
  and because $(\bigproj^*_\Xi)^{\otimes 3} D ( \cc{E}^* \vee E^* ) = 0$, so $D_\red ( g_\red ) = 0$.
\end{proof}
Note that $g_\red$ can equivalently be obtained from the standard (pseudo-)metric
$g \coloneqq \sum_{k=0}^n \nu_k \D \cc{z}^k \vee \D z^k$ on $\CC^{1+n}$ in signature $s$
by first restricting $(\bigproj_\Xi^*)^{\otimes 2} g$ to $\Levelset$ and then projecting down on 
$\Mred$.
In local coordinates on the domain of definition of the forms $W_k^*$ one finds that
\begin{equation}
  (\bigproj^*_\Xi)^{\otimes 2}
  \bigg(
    \sum_{k=0}^n \frac{\nu_k \D\cc{z}^k \vee \D z^k}{ \momentmap } \at[\bigg]{\CCres} 
  \bigg)
  = 
  \frac{\abs{z^0}^2}{\momentmap} \sum_{k=1}^n \nu_k \cc{W}_k^* \vee W_k^*
  -
  \frac{\abs{z^0}^2}{{\momentmap}^2} \sum_{k,\ell=1}^n \nu_k\nu_\ell \cc{z}^\ell z^k \cc{W}_k^* \vee W_\ell^*
\end{equation}
and hence that
\begin{equation}
  g_\red
  =
  \frac{\sum_{k=1}^n \nu_k \D \cc{w}^k \vee \D w^k}{1 + \sum_{k=1}^n \nu_k \cc{w}^k w^k} 
  -
  \frac{ \sum_{k,\ell=1}^n \nu_k\nu_\ell \cc{w}^\ell w^k \D \cc{w}^k \vee \D w^\ell }{(1 + \sum_{k=1}^n \nu_k \cc{w}^k w^k)^2}
  \punkt
\end{equation}
In particular, in signature $s=1+n$ one obtains for $g_\red$ the usual Fubini--Study metric on 
$\Mred[{\signature[1+n]}] \cong \CC\PP^n$, and for $s=1$ the negative of the usual Fubini--Study metric 
on $\Mred[{\signature[1]}] \cong \DD^n$.




\subsection{The Polynomial Case}   \label{subsec:physics:polynomial}

In this section we will replace the formal parameter $\lambda$ by a complex number 
$\hbar$. In order to make sense of the convergence of the formal power series describing the
star product, we restrict ourselves to polynomial functions.

From the definition of the Poisson bracket in \eqref{eq:definition:poissonBracket} it is clear that it 
restricts to a well-defined map
$\poi\argument\argument \colon \Polynomials(\CC^{1+n}) \times \Polynomials(\CC^{1+n}) \to \Polynomials(\CC^{1+n})$ 
that is given by the same formula, and similarly for the Wick star product:
\begin{lemma} \label{lemma:starproductstructure}
	In the basis $\Monom{P}{Q}$ defined in Definition~\ref{definition:basis:polynomials}, the 
	Poisson bracket is
	\begin{equation} \label{eq:poissonBracket:onPolynomials}
	\poi[\big]{\Monom{P}{Q}}{\Monom{R}{S}} = \frac{1}{\I} \sum_{k=0}^n \nu_k (Q_k R_k - P_k S_k) \Monom{P+R-E_k}{Q+S-E_k}
	\end{equation}
	with $E_k = (0,\dots,0,1,0,\dots,0) \in \NN_0^{1+n}$ having the $1$ at position $k \in \{0,\dots,n\}$,
	and the	product $\star$ from Definition~\ref{definition:wick} is
	\begin{equation}
	\Monom{P}{Q} \star \Monom{R}{S} = \sum_{T=0}^{\min\{Q,R\}} \sgn(T) \lambda^{\abs T} T! 
	\binom Q T \binom R T \Monom{P+R-T}{Q+S-T}
	\end{equation}
	with $\sgn(T) = \prod_{k=0}^n \nu_k^{T_k} = \prod_{k=1}^n \nu_k^{T_k}$ like in Lemma~\ref{lemma:redexpansioninredfund}.
\end{lemma}
So by setting $\lambda$ to $\hbar \in \CC$, this yields a well-defined map
$\star_{\hbar} \colon \Polynomials(\CC^{1+n}) \times \Polynomials(\CC^{1+n}) \to \Polynomials(\CC^{1+n})$.
Next we consider the equivalence transformation $S$ from Proposition~\ref{proposition:Trafo2}:

\begin{lemma} \label{lemma:Sonpolys}
	For $P, Q \in \NN_0^{1+n}$ with $\abs P = \abs Q$, the equivalence transformation $S$ is given by
	\begin{equation} \label{eq:equivalenceTrafo:strict}
	S(\Monom{P}{Q})	= \bigg(\frac \lambda \momentmap \bigg)^{\abs P} \falling[\bigg]{\frac 
	\momentmap \lambda}{\abs P} 
	\Monom{P}{Q} \punkt
	\end{equation}
\end{lemma}

\begin{proof}
	As ${\momentmap}^{-\abs P} \Monom{P}{Q}$ is $\CCx$-invariant, we get using Proposition~\ref{proposition:Trafo1}:
	\begin{equation*} 
	S(\Monom{P}{Q}) = S({\momentmap}^{\abs P} {\momentmap}^{-\abs P} \Monom{P}{Q}) = S({\momentmap}^{\abs P}) 
	{\momentmap}^{-\abs P} \Monom{P}{Q} = \falling{\momentmap/\lambda}{\abs P} 
	(\lambda/\momentmap)^{\abs P} 
	\Monom{P}{Q} \punkt
	\end{equation*}
	Note that this is indeed a well-defined formal power series in $\lambda$ 
	as the term $\momentmap/\lambda$ that occurs in $\falling{\momentmap/\lambda}{\abs P}$
	if $\abs{P} \ge 1$ is cancelled.
\end{proof}
Replacing $\lambda$ by $\hbar \in \CC$ yields a rational expression in $\hbar$
and we have to be aware of some poles:

\begin{definition}
  We define the open subset $\StarDomain$ of $\CC$ as
  \begin{align}
    \StarDomain \coloneqq \CC \setminus \big( \set{1/k}{k\in \NN} \cup \{0\} \big)\punkt
  \end{align}
\end{definition}
We have already seen in Theorem~\ref{theorem:reducedWickProductFormula} that
the reduced Wick star product $f \starred g$ of polynomials $f,g \in \Polynomials(\Mred)$
is rational in $\hbar$ with poles in $\set{1/k}{k\in \NN}$. More precisely, we get:

\begin{proposition} \label{proposition:starproductstructurered}
	For $P, Q, R, S \in \NN_0^{1+n}$ with $\abs P = \abs Q$ and $\abs R = \abs S$, the
	reduced star product from Definition~\ref{definition:reducedStar} is given by
	\begin{equation} \label{eq:starProduct:strict:reduced}
	\MonomRed{P}{Q} \starred \MonomRed{R}{S} = \sum_{T=0}^{\min\{Q, R\}} \sgn(T) \frac{ 
		\falling{1/\lambda}{\abs{Q+R-T}}}{\falling{1 / \lambda}{\abs Q} 
		\falling{1/\lambda}{\abs R}} T! \binom Q T 
	\binom R T  
	\MonomRed{P+R-T}{Q+S-T} \punkt
	\end{equation}
	Replacing $\lambda$ by 
	$\hbar \in \StarDomain$
	yields a strict (associative) product
	$\starred[,\hbar] \colon 
	\Polynomials(\Mred) \times \Polynomials(\Mred) \to \Polynomials(\Mred)$
	and $\Polynomials(\Mred)$ with this product and pointwise complex conjugation
	becomes a unital $^*$-algebra if $\hbar \in \StarDomain \cap \RR$.
\end{proposition}
\begin{proof}
    First we note that  Lemmas~\ref{lemma:starproductstructure} and
    \ref{lemma:Sonpolys} show that
	\begin{align*}
        \bigg(\frac \lambda \momentmap \bigg)^{\abs Q} &\falling[\bigg]{\frac \momentmap \lambda}{\abs Q} \Monom{P}{Q}
        \startilde 
        \bigg(\frac \lambda \momentmap \bigg)^{\abs R} \falling[\bigg]{\frac \momentmap \lambda}{\abs R} \Monom{R}{S}
        =
        \\
        &=
        \sum_{T=0}^{\min\{Q, R\}} \sgn(T) \lambda^{\abs{T}} T!
        \binom Q T 
        \binom R T  
        \bigg(\frac \lambda \momentmap \bigg)^{\abs {Q+R-T}} \falling[\bigg]{\frac \momentmap \lambda}{\abs {Q+R-T}} 
        \Monom{P+R-T}{Q+S-T}
	\end{align*}
    holds for the transformed star product $\startilde$
    and all $P,Q,R,S\in \NN_0^{1+n}$ with $\abs{P} = \abs{Q}$ and $\abs{R} = \abs{S}$.
    As
    \begin{align*}
      \pr^* \big( \lambda^{\abs Q} \falling[\big]{1 / \lambda}{\abs Q} \MonomRed{P}{Q} \big)
      =
      \iota^*\Big( \big(\lambda / \momentmap \big)^{\abs Q} \falling[\big]{\momentmap / \lambda}{\abs Q} \Monom{P}{Q} \Big)
    \end{align*}
    we find that
	\begin{align*}
        \big( \lambda^{\abs Q} \falling[\big]{1 / \lambda}{\abs Q} &\MonomRed{P}{Q} \big)
        \starred
        \big( \lambda^{\abs R} \falling[\big]{1 / \lambda}{\abs R} \MonomRed{R}{S} \big)
        =
        \\
        &=
        \sum_{T=0}^{\min\{Q, R\}} \sgn(T) \lambda^{\abs{Q+R}} T!
        \binom Q T 
        \binom R T  
        \falling[\big]{1 / \lambda}{\abs {Q+R-T}}
        \MonomRed{P+R-T}{Q+S-T}
        \komma
	\end{align*}
    which yields \eqref{eq:starProduct:strict:reduced} by $\CC\formal{\lambda}$-linearity of $\starred$
    and because $\lambda^{\abs{Q}} \falling{1/\lambda}{\abs{Q}} = (1-\lambda) \dots (1-(\abs{Q}-1)\lambda)$
    is an invertible formal power series.
    Note that the right-hand side of \eqref{eq:starProduct:strict:reduced}
    is indeed a well-defined formal power series in $\lambda$ because the factor
    $1/\lambda$ that occurs in $\falling{1 / \lambda}{\abs {Q+R-T}}$
    for $\abs {Q+R-T} \ge 1$ is cancelled.
    
    We can now substitute $\lambda$ by $\hbar \in \CC$:
	If $\hbar \in \StarDomain$, the falling factorials in the nominator are non-zero, thus 
	\eqref{eq:starProduct:strict:reduced} defines a well-defined product on the whole algebra 
	$\Polynomials(\Mred)$.
	Associativity and compatibility with pointwise complex conjugation follow from the
	properties of the Hermitian formal star product $\starred$, and the unit is the
	constant $1$-function.
\end{proof}
Equation \eqref{eq:starProduct:strict:reduced} immediately yields:

\begin{corollary} \label{corollary:rationalDependenceOnHbar}
	For two fixed polynomials $f, g \in \Polynomials(\Mred)$, the map
	$\hbar \mapsto f \starred[,\hbar] g$ is rational
	and $\lim_{\hbar\to 0} f \starred[,\hbar] g = f g$ holds pointwise.
\end{corollary}

\begin{proposition} \label{proposition:poissonBracketIsFirstOrderCommutator}
	For two polynomials $f, g \in \Polynomials(\Mred)$, we have
	\begin{equation} \label{eq:poissonBracketIsFirstOrderCommutator}
	\lim_{\hbar \to 0} \frac 1 {\I\hbar} \big(f \starred[,\hbar] g - g 
	\starred[,\hbar] f\big) 
	= \poi{f}{g}_\red
	\end{equation}
	pointwise and with the reduced Poisson bracket $\poi{\argument}{\argument}_\red$ on $\Mred$.
\end{proposition}
\begin{proof}
	All terms with $\abs T \geq 2$ in Equation \eqref{eq:starProduct:strict:reduced}
	are at least of order $\hbar^2$ and the $T = 0$ term cancels out when
	taking the commutator. The first order in $\hbar$ of the terms with 
	$\abs T = 1$ produces
	\begin{align*}
	\lim_{\hbar \to 0} \frac 1 {\I \hbar}
	\big(\MonomRed{P}{Q} \starred[,\hbar] \MonomRed{R}{S} - \MonomRed{R}{S} \starred[,\hbar] \MonomRed{P}{Q}\big)
	=
	\frac{1}{\I}\sum_{k=0}^n \nu_k 
    (Q_k R_k - S_kP_k)
	\MonomRed{P+R-E_k}{Q+S-E_k}
	\end{align*}
    with $E_k = (0,\dots,0,1,0,\dots,0) \in \NN_0^{1+n}$ having the $1$ at position $k\in \{0,\dots,n\}$,
    which coincides by Definition~\ref{definition:reducible}
    and Lemma~\ref{lemma:starproductstructure} with $\poi{\MonomRed{P}{Q}}{\MonomRed{R}{S}}_\red$.
\end{proof}




\subsection{The Analytic Case}   \label{subsec:physics:analytic}

The aim of this section is to obtain a strict star product on the algebra $\Analytic(\Mred)$. We achieve this 
by proving the continuity of the star product $\starred[,\hbar]$ on $\Polynomials(\Mred)$ 
with respect to the locally convex topology that $\Polynomials(\Mred)$ inherits
from $\Analytic(\Mred)$, i.e.\ the topology of locally uniform convergence
of the holomorphic extensions to $\MredExt$. This then implies that $\starred[,\hbar]$ extends 
uniquely to a continuous star product on $\Analytic(\Mred)$.

Recall from Proposition~\ref{proposition:topologies:quotientTopology} that the topology on $\Analytic(\Mred)$
is just the quotient topology of the topology on $\Analytic(\CC^{1+n})^{\group U(1)}$
defined by locally uniform convergence of the holomorphic extensions to $\CC^{1+n}\times\CC^{1+n}$.
For 
$\hbar \in \StarDomain$ define a product 
$*_\hbar$ on $\Polynomials(\CC^{1+n})^{\group{U}(1)}$ by bilinearly extending
\begin{equation} 
\Monom{P}{Q} *_\hbar \Monom{R}{S} 
\coloneqq 
\sum_{T=0}^{\min\{Q, R\}} \sgn(T) 
\frac{ 
    \falling{1/\hbar}{\abs Q+R-T}
}{
    \falling{1 / \hbar}{\abs Q} \falling{1/\hbar}{\abs R}
}
T! \binom Q T \binom R T \Monom{P+R-T}{Q+S-T}
\end{equation}
for all $P,Q,R,S\in \NN_0^{1+n}$ with $\abs{P} = \abs{Q}$ and $\abs{R} = \abs{S}$.
Note that this product might be badly behaved, for example it does not need to be associative. However, from 
Proposition~\ref{proposition:starproductstructurered} it follows immediately that
dividing out the vanishing ideal of $\momentmap - 1$ is possible and reproduces the product 
$\starred[,\hbar]$.
Consequently, continuity of $*_\hbar$ with respect to the seminorms 
$\seminorm{r}{\argument}$ with $r \in {[1,\infty[}$ defined in Equation 
\eqref{eq:normoben} implies continuity of $\starred[,\hbar]$.

\begin{proposition} \label{proposition:continuityEstimates}
	The product $*_\hbar$ is continuous with respect to the locally convex topology 
	defined by the seminorms $\seminorm{r}{\argument}$ with $r \in {[1,\infty[}$ 
	as in Equation \eqref{eq:normoben}. More precisely, for every $r \in {[1,\infty[}$
	and every compact subset $K$ of $\StarDomain$ there exists $r' \in {[1,\infty[}$ such that
	\begin{align}
	 \seminorm[\big]{r}{f *_\hbar g}
	 \le
	 \seminorm{r'}{f} \seminorm{r'}{g}
	\end{align}
    holds for all $\hbar \in K$ and all $f,g\in \Polynomials(\CC^{1+n})^{\group{U}(1)}$.
\end{proposition}
\begin{proof}
	It is well-known that for any compact set $K' \subseteq \CC \setminus \NN_0$ there are 
	constants $c, C > 0$ such that
	\begin{equation*}
	c^n n! \leq \abs{\falling{z}{n}} \leq C^n n!
	\end{equation*}
	holds for all $z \in K$ and all $n \in \NN_0$.
	For a compact set $K \subseteq \StarDomain$ also
	$K' \coloneqq \set{z \in \CC \setminus \{ 0 \}}{z^{-1} \in K}$ is compact
	and a subset of $\CC \setminus \NN_0$. Therefore it follows for any 
	$r \in {[1,\infty[}$ and $P, Q, R, S \in \NN_0^{1+n}$ with $\abs P = \abs Q$ and $\abs R = \abs S$, and 
	assuming without loss of generality that $C \geq 1$,
	that 
	\begin{align*}
	\seminorm[\big]{r}{\Monom{P}{Q} *_\hbar \Monom{R}{S}} 
	& = 
	\seminorm[\bigg]{r}{
		\sum_{T=0}^{\min\{Q, R\}} \sgn(T) 
			\frac{\falling{1/\hbar}{\abs{Q+R-T}}}
				 {\falling{1 / \hbar}{\abs Q} \falling{1/\hbar}{\abs R}}
			T! \binom Q T \binom R T \Monom{P+R-T}{Q+S-T}
	} 
	\\
	& \leq 
	\sum_{T=0}^{\min\{Q, R\}} 
		\abs[\bigg]{
			\frac{\falling{1/\hbar}{\abs{Q+R-T}}}
				 {\falling{1 / \hbar}{\abs Q} \falling{1/\hbar}{\abs R}}
        }
		T! \binom Q T \binom R T
		r^{\abs{P+Q+R+S-2T}} 
	\\
	& \leq 
	\sum_{T=0}^{\min\{Q, R\}}
        \frac{C^{\abs{Q+R-T}}}{c^{\abs{Q+R}}}
        \frac{\abs{Q+R-T}! T!}{\abs Q! \abs R!}
		2^{\abs{Q+R}}
		r^{\abs{P+Q+R+S-2T}}
	\\
	& \leq 
	\sum_{T=0}^{\min\{Q, R\}} (c^{-1} C)^{\abs {Q + R}} 4^{\abs {Q+R}}
	r^{\abs {P + Q + R + S}}
	\\
	& \leq 
	(8 c^{-1} C r)^{\abs {P + Q + R + S}}
	\punkt
	\end{align*}
	So given $\group{U}(1)$-invariant polynomials 
	$f = \sum_{P,Q} \expansionCoefficients{f}{P}{Q} \Monom{P}{Q}$ and
	$g = \sum_{R,S} \expansionCoefficients{g}{R}{S} \Monom{R}{S}$ on $\CC^{1+n}$
	with complex coefficients $\expansionCoefficients{f}{P}{Q}$ and 
	$\expansionCoefficients{g}{R}{S}$, then
	\begin{align*}
	 \seminorm[\big]{r}{f *_\hbar g}
	 \le
	 \sum_{\substack{P,Q \in \NN_0^{1+n} \\ \abs{P}=\abs{Q}}}
	 \abs[\big]{\expansionCoefficients{f}{P}{Q}}
	 (8 c^{-1} C r)^{\abs {P + Q}}
	 \sum_{\substack{R,S \in \NN_0^{1+n} \\ \abs{R}=\abs{S}}}
	 \abs[\big]{\expansionCoefficients{g}{R}{S}}
	 (8 c^{-1} C r)^{\abs {R + S}}
	 =
	 \seminorm[\big]{8 c^{-1} C r}{f} \seminorm[\big]{8 c^{-1} C r}{g}
	 \punkt\,\,
	\end{align*}
\end{proof}
We would like to remark that, similar as in 
\cite{esposito.schmitt.waldmann:OnlineComparisonContinuityWickStarProducts},
one can also use the description of the star product using bidifferential operators to prove its continuity.

\begin{theorem} \label{theorem:strictProduct}
  For every $\hbar \in \StarDomain$, the product
  $\starred[,\hbar]$ on $\Polynomials(\Mred)$ extends to a continuous
  associative product on $\Analytic(\Mred)$. Moreover $\Analytic(\Mred)$ becomes a unital
  Fréchet-$^*$-algebra with this product and pointwise complex conjugation as $^*$-involution
  in the case that $\hbar \in \StarDomain \cap \RR$.
  Finally, for any two fixed elements $f, g \in \Analytic(\Mred)$ and $[\rho] \in \Mred$,
  the map $\StarDomain \to \CC$, $\hbar \mapsto (f \starred[,\hbar] g)([\rho])$
  is holomorphic.
\end{theorem}

\begin{proof}
    By the previous Proposition~\ref{proposition:continuityEstimates} and the discussion above, 
    the associative product $\starred[,\hbar]$ is continuous on $\Polynomials(\Mred)$
    with respect to the topology inherited from $\Analytic(\Mred)$, and thus extends
    to an associative and continuous product on $\Analytic(\Mred)$ because $\Polynomials(\Mred)$
    is dense in $\Analytic(\Mred)$ by Corollary~\ref{corollary:polynomialsdense}. The constant
    $1$-function remains the unit like on polynomials.
    Compatibility with the $^*$-involution is clear as well if $\hbar$ is additionally real.
    
    Now recall that for polynomials $p, q \in \Polynomials(\Mred)$, the map
	$\hbar \mapsto (p \starred[,\hbar] q)([\rho])$
	is rational by Corollary~\ref{corollary:rationalDependenceOnHbar}. Since the estimates in
	Proposition~\ref{proposition:continuityEstimates} are locally uniform in $\hbar$, it follows that
	$\hbar \mapsto (f \starred[,\hbar] g)([\rho])$ is
	a locally uniform limit of rational functions and therefore holomorphic.
\end{proof}
Note that $0 \notin \StarDomain$, so one would like to understand whether in the limit
$\hbar\to 0$, the product $\starred[,\hbar]$ yields the
pointwise one, and whether its commutator yields the Poisson bracket
also on $\Analytic(\Mred)$.
Despite the results from Corollary~\ref{corollary:rationalDependenceOnHbar} and
Proposition~\ref{proposition:poissonBracketIsFirstOrderCommutator} in the polynomial
case, this is not so obvious because $0$ is an accumulation point of the
poles of $\starred[,\hbar]$. We will come back to this question later in
Proposition~\ref{proposition:poissonBracketIsFirstOrderCommutator:Analytic}.




\section{Wick Rotation}   \label{sec:wick}

The dependence on the choice of signature $s$ will now always be
made explicit by a superscript ``$^{\signature}$''.

We have already seen that the construction of the formal and 
non-formal star products on $\Mred[\signature]$ works completely independent of
$s \in \{1,\dots,1+n\}$. We will see now that the non-formal star product algebras are even
all isomorphic as unital complex algebras. This will be proven
by construction of a Wick transformation: A holomorphic isomorphism between
the complex manifolds $\MredExt[\signature]$ for different values of $s$ which gives 
rise to isomorphisms of the algebras $\Polynomials(\Mred[\signature])$ and
$\Analytic(\Mred[\signature])$ (with the pointwise product) and which are also 
compatible with the Poisson brackets and the non-formal star products, i.e.\ describe
isomorphisms of Poisson algebras and associative algebras, respectively.
However, we will also see that these isomorphisms are not compatible with the
$^*$-involution which is given by pointwise complex conjugation, hence are
not $^*$-isomorphisms. This demonstrates how important it is to consider $^*$-algebras
and not just algebras in non-formal deformation quantization: After all, one would
surely want to be able to distinguish the quantization of the complex projective space
$\CC\PP^n$ from the one of the hyperbolic disc $\DD^n$.




\subsection{Geometric Wick Rotation}   \label{subsec:wickgeo}

We start first with discussing the complex manifolds $\CC^{1+n}\times\CC^{1+n}$ and then
proceed to $\MredExt[\signature]$: 

The Lie group $\GL\times\GL$ acts on $\CC^{1+n}\times\CC^{1+n}$
from the left via $\argument \acts \argument$, which induces a right action $\argument \racts \argument$
on $\Holomorphic(\CC^{1+n}\times\CC^{1+n})$. It is easy to check that
$\Diag( A\acts \rho ) = (A,\cc{A}) \acts \Diag(\rho)$ for all $\rho \in \CC^{1+n}$
and all $A\in \GL$, where $\cc{A}$ denotes the elementwise complex conjugate of $A$.
For all $s\in \{1,\dots,1+n\}$, let $W^{\signature} \in \GL$ be
${(W^{\signature})^k}_\ell \coloneqq 1$ if $k=\ell\in\{0,\dots,s-1\}$ and
${(W^{\signature})^k}_\ell \coloneqq \I$ if $k=\ell\in\{s,\dots,n\}$, and otherwise
${(W^{\signature})^k}_\ell \coloneqq 0$. Then the action of $(W^{\signature},W^{\signature})$
on $\CC^{1+n}\times\CC^{1+n}$ does not come from an action on $\CC^{1+n}$, except in the trivial
case that $s=1+n$. However, the identity
\begin{equation}
\momentmapExt[\signature] \racts (W^{\signature},W^{\signature}) 
=
\momentmapExt[{\signature[1+n]}]
\label{eq:momentmapExtWicked}
\end{equation}
holds and thus the holomorphic automorphism of $\CC^{1+n}\times\CC^{1+n}$
that is given by the action of $(W^{\signature},W^{\signature})$ restricts to a holomorphic 
isomorphism from $\LevelsetExt[{\signature[1+n]}]$ to $\LevelsetExt[\signature]$.
It is then immediate that this
restriction even descends to a holomorphic isomorphism from
$\MredExt[{\signature[1+n]}]$ to $\MredExt[\signature]$, because
$(W^{\signature},W^{\signature})$ commutes with all elements of the
$\CCx$-subgroup of $\GL\times\GL$.

\begin{definition}
  For every $s \in \{1,\dots,1+n\}$ we define the map
  $\geomWickRotRed[\signature] \colon \MredExt[{\signature[1+n]}] \to \MredExt[\signature]$,
  \begin{align}
    [(\xi,\eta)] \mapsto \geomWickRotRed[\signature]\big( [(\xi,\eta)] \big) \coloneqq \big[(W^{\signature},W^{\signature}) \acts (\xi,\eta) \big]
    \punkt
  \end{align}
\end{definition}
The above discussion shows that $\geomWickRotRed[\signature]$ is well-defined and even more:

\begin{proposition} \label{proposition:geometricWickRotation}
  The maps $\geomWickRotRed[\signature] \colon \MredExt[{\signature[1+n]}] \to \MredExt[\signature]$
  are holomorphic isomorphisms of complex manifolds  for all $s\in \{1,\dots,1+n\}$.
\end{proposition}
Moreover, Equation~\eqref{eq:momentmapExtWicked} also shows that the inner automorphism 
of the Lie group $\GL \times \GL$ that is given by conjugation with $(W^{\signature},W^{\signature})$,
i.e.
\begin{equation}
(A,B) \mapsto \big(
	W^{\signature} A (W^{\signature})^{-1}, W^{\signature} B (W^{\signature})^{-1}
\big) \label{eq:GLautomorphism}
\komma
\end{equation}
restricts to an isomorphism from $\StabExt[{\signature[1+n]}]$ to $\StabExt[\signature]$.
Note that we have already seen in Section~\ref{sec:geometry} that $\StabExt[\signature]$ is 
isomorphic to $\GL$ for all $s\in \{1,\dots,1+n\}$.

As a final remark, we note that the isomorphisms of $\MredExt[\signature]$ with different
signature $s$ clearly do not descend to isomorphisms of $\Mred[\signature]$. For example, 
$\Mred[{\signature[1+n]}] \cong \CC\PP^n$ is compact while $\Mred[{\signature[1]}] \cong \DD^n$
is not.




\subsection{Algebraic Wick Rotation}   \label{subsec:wickana}

The isomorphism of the complex manifolds $\MredExt[\signature]$ for different signatures
from Proposition~\ref{proposition:geometricWickRotation} immediately shows that the corresponding unital 
associative algebras $\Holomorphic(\MredExt[\signature])$ are also isomorphic. By
Proposition~\ref{proposition:analyticholomorphiciso}, the algebras $\Analytic(\MredExt[\signature])$
for different signatures are isomorphic as unital associative algebras as well
(but not necessarily as $^*$-algebras).

\begin{definition}
  For every $s \in \{1,\dots,1+n\}$ we define the maps
  $\WickRot[\signature] \colon \Analytic(\CC^{1+n}) \to \Analytic(\CC^{1+n})$,
  \begin{align}
    f \mapsto \WickRot[\signature](f) \coloneqq \Diag^*\big( \hat{f} \racts (W^{\signature},W^{\signature}) \big)
  \end{align}
  as well as
  $\WickRotRed[\signature] \colon \Analytic(\Mred[\signature]) \to \Analytic(\Mred[{\signature[1+n]}])$,
  \begin{align}
    g &\mapsto \WickRotRed[\signature](g) \coloneqq \big(\DiagM^{\signature[1+n]}\big)^*( \hat{g} \circ \geomWickRotRed[\signature] )
    \komma
  \end{align}
  where $\hat{f} \in \Holomorphic(\CC^{1+n}\times \CC^{1+n})$ and $\hat{g} \in \Holomorphic(\MredExt[\signature])$
  are such that $\Diag^*(\hat{f}) = f$ and $(\DiagM^{\signature})^*(\hat{g}) = g$. We will refer to 
  $\WickRot[\signature]$ and $\WickRotRed[\signature]$ as the \neu{Wick rotation} and the
  \neu{reduced Wick rotation}, respectively.
\end{definition}
Proposition~\ref{proposition:analyticholomorphiciso} and Proposition~\ref{proposition:geometricWickRotation}
together with the observation that $(W^{\signature},W^{\signature})$ commutes with the whole $\CCx$-subgroup
of $\GL\times\GL$ immediately shows:

\begin{theorem} \label{theorem:wickRotation}
  The Wick rotation $\WickRot[\signature]$ is a well-defined homeomorphic automorphism of the 
  unital associative Fréchet algebra $\Analytic(\CC^{1+n})$ that restricts to an automorphism
  of $\Analytic(\CC^{1+n})^{\group{U}(1)}$, and the reduced Wick rotation
  $\WickRotRed[\signature] \colon \Analytic(\Mred[\signature]) \to \Analytic(\Mred[{\signature[1+n]}])$
  is a well-defined homeomorphic isomorphism of unital associative Fréchet algebras.
\end{theorem}
The Wick rotations are also compatible with the reduction procedure:

\begin{proposition} \label{proposition:wickRotationLift}
  Given $f \in \Analytic(\CC^{1+n})^{\group{U}(1)}$, then
  \begin{align*}
    \big(\WickRot[\signature](f)\big)_\red = \WickRotRed[\signature](f_\red)
    \punkt
  \end{align*}
\end{proposition}
\begin{proof}
  By Proposition~\ref{proposition:analyticred} there exists an $\hat{f} \in \Holomorphic(\CC^{1+n}\times 
  \CC^{1+n})^{\CCx}$
  like in Lemma~\ref{lemma:analyticred} such that $\Diag^*(\hat{f}) = f$
  and $(\DiagM^{\signature})^*(\hat{f}_{\hat{\red}}) = f_\red$.
  Using the commutativity of the diagram from Section~\ref{sec:geometry} and the properties 
  of the action of $(W^{\signature},W^{\signature})$ one finds:
  \begin{align*}
    (\iota^{\signature[1+n]})^*\big(\WickRot[\signature](f)\big)
    &=
    (\iota^{\signature[1+n]})^*\big(\Diag^* \big(\hat{f}\racts (W^{\signature},W^{\signature})\big) \big)
    \\
    &=
    (\DiagZ^{\signature[1+n]})^*\big((\hat{\iota}{}^{\signature[1+n]})^* \big(\hat{f}\racts (W^{\signature},W^{\signature})\big) \big)
    \\
    &=
    (\DiagZ^{\signature[1+n]})^*\big((\hat{\iota}{}^{\signature})^* (\hat{f}) \racts (W^{\signature},W^{\signature})\big)    
    \\
    &=
    (\DiagZ^{\signature[1+n]})^*\big((\hat{\pr}{}^{\signature})^* (\hat{f}_{\hat{\red}}) \racts (W^{\signature},W^{\signature})\big)    
    \\
    &=
    (\DiagZ^{\signature[1+n]})^*\big((\hat{\pr}{}^{\signature[1+n]})^* \big(\hat{f}_{\hat{\red}} \circ \geomWickRotRed[\signature]\big) \big)
    \\
    &=
    (\pr^{\signature[1+n]})^* \big((\DiagM^{\signature[1+n]})^*\big(\hat{f}_{\hat{\red}} \circ \geomWickRotRed[\signature]\big) \big)
    \\
    &=
    (\pr^{\signature[1+n]})^* \big(\WickRotRed[\signature](f_\red) \big) \punkt
  \end{align*}
\end{proof}
In the following we will see that the Wick rotations are not only isomorphisms of unital associative algebras,
but also compatible with Poisson brackets and star products:

\begin{lemma} \label{lemma:wickrotPol}
  Given $s\in \{1,\dots,1+n\}$, then the identity
  \begin{align}
    \WickRot[\signature]( \Monom{P}{Q} ) &= \I^{\sum_{k=s}^{n} (P_k + Q_k)} \Monom{P}{Q} \label{eq:wickrotPol1}
  \intertext{holds for all $P,Q\in \NN_0^{1+n}$,}
    \WickRotRed[\signature]\big( \MonomRed[\signature]{P}{Q} \big) &= \I^{\sum_{k=s}^{n} (P_k + Q_k)} \MonomRed[{\signature[1+n]}]{P}{Q} \label{eq:wickrotPol2}
  \intertext{holds for all $P,Q\in \NN_0^{1+n}$ with $\abs{P}=\abs{Q}$, and}
    \WickRotRed[\signature]\big( \MonomRedFund[\signature]{P}{Q} \big) &= \I^{\sum_{k=s}^{n} (P_k + Q_k)} \MonomRedFund[{\signature[1+n]}]{P}{Q} \label{eq:wickrotPol3}
  \end{align}
  holds for all $P,Q\in \NN_0^n$, where $\MonomRedFund{P}{Q}$ are the fundamental monomials
  from Definition \ref{definition:fundamentalMonomials}.
  Moreover, $\WickRot[\signature]$ restricts to an automorphism of the unital subalgebra
  $\Polynomials(\CC^{1+n})$ of $\Analytic(\CC^{1+n})$, and $\WickRotRed[\signature]$ restricts to an isomorphism from
  the unital subalgebra $\Polynomials(\Mred[\signature])$ of $\Analytic(\Mred[\signature])$ to the unital subalgebra
  $\Polynomials(\Mred[{\signature[1+n]}])$ of $\Analytic(\Mred[{\signature[1+n]}])$.
\end{lemma}
\begin{proof}
  Using $\Monom{P}{Q} = \Diag^*( x^P y^Q )$ with $x^0,\dots,x^n,y^0,\dots,y^n \colon \CC^{1+n}\times\CC^{1+n} \to \CC$
  the standard coordinates, it is easy to check that Equation~\eqref{eq:wickrotPol1} holds. 
  Equation~\eqref{eq:wickrotPol2} then follows by applying the previous Proposition~\ref{proposition:wickRotationLift},
  which gives Equation~\eqref{eq:wickrotPol3} as a special case. The rest is clear.
\end{proof}

\begin{theorem} \label{theorem:wickstar}
  The Wick rotations remain isomorphisms of unital associative algebras also for the deformed products.
  More precisely, given $s\in \{1,\dots,1+n\}$, then the identities
  \begin{align}
    \WickRot[\signature]\big( f \star^{\signature}_\hbar g \big)
    &=
    \WickRot[\signature](f) \star^{\signature[1+n]}_\hbar \WickRot[\signature](g)
  \shortintertext{and}
    \WickRot[\signature]\poi[\big]{f}{g}^{\signature}
    &=
    \poi[\big]{\WickRot[\signature](f)}{\WickRot[\signature](g)}^{\signature[1+n]}
  \intertext{hold for all $f,g \in \Analytic(\CC^{1+n})$ and all $\hbar \in \CC$. Similarly, the identities}
    \WickRotRed[\signature]\big( f \star^{\signature}_{\red,\hbar} g \big)
    &=
    \WickRotRed[\signature](f) \star^{\signature[1+n]}_{\red,\hbar} \WickRotRed[\signature](g)
  \shortintertext{and}
    \WickRotRed[\signature]\poi[\big]{f}{g}^{\signature}_\red
    &=
    \poi[\big]{\WickRotRed[\signature](f)}{\WickRotRed[\signature](g)}^{\signature[1+n]}_\red
  \end{align}
  hold for all $f,g \in \Analytic(\Mred[\signature])$ and all $\hbar \in \StarDomain$.
\end{theorem}

\begin{proof}
  First note that as a consequence of the previous Lemma~\ref{lemma:wickrotPol}, the identity
  \begin{align*}
    \WickRot[\signature]\big( \sgn^{\signature}(T) \Monom{P+R-T}{Q+S-T} \big)
    =
    \I^{\sum_{k=s}^n (P_k+Q_k+R_k+S_k)} \Monom{P+R-T}{Q+S-T}
  \end{align*}
  holds for all $P,Q,R,S,T \in \NN_0^{1+n}$ with $T \le \min \{Q,R\}$, and similarly,
  \begin{align*}
    \WickRotRed[\signature]\big( \sgn^{\signature}(T) \MonomRed[\signature]{P+R-T}{Q+S-T} \big)
    =
    \I^{\sum_{k=s}^n (P_k+Q_k+R_k+S_k)} \MonomRed[{\signature[1+n]}]{P+R-T}{Q+S-T}
  \end{align*}
  holds for all $P,Q,R,S,T \in \NN_0^{1+n}$ with $\abs{P} = \abs{Q}$, $\abs{R} = \abs{S}$ and $T \le \min 
  \{Q,R\}$.
  Using this and the explicit formulas from Lemma~\ref{lemma:starproductstructure} and
  Proposition~\ref{proposition:starproductstructurered} and noting that $\sgn^{\signature[1+n]}(T) = 1$
  for all $T\in \NN_0^{1+n}$, it is easy to check the identities for the
  star products in the special case that $f$ and $g$ are monomials. The identities for the
  Poisson brackets are an immediate consequence thereof due to the representation of the
  Poisson brackets as a limit of the star product commutator like in 
  Proposition~\ref{proposition:poissonBracketIsFirstOrderCommutator}.
  The general case then follows by bilinearity and continuity of the star product and the Poisson bracket.
\end{proof}
While it is completely clear that the Wick rotations do not commute with the $^*$-involution given 
by pointwise complex conjugation, it is somewhat harder to show that the algebras 
$\Analytic(\Mred[{\signature[s]}])$ with product $\starreds{,\hbar}{\signature[s]}$ 
are in general not $^*$-isomorphic, not even via some other isomorphism. One possibility
to prove this is to examine their \neu{positive linear functionals}:
A linear functional $\phi \colon \Analytic(\Mred[{\signature[s]}]) \to \CC$ is called
\neu{positive} for the product $\starreds{,\hbar}{\signature}$ with 
$\hbar \in \StarDomain \cap \RR$ if
\begin{align}
 \phi \big( \cc{f} \starreds{,\hbar}{\signature} f\big) \ge 0
\end{align}
holds for all $f\in\Analytic(\Mred[{\signature[s]}])$. It is easy to see that the pullback
of a positive linear functional with a $^*$-homomorphism between two $^*$-algebras yields
again a positive linear functional. In the special case of $s = 1$, i.e.\ $\Mred[{\signature[1]}] \cong 
\DD^n$,
the existence of non-trivial positive linear functionals for negative $\hbar$ is known:

\begin{proposition} \label{proposition:evfunctionalsdisc}
  The evaluation functionals $\delta_{[\rho]}^{\signature[1]} \colon \Analytic(\DD^n) \to \CC$,
  \begin{align*}
    f \mapsto \delta_{[\rho]}^{\signature[1]}(f) \coloneqq f([\rho])
  \end{align*}
  with $[\rho]\in\Mred[{\signature[1]}] \cong \DD^n$
  are positive linear functionals on $\Analytic(\DD^n)$
  with product $\starreds{,\hbar}{\signature[1]}$ for all $\hbar \in {]{-\infty},0[}$.
\end{proposition}

\begin{proof}
	Positivity of evaluation functionals has been proven in 
	\cite[Sec.~5.4]{beiser.waldmann:FrechetAlgebraicDeformationOfThePoincareDisc}
	on an algebra containing (at least) $\Polynomials(\DD^n)$
	with a product $*_{\hbar}$ fulfilling 
	$f \starreds{,\hbar}{\signature[1]} g = g *_{-\hbar/2} f$
	for all $f,g \in \Polynomials(\DD^n)$. By continuity
	of the evaluation functionals, the pointwise complex conjugation and the product 
	$\starreds{,\hbar}{\signature[1]}$, this extends to whole
	$\Analytic(\DD^n)$.
\end{proof}
However, there are some limitations to the existence
of positive linear functionals in the special case of $s = 1+n$,
i.e.\ $\Mred[{\signature[1+n]}] \cong \CC\PP^n$, at least if $n=1$:

\begin{lemma} \label{lemma:negativesqaure}
 Consider only the case $n=1$ and $s=1+n = 2$. Then the identity
  \begin{align}
    \sum_{i,j=0}^1
    \MonomRed[{\signature[2]}]{E_j}{E_i}
    \starreds{,\hbar}{{\signature[2]}}
    \MonomRed[{\signature[2]}]{E_i}{E_j} 
    =
    1+\hbar
    \label{eq:negativesqaure}
  \end{align}
  holds for all $\hbar \in \StarDomain$,
  where $E_0 = (1,0) \in \NN_0^{1+1}$ and $E_1 = (0,1) \in \NN_0^{1+1}$.
\end{lemma}

\begin{proof}
  Proposition~\ref{proposition:starproductstructurered} yields
  \begin{align*}
    \MonomRed[{\signature[2]}]{E_j}{E_i}
    \starreds{,\hbar}{{\signature[2]}}
    \MonomRed[{\signature[2]}]{E_i}{E_j} 
    &=
    \frac{\falling{1/\hbar}{2}}{\falling{1/\hbar}{1}\falling{1/\hbar}{1}}
    \MonomRed[{\signature[2]}]{E_i+E_j}{E_i+E_j}
    +
    \frac{\falling{1/\hbar}{1}}{\falling{1/\hbar}{1}\falling{1/\hbar}{1}}
    \MonomRed[{\signature[2]}]{E_j}{E_j}
    \\
    &=
    (1-\hbar)
    \MonomRed[{\signature[2]}]{E_i+E_j}{E_i+E_j}
    +
    \hbar
    \MonomRed[{\signature[2]}]{E_j}{E_j}
  \end{align*}  
  for all $i,j \in \{0,1\}$. By summation over $i$ and $j$ we get
  \begin{align*}
    \sum_{i,j=0}^1
    &\MonomRed[{\signature[2]}]{E_j}{E_i}
    \starreds{,\hbar}{{\signature[2]}}
    \MonomRed[{\signature[2]}]{E_i}{E_j} 
    =
    \\
    &(1-\hbar)
    \big(
      \MonomRed[{\signature[2]}]{2E_0}{2E_0}
      +
      2
      \MonomRed[{\signature[2]}]{E_0+E_1}{E_0+E_1}
      +
      \MonomRed[{\signature[2]}]{2E_1}{2E_1}
    \big)
    +
    2 \hbar
    \big(
      \MonomRed[{\signature[2]}]{E_0}{E_0}
      +
      \MonomRed[{\signature[2]}]{E_1}{E_1}
    \big)
  \punkt
  \end{align*}
  Keeping in mind that the reduced monomials are not linearly independent,
  this can be simplified: We find that
  $\MonomRed[{\signature[2]}]{E_0}{E_0} + \MonomRed[{\signature[2]}]{E_1}{E_1}
  = 
  \momentmap[{\signature[2]}]_\red$
  is the constant $1$-function, and the same is true for their pointwise
  square
  $(\MonomRed[{\signature[2]}]{E_0}{E_0} + \MonomRed[{\signature[2]}]{E_1}{E_1})^2
  =
  \MonomRed[{\signature[2]}]{2E_0}{2E_0}
  +
  2
  \MonomRed[{\signature[2]}]{E_0+E_1}{E_0+E_1}
  +
  \MonomRed[{\signature[2]}]{2E_1}{2E_1}
  $.
\end{proof}

\begin{proposition}
  Consider only the case $n=1$ and $\hbar \in {]{-\infty},{-1}[}$.
  For signature $s = 2$,
  the only linear functional
  $\phi \colon \Analytic(\CC\PP^1) \to \CC$, which is positive
  for the product $\starreds{,\hbar}{{\signature[2]}}$,
  is $\phi = 0$.
  But for signature $s=1$, the evaluation functionals
  from Proposition~\ref{proposition:evfunctionalsdisc} are non-trivial
  positive linear functionals for the product 
  $\starreds{,\hbar}{{\signature[1]}}$ on
  $\Analytic(\DD^1)$. 
  
  As a consequence, the $^*$-algebra
  $\Analytic(\DD^1)$ with product $\starreds{,\hbar}{{\signature[1]}}$
  and pointwise complex conjugation as $^*$-involution
  is not $^*$-isomorphic to the $^*$-algebra
  $\Analytic(\CC\PP^1)$ with product $\starreds{,\hbar}{{\signature[2]}}$
  and pointwise complex conjugation as $^*$-involution.
\end{proposition}
\begin{proof}
  Let $s=2$ and let $\phi\colon \Analytic(\CC\PP^1) \to \CC$ be a positive
  linear functional for the product $\starreds{,\hbar}{{\signature[2]}}$.
  Then the previous Lemma~\ref{lemma:negativesqaure} shows that there
  exist functions $f_1, \dots,f_4 \in \Analytic(\CC\PP^1)$ such that
  \begin{align*}
    0 
    \le 
    \sum_{k=1}^4
    \phi\big(
      \cc{f_k} \starreds{,\hbar}{{\signature[2]}} f_k
    \big)
    =
    \phi(1+\hbar)
    =
    (1+\hbar) \,\phi(1)
    =
    (1+\hbar)\,\phi\big(\cc{1} \starreds{,\hbar}{{\signature[2]}} 1 \big)
    \le
    0
  \end{align*}
  holds because $\hbar < -1$, so $\phi(1) = 0$. But then the Cauchy Schwarz inequality applied to
  the (possibly degenerate) inner product
  $\Analytic(\CC\PP^1) \ni (f,g) \mapsto \phi( \cc{f} \starreds{,\hbar}{{\signature[2]}} g ) \in \CC$
  shows that
  \begin{align*}
    \abs[\big]{\phi(f)}^2
    =
    \abs[\big]{\phi\big(\cc{1} \starreds{,\hbar}{{\signature[2]}} f\big)}^2
    \le
    \phi\big(\cc{1} \starreds{,\hbar}{{\signature[2]}} 1\big)\,
    \phi\big(\cc{f} \starreds{,\hbar}{{\signature[2]}} f\big)
    =
    \phi(1)\,
    \phi\big(\cc{f} \starreds{,\hbar}{{\signature[2]}} f\big)
    =
    0
  \end{align*}
  holds for all $f\in\Analytic(\CC\PP^1)$, and therefore $\phi = 0$. The rest is clear.
\end{proof}




\subsection{Applications}  \label{subsec:wickapplications}

In this section we use the reduced Wick rotation to transfer some of the results obtained in
\cite{kraus.roth.schoetz.waldmann:OnlineConvergentStarProductOnPoincareDisc} for the special
case of the hyperbolic disc, i.e.\ $s=1$, to general 
signatures. Note that one could also check that all the 
proofs in \cite{kraus.roth.schoetz.waldmann:OnlineConvergentStarProductOnPoincareDisc} work for an arbitrary 
signature, but the Wick rotation provides a more elegant way to generalize these results.
We will again drop the superscripts $^{\signature}$ most of the time,
the following is valid for every choice of signature $s \in \{1,\dots,1+n\}$.

\begin{proposition} \label{proposition:wickapplication_basis}
	The fundamental monomials $\MonomRedFund{P}{Q}$ with $P, Q \in \NN_0^n$ 
	form an absolute Schauder basis of $\Analytic(\Mred)$.
	More precisely, every $f \in \Analytic(\Mred)$ can be expanded
	in a unique way as an absolutely convergent series
	\begin{align}
	  f = \sum_{P,Q\in\NN_0^n} \expansionCoefficients{f}{P}{Q} \MonomRedFund{P}{Q}
	\end{align}
    with complex coefficients $\expansionCoefficients{f}{P}{Q}$ that fulfil the estimate
    \begin{align}
      \seminorm{\red,r}{f} 
      \coloneqq
      \sum_{P,Q\in\NN_0^n} \abs{\expansionCoefficients{f}{P}{Q}} r^{\abs{P}+\abs{Q}}
      <
      \infty
    \end{align}
    for all $r\in {[1,\infty[}$. Moreover, the topology of $\Analytic(\Mred)$
    (i.e.\ the topology of locally uniform convergence of the holomorphic extensions to $\MredExt$)
	can equivalently be described by these seminorms $\seminorm{\red,r}{f}$ and 
	the coefficients $\expansionCoefficients{f}{P}{Q}$ can be calculated explicitly by 
	means of the integral formula
	\begin{equation}
	\expansionCoefficients{f}{P}{Q} 
	=
	\frac{1}{(-4\pi^2)^n}
	\oint_C \dots \oint_C \hat f \frac{
		\big(1+\sum_{k=1}^n \nu_k u^k v^k\big)^{\max\{\abs P, \abs Q\}-1}
	}{
		u^{P+(1, \dots, 1)}	v^{Q+(1, \dots, 1)}
	} \D^n u \wedge \D^n v \label{eq:cauchydownstairs}
	\end{equation}
	for all $P,Q\in\NN_0^n$.
	Here $f \in \Analytic(\Mred)$ and 
	$\hat f \in \Holomorphic(\MredExt)$
	satisfies $\DiagM^*(\hat f) = f$. 
	The coordinates $u$ and $v$ were defined in Equation \eqref{eq:coordinates:extended}
	and $C \subseteq \CC$ is, in these projective coordinates, a circle around zero with radius
	in ${]0,1/\sqrt{n}[}$.
\end{proposition}

\begin{proof}
	For $s = 1$ this is exactly the statement of
	\cite[Thm.~3.16]{kraus.roth.schoetz.waldmann:OnlineConvergentStarProductOnPoincareDisc}.
	Because the Wick rotation is a homeomorphic isomorphism and using Lemma~\ref{lemma:wickrotPol},
	the generalization to arbitrary signatures is immediately clear for everything except
	the integral formula.
	In order to prove that \eqref{eq:cauchydownstairs} holds, we have to check that it is 
	compatible with the holomorphic isomorphisms $\geomWickRotRed[\signature]$: 
	
	We have to use superscripts $^{\signature}$ again to indicate the signature $s$.
	As $u^{\signature,k} \circ \geomWickRotRed[\signature] = u^{{\signature[1+n]},k}$ 
	and $v^{\signature,k} \circ \geomWickRotRed[\signature] = v^{{\signature[1+n]},k}$ 
	for all $k \in \{1,\dots,s-1\}$ as well as 
	$u^{\signature,k} \circ \geomWickRotRed[\signature] = \I u^{{\signature[1+n]},k}$ 
	and $v^{\signature,k} \circ \geomWickRotRed[\signature] = \I v^{{\signature[1+n]},k}$ 
	for all $k \in \{s,\dots,n\}$ hold, we get
	\begin{align*}
      \big(\geomWickRotRed[\signature]\big)^*(\D^n u^{\signature} \wedge \D^n v^{\signature})
	  &=
	  (-1)^{n+1-s} \D^n u^{\signature[1+n]} \wedge \D^n v^{\signature[1+n]}
    \shortintertext{as well as}
	  \big(\geomWickRotRed[\signature]\big)^*\bigg(1+\sum_{k=1}^n \nu_k^{\signature} u^{\signature,k} v^{\signature,k}\bigg)
	  &=
	  1+\sum_{k=1}^n \nu_k^{\signature[1+n]} u^{{\signature[1+n]},k} v^{{\signature[1+n]},k}
    \shortintertext{and}
      \big(\geomWickRotRed[\signature]\big)^*\big( (u^{\signature})^{P+(1,\dots,1)} (v^{\signature})^{Q+(1,\dots,1)} \big)
      &=
      \I^{\sum_{k=s}^n (P_k+Q_k+2) } (u^{\signature[1+n]})^{P+(1,\dots,1)} (v^{\signature[1+n]})^{Q+(1,\dots,1)}
	\end{align*}	
	for all $P,Q\in \NN_0^n$. So given $\hat{f} \in \Holomorphic(\MredExt[\signature])$, then
	the right-hand side of \eqref{eq:cauchydownstairs} for $\hat{f}$ in signature $s$,
	multiplied with the factor $\I^{\sum_{k=s}^n (P_k+Q_k) }$,
	gives the same result as for $\hat{f}\circ \geomWickRotRed[\signature] \in \Holomorphic(\MredExt[\signature])$ in signature $1+n$.
	This matches precisely with Lemma~\ref{lemma:wickrotPol}, which shows that
	\begin{align*}
	  \WickRotRed[\signature](f)
	  = 
	  \sum_{P,Q\in \NN_0^n} \expansionCoefficients{f}{P}{Q} 
	  \WickRotRed[\signature](\MonomRed[\signature]{P}{Q})
	  = 
	  \sum_{P,Q\in \NN_0^n} \expansionCoefficients{f}{P}{Q} \,
	  \I^{\sum_{k=s}^n (P_k+Q_k) } \MonomRed[{\signature[1+n]}]{P}{Q}
	\end{align*}
	for all $f\in \Analytic(\Mred[\signature])$ with expansion coefficients
	$\expansionCoefficients{f}{P}{Q}$.
    This way, one first sees that \eqref{eq:cauchydownstairs} holds not only for signature $s=1$
    but also for $s=1+n$, and then that it even holds for all $s\in \{1,\dots,n+1\}$.
\end{proof}
We would now like to generalize Corollary~\ref{corollary:rationalDependenceOnHbar} and
Proposition~\ref{proposition:poissonBracketIsFirstOrderCommutator} for analytic functions.
Because of the poles in $\hbar$ we only discuss one-sided limits:
For some function $f: \StarDomain \cap \RR \to \CC$ the limit when $\hbar$ approaches $0$ 
from the left is denoted by $\lim_{\hbar \to 0^-} f(\hbar)$ (if it exists).
 
\begin{proposition} \label{proposition:poissonBracketIsFirstOrderCommutator:Analytic}
	For any two analytic functions $f, g \in \Analytic(\Mred)$ the limits 
	$\lim_{\hbar \to 0^-} f \starred[,\hbar] g$ and $\lim_{\hbar \to 0^-} 
	\frac 1 {\I\hbar} \big(f \starred[,\hbar] g - g 
	\starred[,\hbar] f\big)$ exist. They are given by
	\begin{align}
	\lim_{\hbar \to 0^-} f \starred[,\hbar] g &= f g
	\shortintertext{and}
	\lim_{\hbar \to 0^-} \frac 1 {\I\hbar} \big(f \starred[,\hbar] g - g 
	\starred[,\hbar] f\big) 
	&= \poi{f}{g}_\red 
	\end{align}
	with the reduced Poisson bracket $\poi{\argument}{\argument}_\red$ on $\Mred$.
\end{proposition}

\begin{proof}
    This was proven in \cite[Thm.~4.5]{kraus.roth.schoetz.waldmann:OnlineConvergentStarProductOnPoincareDisc}
    in the special case of signature $s=1$ for a product $*_\hbar$ with $-\hbar \in \StarDomain$ fulfilling
    $f \star_{\red,\hbar}^{\signature[1]} g = g *_{-\hbar/2} f$ 
    for all $f,g\in \Analytic(\Mred[{\signature[1]}])$ and the corresponding Poisson bracket 
    $\poi{\argument}{\argument}_* = -2\poi{\argument}{\argument}_\red$. 
    The statements for arbitrary signatures $s$ follow immediately from 
    Theorem~\ref{theorem:wickRotation} and Theorem~\ref{theorem:wickstar}.
\end{proof}
Note also that \cite[Expl.~4.2]{kraus.roth.schoetz.waldmann:OnlineConvergentStarProductOnPoincareDisc}
shows that there exist two functions $f,g\in \Analytic(\Mred)$ for which
$f\starred[,\hbar] g$ has non-trivial first order poles at all $\hbar = 1/m$ with $m\in \NN$.
As a consequence, the result of the above Proposition~\ref{proposition:poissonBracketIsFirstOrderCommutator:Analytic}
cannot be generalized to limits over arbitrary sequences $(\hbar_k)_{k\in \NN}$ in $\Omega$ with
limit $0$.




\appendix
\section{Symmetrized Covariant Derivatives}   \label{sec:derivatives}

On a smooth manifold $M$ we define the spaces of tensor fields
\begin{align*}
 \SymSec^\ell(M)
 \coloneqq
 \SmoothSections(\Symten{\ell}\CoTangent M)
 \quad\quad\text{and}\quad\quad
 (\ASymSec\otimes \SymSec)^{k,\ell}(M)
 \coloneqq
 \SmoothSections(\ASymten{k}\CoTangent M \otimes \Symten{\ell}\CoTangent M)
\end{align*}
for all $k,\ell \in \ZZ$, as well as the $\ZZ$-graded algebra $\SymSec^\bullet(M) \coloneqq 
\bigoplus_{\ell\in\ZZ} \SymSec^\ell(M)$
with the usual pointwise symmetric tensor product $\vee$ and the $\ZZ^2$-graded algebra $(\ASymSec \otimes 
\SymSec)^{\bullet,\bullet}(M) \coloneqq 
\bigoplus_{k,\ell\in\ZZ} (\ASymSec\otimes \SymSec)^{k,\ell}(M)$ with product $\circ$ given by the combination
of the pointwise antisymmetric and symmetric tensor products. In order to define graded commutators,
a $\ZZ_2$-grading on these two algebras is needed: In the case of $\SymSec^\bullet(M)$, this is the trivial
one, in which all elements of $\SymSec^\bullet(M)$ have even degree, and on 
$(\ASymSec \otimes \SymSec)^{\bullet,\bullet}(M)$ we consider the antisymmetric degree only,
i.e.~all elements of $(\ASymSec \otimes \SymSec)^{k,\ell}(M)$ with $k\in 2\ZZ, \ell \in \ZZ$ have even degree
and all elements of $(\ASymSec \otimes \SymSec)^{k,\ell}(M)$ with $k\in 1+2\ZZ, \ell \in \ZZ$ have odd degree.
This way, both $\SymSec^\bullet(M)$ and $(\ASymSec \otimes \SymSec)^{\bullet,\bullet}(M)$ are graded commutative.
For later use we also define the \neu{total degree} $\Deg$ on $(\ASymSec \otimes \SymSec)^{\bullet,\bullet}(M)$ by setting
\begin{align}
  \Deg \Omega = (k+\ell)\Omega
\end{align}
for all $\Omega \in (\ASymSec \otimes \SymSec)^{k,\ell}(M)$ with $k,\ell\in \ZZ$.
Clearly $\Deg$ is a graded derivation of degree $(0,0)$.
Note that
$\SymSec^0(M) \cong \Smooth(M) \cong (\ASymSec \otimes \SymSec)^{0,0}(M)$ and that 
$\SymSec^\bullet(M)$ is generated as a complex algebra by $\SymSec^0(M) \oplus \SymSec^1(M)$,
whereas $(\ASymSec \otimes \SymSec)^{\bullet,\bullet}(M)$ is generated as a complex algebra by
$(\ASymSec \otimes \SymSec)^{0,0}(M) \oplus 
 (\ASymSec \otimes \SymSec)^{1,0}(M) \oplus
 (\ASymSec \otimes \SymSec)^{0,1}(M)$.

We will need two other operators, the \neu{Koszul differentials}: 
There are unique $\CC$-linear graded derivations $\delta,\delta^*$
of $(\ASymSec \otimes \SymSec)^{\bullet,\bullet}(M)$ of degree $(+1,-1)$
and $(-1,+1)$ that fulfil
\begin{align}
 \delta(1\otimes \omega) = \omega\otimes 1
 \quad\quad\text{as well as}\quad\quad
 \delta^*(\rho\otimes 1) = 1 \otimes\rho
 \komma
\end{align}
respectively,
for all $\rho,\omega \in \SmoothSections(\CoTangent M)$.
In local coordinates, 
$\delta(\rho \otimes \omega) = \sum_i(\D x^i \wedge \rho) \otimes  (\iota_{\partial / \partial x^i} \omega)$
and 
$\delta^*(\rho \otimes \omega) = \sum_i (\iota_{\partial / \partial x^i} \rho) \otimes  (\D x^i \vee \omega)$
hold for all $\rho \in \SmoothSections(\ASymten{k}\CoTangent M)$
and $\omega \in \SmoothSections(\Symten{\ell}\CoTangent M)$. Of course,
$\delta$ and $\delta^*$ are not only $\CC$-linear but even $\Smooth(M)$-linear.

\begin{lemma}
	For the graded commutators we have
	\begin{gather*}
	\kom{\delta}{\delta} = 2\delta^2 = 0
	\komma\quad\quad
	\kom{\delta^*}{\delta^*} = 2(\delta^*)^2 = 0
	\komma\quad\quad
	\kom{\delta}{\delta^*} = \kom{\delta^*}{\delta} = \Deg 
	\komma\\
	\kom{\Deg}{\delta} = -\kom{\delta}{\Deg} = 0
	\quad\quad\text{and}\quad\quad
	\kom{\Deg}{\delta^*} = -\kom{\delta^*}{\Deg} = 0 \punkt
	\end{gather*}
\end{lemma} 
\begin{proof}
	One checks easily that this holds on
	$(\ASymSec \otimes \SymSec)^{0,0}(M)$,
	$(\ASymSec \otimes \SymSec)^{1,0}(M)$ and 
	$(\ASymSec \otimes \SymSec)^{0,1}(M)$.
	But graded derivations are already uniquely determined by how they act on these spaces.
\end{proof}
One can also check that $\delta$ and $\delta^*$ commute with pullbacks. That is, whenever $\Psi : M \to N$ is 
smooth, then $\delta \circ \Psi^* = \Psi^* \circ \delta$ and $\delta^* \circ \Psi^* = \Psi^* \circ \delta^*$ where
$\Psi^* : (\ASymSec \otimes \SymSec)^{\bullet,\bullet}(N) \to (\ASymSec \otimes \SymSec)^{\bullet,\bullet}(M)$ 
denotes the usual pullback. 

Next we consider the insertion of vector fields into the antisymmetric and symmetric part:
Given $X \in \SmoothSections(\Tangent M)$, then there exist unique 
$\CC$-linear graded derivations $\iota^a_X, \iota^s_X$ of $(\ASymSec \otimes \SymSec)^{\bullet,\bullet}(M)$
of degree $(-1,0)$ and $(0,-1)$ that fulfil
\begin{align}
  \iota^a_X (\rho \otimes 1) = \dupr{\rho}{X}
  \quad\quad\text{as well as}\quad\quad
  \iota^s_X (1 \otimes \omega) = \dupr{\omega}{X}
  \komma
\end{align}
respectively, for all $\rho,\omega \in \SmoothSections(\CoTangent M)$.
Clearly, $\iota^a_X$ and $\iota^s_X$ are even $\Smooth(M)$-linear and:

\begin{lemma}
For the graded commutators we have
\begin{gather*}
\kom{\iota^a_X}{\iota^a_Y} 
= 
\kom{\iota^a_X}{\iota^s_Y} 
= 
\kom{\iota^s_X}{\iota^a_Y}
= 
\kom{\iota^a_X}{\delta^*}
= 
\kom{\iota^s_X}{\delta}
=
0 \komma \\
\kom{\iota^a_X}{\delta} = \iota^s_X
\komma\quad
\kom{\iota^s_X}{\delta^*} = \iota^a_X
\komma\quad
\kom{\Deg}{\iota^a_X} = -\iota^a_X
\komma\quad
\kom{\Deg}{\iota^s_X} = -\iota^s_X
\end{gather*}
for all $X,Y\in\Smooth(M)$.
\end{lemma}
\begin{proof}
 These identities are easy to check on 
$(\ASymSec \otimes \SymSec)^{0,0}(M)$,
$(\ASymSec \otimes \SymSec)^{1,0}(M)$ and 
$(\ASymSec \otimes \SymSec)^{0,1}(M)$.
\end{proof}
We see that the $\CC$-linear span of $\delta$, $\delta^*$, $\Deg$ 
and all $\iota^a_X$ and $\iota^s_X$ with $X\in\Smooth(M)$ in the graded Lie algebra
of $\CC$-linear graded derivations of $(\ASymSec \otimes \SymSec)^{\bullet,\bullet}(M)$
is a graded Lie subalgebra.
Now we can define exterior covariant derivatives:

\begin{definition}
	A $\CC$-linear graded derivation $D$ of $(\ASymSec\otimes \SymSec)^{\bullet,\bullet}(M)$
	of degree $(+1,0)$ that fulfils
	$D (\rho \otimes 1) = \D \rho \otimes 1$ for all $\rho \in \SmoothSections(\ASymten{\bullet} \CoTangent M)$
	is called an \neu{exterior covariant derivative} on $M$.
\end{definition}
For every covariant derivative $\nabla$ on $M$ there exists a unique exterior covariant derivative 
$D^\nabla$ on $M$
that fulfils 
\begin{equation}
  \iota^a_X D^\nabla (1\otimes \omega) = 1\otimes \nabla_X \omega
\end{equation}
for all $\rho \in \SmoothSections(\ASymten{\bullet} \CoTangent M)$,
$\omega\in \SymSec^\bullet(M)$ and $X \in \SmoothSections(\Tangent M)$.   
In local coordinates,
\begin{equation} \label{eq:ExtCovDerAssociatedToCovDerLocal}
  D^\nabla (\rho \otimes \omega)
  = 
  \D \rho \otimes \omega + \sum_i(\D x^i \wedge \rho) \otimes \nabla_{\partial / \partial x^i} \omega
\end{equation}
for all $\rho \in \SmoothSections(\ASymten{\bullet} \CoTangent M)$ and
$\omega\in \SymSec^\bullet(M)$.
Conversely, every exterior covariant derivative $D$ on $M$
determines a unique covariant derivative $\nabla^D$ on $M$ that fulfils
\begin{equation}\label{eq:CovDerAssociatedToExtCovDer}
\dupr{\omega}{\nabla^D_X Y} 
= 
X\big( \dupr{\omega}{Y}\big) - \dupr{\nabla^D_X \omega}{Y} 
= 
X\big( \dupr{\omega}{Y}\big) - \iota^s_Y \iota^a_X D(1\otimes \omega)
\end{equation}
for all $X,Y\in \SmoothSections(\Tangent M)$ and all $\omega \in \SmoothSections(\CoTangent M)$.
One can check that $\nabla^{D^\nabla} = \nabla$ for every covariant derivative $\nabla$ on $M$ 
and that $D^{\nabla^D} = D$ for every exterior covariant derivative on $M$. So there is
a $1$-to-$1$ correspondence between covariant derivatives and exterior covariant derivatives.

We say that an exterior covariant derivative $D$ is torsion-free if the associated covariant derivative 
$\nabla^D$ is torsion-free.

\begin{proposition}
	An exterior covariant derivative $D$ on $M$ is torsion-free if and only if $\kom{D}{\delta} = 0$.
\end{proposition}
\begin{proof}
	Denote the torsion of $\nabla^D$ by $T$.
	We compute
	\begin{align*}
	\iota^a_Y \iota^a_X \kom{D}{\delta}(1 \otimes \omega) 
	&= 
	\iota^a_Y \iota^a_X (\D \omega \otimes 1) + \iota^a_Y \iota^a_X \delta D (1\otimes 
	\omega) 
	\\
	&=
	2\dupr{\D\omega}{X\wedge Y} - \iota^a_Y \delta \iota^a_X D (1\otimes \omega) + \iota^a_Y 
	\iota^s_X D (1\otimes \omega)
	\\
	&=
	2\dupr{\D\omega}{X\wedge Y} - \iota^a_Y \delta (1 \otimes \nabla_X \omega) + \iota^s_X 
	\iota^a_Y D 
	(1\otimes \omega)
	\\
	&=
	2\dupr{\D\omega}{X\wedge Y} - \dupr{\nabla_X \omega}{Y} + \dupr{\nabla_Y \omega}{X}
	\\
	&=
	2\dupr{\D\omega}{X\wedge Y} - X\big( \dupr{\omega}{Y} \big) +  Y\big( \dupr{\omega}{X} \big)
	+\dupr{\omega}{\nabla_XY} - \dupr{\omega}{\nabla_YX}
	\\
	&=
	\dupr{\omega}{-\kom{X}{Y}+\nabla_X Y - \nabla_YX }
	\\
	&=
	\dupr{\omega}{T_{X,Y}} \punkt
	\end{align*}
	In particular, if $\kom{D}{\delta} = 0$, then $\nabla^D$ is torsion-free. Conversely, if $\nabla^D$ is 
	torsion-free, then $\kom{D}{\delta}$ vanishes on $(\ASymSec\otimes\SymSec)^{0,1}(M)$ by the above 
	calculation.
	But $\kom{D}{\delta}$ is a $\CC$-linear graded derivation of 
	$(\ASymSec\otimes\SymSec)^{\bullet,\bullet}(M)$
	of degree $(+2,-1)$, so $\kom{D}{\delta} = 0$ in this case.
\end{proof}
If $g \in \SymSec^2(M)$ is real and non-degenerate,
then there exists a unique exterior covariant derivative $D$ on $M$ that fulfils $D(1\otimes g) = 0 = 
\kom{D}{\delta}$,
namely the one corresponding to the Levi-Civita connection.
This \neu{exterior Levi-Civita connection} will be interesting for us:

\begin{lemma} \label{lemma:derivinvariant}
	Let $M$ be a smooth manifold, $g \in \SmoothSections(\Symten{2} \CoTangent M)$ a real and 
	non-degenerate symmetric tensor with Levi-Civita connection $\nabla$, and 
	$\Phi\colon M\to M$ a diffeomorphism. If $\nabla_X \Phi^*(g) = 0$ for all $X\in\SmoothSections(\Tangent M)$, then
	the exterior Levi-Civita connection $D$ associated to $g$ commutes with the pullback $\Phi^*$,
	i.e.\ $D \Phi^*(\Omega) = \Phi^*(D \Omega)$ for all $\Omega \in (\ASymSec \otimes 
	\SymSec)^{\bullet,\bullet}(M)$.
\end{lemma}
\begin{proof}
	It suffices to show that 
	$D' \colon (\ASymSec \otimes \SymSec)^{\bullet,\bullet}(M) \to (\ASymSec \otimes \SymSec)^{\bullet,\bullet}(M)$,
	$\Omega \mapsto D'\Omega \coloneqq (\Phi^{-1})^* ( D \Phi^*(\Omega))$ is an
	exterior covariant derivative and fulfils 
	$D'(1\otimes g) = 0 = \kom{D'}{\delta}$: It is easy to see that $D'$ is a
	$\CC$-linear graded derivation of $(\ASymSec\otimes \SymSec)^{\bullet,\bullet}(M)$ of degree $(+1,0)$
	that fulfils $D' (\rho \otimes 1) = \D \rho \otimes 1$ for all 
	$\rho \in \SmoothSections(\ASymten{\bullet} \CoTangent M)$, hence an exterior covariant derivative.
	It commutes with $\delta$ (in the graded sense) because $\delta$ commutes with $D$ and all pullbacks. 
	Finally, $D'(1\otimes g)$ holds because $\nabla_X \Phi^*(g) = 0$ for all $X\in\SmoothSections(\Tangent M)$.
\end{proof}
Note that the condition $\nabla \Phi^*(g) = 0$ is fulfilled e.g.\ if $\Phi^*(g) = g$, but also more generally 
if $\Phi^*(g) = \lambda g$ with $\lambda \in \CC$.

\begin{proposition} \label{proposition:covDer:reduction}
	Let $M$ be a smooth manifold endowed with a free and proper action 
	$\argument \acts \argument$ of a Lie group $G$ and a
	$G$-invariant exterior covariant derivative $D$ on $M$ (i.e.\ $D$ commutes with the action of $G$ on 
	$(\ASymSec \otimes \SymSec)^{\bullet,\bullet}(M)$ by pullbacks like in the previous Lemma~\ref{lemma:derivinvariant}). 
	Moreover, write $\bigpr \colon M \to M/G$  for the canonical projection onto the quotient manifold $M/G$ and 
	assume we have chosen a smooth $G$-invariant complement $\Xi = \bigcup_{p\in M} \Xi_p$ 
	of $\ker(\Tangent\bigpr)$, i.e.\ a linear subbundle of $\Tangent M$ such that 
	$\Tangent M = \Xi \oplus \ker(\Tangent \bigpr)$ and such that 
	$\Xi_{g\acts p} = \big(\Tangent_p (\argument \racts g)\big) (\Xi_p)$ for all $p\in M$.
	Let $\bigproj_\Xi \colon \SmoothSections(\Tangent M) \to  \SmoothSections(\Tangent M)$
	be the corresponding projection on this subbundle $\Xi$ and 
	$\bigproj_\Xi^* \colon  \SmoothSections(\CoTangent M) \to  \SmoothSections(\CoTangent M)$
	its dual projection. Then
	\begin{align}
    \bigpr^* \big(D_\red \Omega\big) \coloneqq (\bigproj_\Xi^*)^{\otimes (k+1+\ell)} D \bigpr^*(\Omega)
    \label{eq:def:Dred}
  \end{align}
  for all $\Omega \in (\ASymSec\otimes \SymSec)^{k,\ell}(M/G)$, $k,\ell \in \NN_0$
  defines an exterior covariant derivative on $M/G$. If $D$ is torsion-free, then $D_\red$
	also remains torsion-free.
\end{proposition}

\begin{proof}
  Since $D$ and $\Xi$ are $G$-invariant, it follows that
  $(\bigproj_\Xi^*)^{\otimes (k+1+\ell)} D \bigpr^*(\Omega)$ is $G$-invariant, so
  \eqref{eq:def:Dred} does describe a well-defined $\CC$-linear endomorphism
  $D_\red$ of $(\ASymSec\otimes \SymSec)^{\bullet,\bullet}(M/G)$ of degree $(1,0)$
  and one can also check that $D_\red$
  is again a graded derivation.
  Using that $\Tangent \bigpr \circ \bigproj_\Xi = \Tangent \bigpr$ one sees that
  $D_\red(\rho \otimes 1) = \D \rho \otimes 1$ holds for all
  $\rho \in \SmoothSections(\Lambda^\bullet \CoTangent M)$.
  As $\delta$ commutes with pullbacks and $\bigproj_\Xi^*$, one also finds that
  $D_\red$ (graded) commutes with $\delta$ if $D$ does.
\end{proof}
Next we consider the graded commutator of an exterior covariant derivative $D$
on a smooth manifold $M$ with $\delta^*$, which is a $\CC$-linear graded derivation
of $(\ASymSec\otimes\SymSec)^{\bullet,\bullet}(M)$
of degree $(0,+1)$ and satisfies 
$\kom{D}{\delta^*} (f) = \delta^* D f = 1\otimes \D f$ for all $f\in \Smooth(M)$.
So $\kom{D}{\delta^*}$ restricts to a $\CC$-linear derivation $D^{\sym}$ of $\SymSec^\bullet(M)$
of degree $1$. 

\begin{definition}
	A $\CC$-linear derivation $\Delta$ of $\SymSec^\bullet(M)$ of degree $1$
	that fulfils $\Delta f = \D f$ for all $f\in \Smooth(M)$ is called a 
	\neu{symmetrized covariant derivative}, and for every exterior covariant
	derivative $D$ of $M$ we define its induced symmetrized covariant derivative
	$D^\sym \colon \SymSec^\bullet(M) \to \SymSec^\bullet(M)$ by
	\begin{align}
	1 \otimes D^\sym \omega \coloneqq \kom{D}{\delta^*} (1 \otimes \omega)
	\end{align}
	for all $\omega \in \SymSec^\bullet(M)$.
\end{definition}
Given an exterior covariant derivative $D$, we compute that its induced symmetrized covariant derivative 
$D^\sym$ fulfils 
\begin{align*}
1 \otimes \iota_Y \iota_X D^\sym \omega 
&=
\iota^s_Y \iota^s_X \kom{D}{\delta^*}(1 \otimes \omega) 
\\
&=
\iota^s_Y \iota^s_X \delta^* D (1 \otimes \omega) 
\\
&= 
\iota^s_Y  \iota^a_X D (1 \otimes \omega) + \iota^s_Y \delta^* \iota^s_X D (1 \otimes \omega)
\\
&= 
1 \otimes \iota_Y \nabla_X^D \omega + \iota^a_Y \iota^s_X D (1 \otimes \omega)
\\
&= 
1\otimes \big( \iota_Y \nabla_X^D \omega + \iota_X \nabla_Y^D \omega \big)
\\
&= 
1 \otimes \big( \dupr{\nabla_X^D \omega}{Y} + \dupr{\nabla_Y^D \omega}{X} \big)
\end{align*}
for all $\omega \in \SmoothSections(\CoTangent M)$. So in local coordinates,
$D^{\sym} \omega = \D x^i \vee \nabla^D_{\partial/\partial x^i} \omega$.

Conversely, every $\CC$-linear derivation $\Delta$ of $\SymSec^\bullet(M)$
that fulfils $\Delta f = \D f$ for all $f\in \Smooth(M)$ defines 
a covariant derivative $\nabla^\Delta$ on $M$ by
\begin{align}
  \dupr[\big]{\nabla^\Delta_X \omega }{Y}
  \coloneqq
  \dupr{\Delta \omega}{ X\vee Y }
  + 
  \frac{1}{2} \big( 
    X(\dupr{\omega}{Y})
    -
    Y(\dupr{\omega}{X})
    -
    \dupr{\omega}{\kom{X}{Y}}
  \big)
\end{align}
for all $\omega \in \SmoothSections(\CoTangent M)$ and all $X,Y \in \SmoothSections(\Tangent M)$.
This covariant derivative $\nabla^\Delta$ then is torsion-free because
\begin{align}
  \dupr{\nabla_X^\Delta \omega}{Y} - \dupr{\nabla_Y^\Delta \omega}{X}
  =
    X(\dupr{\omega}{Y})
    -
    Y(\dupr{\omega}{X})
    -
    \dupr{\omega}{\kom{X}{Y}}  
\end{align}
and fulfils
\begin{align}
  \dupr{\nabla_X^\Delta \omega}{Y} + \dupr{\nabla_Y^\Delta \omega}{X}
  =
  2 \dupr{\Delta \omega}{ X\vee Y }
  =
  \iota_Y \iota_X \Delta \omega
  \punkt
\end{align}
Consequently there is a $1$-to-$1$-correspondence between torsion-free covariant
derivatives (or their exterior covariant derivatives) and symmetrized
covariant derivatives. For the reduction of symmetrized covariant derivatives
we get:

\begin{proposition} \label{proposition:covsymDer:reduction}
  Let $M$, $G$, $D$, $\bigpr$ and $\Xi$ be like in Proposition~\ref{proposition:covDer:reduction}. Then 
  $D^\sym_\red$, the symmetrized covariant derivative on $M/G$ constructed out of
  the reduced exterior covariant derivative $D_\red$, fulfils
  \begin{align}
    \bigpr^* \big(D_\red^\sym \omega \big) = (\bigproj_\Xi^*)^{\otimes (k+1)} D^\sym \bigpr^*(\omega)
    \label{eq:Dredsym}
  \end{align}
  for all $\omega \in \SymSec^{k}(M/G)$, $k \in \NN_0$.
\end{proposition}

\begin{proof}
	As $\delta^*$ commutes with the pullback
	$\bigpr^*$ and the projection $\bigproj_\Xi^*$
	this follows immediately from \eqref{eq:def:Dred}.
\end{proof}
Being an endomorphism of $\SymSec^\bullet(M)$, a symmetrized covariant derivative
$D^\sym$ can be iterated. Given
$k \in \NN_0$, $X_0 \in \SmoothSections(\Symten{0} \Tangent M)$, \dots, 
$X_k \in \SmoothSections(\Symten{k} \Tangent M)$, then
\begin{align*}
\Smooth(M) \ni f \mapsto \sum_{r=0}^k \dupr[\big]{(D^\sym)^r f}{X_r} \in \Smooth(M)
\end{align*}
is a differential operator of degree $k$. Conversely, by induction over their symbols,
one can show that all differential operators of degree $k$ on $\Smooth(M)$ are of
this form. So symmetrized covariant derivatives allow to describe differential
operators rather explicitly but without requiring a choice of coordinates.

Finally, if $M$ is a complex manifold, then its tangent and cotangent space split into $(1,0)$ and $(0,1)$ 
parts.
Consequently
\begin{align}
  \SymSec^k(M) &= \bigoplus_{p+q=k} \SymSec^{(p,q)}(M)
\shortintertext{and}
  (\ASymSec\otimes\SymSec)^{k,\ell}(M) &= \bigoplus_{\substack{p+q=k\\r+s=\ell}} (\ASymSec\otimes\SymSec)^{(p,q),(r,s)}(M)
\end{align}
also split into subspaces
\begin{align}
  \SymSec^{(p,q)}(M) 
  &\coloneqq 
  \SmoothSections\big(\Symten p \Tangent^{*,(1,0)} M \vee \Symten q \Tangent^{*,(0,1)} M\big)
\shortintertext{and}
  (\ASymSec\otimes\SymSec)^{(p,q),(r,s)}(M)
  &\coloneqq 
  \SmoothSections\big(
    \ASymten p \Tangent^{*,(1,0)} M \wedge \ASymten q \Tangent^{*,(0,1)} M
    \otimes
    \Symten r \Tangent^{*,(1,0)} M \vee \Symten s \Tangent^{*,(0,1)} M
  \big)
  \punkt
\end{align}
Note that $\delta$ and $\delta^*$ are compatible with this splitting in the sense that
\begin{align}
  \delta \big( (\ASymSec\otimes\SymSec)^{(p,q),(r,s)}(M) \big)
  \subseteq
  (\ASymSec\otimes\SymSec)^{(p+1,q),(r-1,s)}(M) \oplus (\ASymSec\otimes\SymSec)^{(p,q+1),(r,s-1)}(M)
\shortintertext{and}
  \delta^* \big( (\ASymSec\otimes\SymSec)^{(p,q),(r,s)}(M) \big)
  \subseteq
  (\ASymSec\otimes\SymSec)^{(p-1,q),(r+1,s)}(M) \oplus (\ASymSec\otimes\SymSec)^{(p,q-1),(r,s+1)}(M)
  \label{eq:deltaStarCompwithI}
\end{align}
hold for all $p,q,r,s\in \NN_0$. For symmetrized covariant derivatives, there is a
similar compatibility condition:

\begin{definition}\label{definition:compatibleWithComplexStructure}
	Let $M$ be a complex manifold and $D$ an exterior covariant derivative on $M$.
	Then $D$ is said to be \neu{compatible with the complex structure} if
    \begin{equation} \label{eq:compatibilityWithComplexStructure}
        D\big((\ASymSec\otimes \SymSec)^{(p,q),(r,s)}(M)\big) 
        \subseteq  
        (\ASymSec\otimes \SymSec)^{(p+1,q),(r,s)}(M) \oplus (\ASymSec\otimes \SymSec)^{(p,q+1),(r,s)}(M)
    \end{equation}
	holds for all $p,q,r,s \in \NN_0$.
\end{definition}
If $D$ is the exterior covariant derivative associated to a covariant derivative $\nabla$,
then it follows from Equations \eqref{eq:ExtCovDerAssociatedToCovDerLocal} and \eqref{eq:CovDerAssociatedToExtCovDer}
that $D$ is compatible with the complex structure if and only if $\nabla_X$
preserves the holomorphic and antiholomorphic parts of the tangent bundle 
for all $X \in \SmoothSections(\Tangent \CC^{1+n})$, i.e.\ 
$\nabla_X (\SmoothSections(\Tangent^{(1,0)} \CC^{1+n})) \subseteq \SmoothSections(\Tangent^{(1,0)} 
\CC^{1+n})$ and
$\nabla_X (\SmoothSections(\Tangent^{(0,1)} \CC^{1+n})) \subseteq \SmoothSections(\Tangent^{(0,1)} 
\CC^{1+n})$.
As an example, this is well-known to be the case for the Levi-Civita covariant derivative
on a Kähler manifold.

Using Equation \eqref{eq:deltaStarCompwithI} it is also easy to check that condition 
\eqref{eq:compatibilityWithComplexStructure} implies that the symmetrized covariant derivative fulfils 
$D^\sym (\SymSec^{(p,q)}(M)) \subseteq \SymSec^{(p+1,q)}(M) \oplus \SymSec^{(p,q+1)}(M)$.

\begin{definition} \label{definition:holantiholofD}
	Let $M$ be a complex manifold and $D$ an exterior covariant derivative compatible
	with the complex structure. Then we define     
	\begin{align}
	D_\hol, D_{\antihol} 
	&\colon 
	(\ASymSec\otimes \SymSec)^{\bullet,\bullet}(M) \to (\ASymSec\otimes \SymSec)^{\bullet,\bullet}(M)
	\shortintertext{and}
	D^\sym_\hol, D^\sym_{\antihol} &\colon \SymSec^{\bullet}(M) \to \SymSec^{\bullet}(M)
	\end{align}
	as the $(1,0)$ and $(0,1)$-components of $D$ and $D^\sym$, 
	respectively, i.e.
	\begin{align}
	D_\hol &\coloneqq \sum_{p,q,r,s \in \NN_0} \bigproj^{*,(p+1,q),(r,s)} D \bigproj^{*,(p,q),(r,s)} 
	\komma
	& 
	D_{\antihol} &\coloneqq \sum_{p,q,r,s \in \NN_0} \bigproj^{*,(p,q+1),(r,s)} D \bigproj^{*,(p,q),(r,s)} 
	\shortintertext{and}
	D_\hol^\sym &\coloneqq \sum_{p,q \in \NN_0}\bigproj^{*,(p+1,q)} D^\sym \bigproj^{*,(p,q)}
	\komma& 
	D_{\antihol}^\sym &\coloneqq \sum_{p,q \in \NN_0} \bigproj^{*,(p,q+1)} D^\sym \bigproj^{*,(p,q)}
	\end{align}
	with 
	$\bigproj^{*,(p,q),(r,s)} \colon 
	(\ASymSec\otimes \SymSec)^{(\bullet,\bullet),(\bullet,\bullet)}(M) \to 
	(\ASymSec\otimes \SymSec)^{(p,q),(r,s)}(M)$
	and
	$\bigproj^{*,(p,q)} \colon \SymSec^{(\bullet,\bullet)}(M) \to \SymSec^{(p,q)}(M)$
	the projections on graded subspaces.
\end{definition}
Because of the required compatibility with the complex structure one gets
$D = D_\hol + D_{\antihol}$ and $D^\sym = D^\sym_\hol + D^\sym_{\antihol}$.
Furthermore,
\begin{align}
  \kom[\big]{D_\hol}{\delta^*}(1\otimes \omega) = 1 \otimes D^\sym_\hol \omega
\end{align}
holds for all $\omega \in \SymSec^\bullet(M)$, and analogously for the antiholomorphic part. Consequently:

\begin{proposition} \label{proposition:holahol}
  Let $M$ be a complex manifold, $D$ an exterior covariant derivative on $M$ that
  is compatible with the complex structure, and $D^\sym$ its symmetrized
  covariant derivative. Then
  \begin{equation}
    \bigproj^{*,(k,0)} (D^\sym)^k f
    =
    (D^\sym_\hol)^k f
    \quad\quad\text{and}\quad\quad
    \bigproj^{*,(0,k)} (D^\sym)^k f
    =
    (D^\sym_{\antihol})^k f
 \end{equation}
 hold for all $f\in \Smooth(M)$ and all $k\in \NN_0$, with $\bigproj^{*,(k,0)}$ like in the 
 previous Definition~\ref{definition:holantiholofD}.
\end{proposition}
\begin{proof}
  This follows immediately from the decomposition $D^\sym = D^\sym_\hol + D^\sym_{\antihol}$.
\end{proof}





\begin{thebibliography}{10}

\bibitem
  {alekseev.lachowska:InvariantStarProductsOnCoadjointOrbitsAndTheShapovalovPairing}
\textsc{Alekseev, A., Lachowska, A.: }\newblock \emph{Invariant
  {$\ast$}-products on coadjoint orbits and the {S}hapovalov pairing}.
\newblock Commentarii Mathematici Helvetici  \textbf{80} (2005), 795--810.

\bibitem
  {bayen.flato.fronsdal.lichnerowicz.sternheimer:DeformationTheoryAndQuantization}
\textsc{Bayen, F., Flato, M., Fronsdal, C., Lichnerowicz, A., Sternheimer, D.:
  }\newblock \emph{Deformation theory and quantization.}
\newblock Annals of Physics  \textbf{111}.1 (1978), 61--151.

\bibitem {beiser.waldmann:FrechetAlgebraicDeformationOfThePoincareDisc}
\textsc{Beiser, S., Waldmann, S.: }\newblock \emph{Fr{\'e}chet algebraic
  deformation quantization of the Poincar{\'e} disk}.
\newblock Journal f{\"u}r die reine und angewandte Mathematik (Crelles Journal)
   \textbf{688} (2014), 147--207.

\bibitem
  {bertelson.bieliavsky.gutt:ParametrizingEquivalenceClassesOfInvariantStarProducts}
\textsc{Bertelson, M., Bieliavsky, P., Gutt, S.: }\newblock \emph{Parametrizing
  Equivalence Classes of Invariant Star Products}.
\newblock Letters in Mathematical Physics  \textbf{46} (1998), 339--345.

\bibitem {bertelson.cahen.gutt:EquivalenceOfStarProducts}
\textsc{Bertelson, M., Cahen, M., Gutt, S.: }\newblock \emph{Equivalence of
  Star Products}.
\newblock Classical and Quantum Gravity  \textbf{14} (1997), A93--A107.

\bibitem
  {bieliavsky.gayral:DeformationQuantizationForActionsOfKaehlerianLieGroups}
\textsc{Bieliavsky, P., Gayral, V.: }\newblock \emph{Deformation Quantization
  for Actions of {K}{\"a}hlerian Lie Groups}, vol.\ 236.1115 in \emph{Memoirs of
  the American Mathematical Society}.
\newblock American Mathematical Society, Providence, RI, 2015.

\bibitem
  {bordemann.brischle.emmrich.waldmann:PhaseSpaceReductionForStarProducts.ExplicitConstruction}
\textsc{Bordemann, M., Brischle, M., Emmrich, C., Waldmann, S.: }\newblock
  \emph{Phase space reduction for star-products: An explicit construction for
  CP{$^n$}}.
\newblock Letters in Mathematical Physics  \textbf{36}.4 (1996), 357--371.

\bibitem
  {bordemann.brischle.emmrich.waldmann:SubalgebrasWithConvergingStarProducts}
\textsc{Bordemann, M., Brischle, M., Emmrich, C., Waldmann, S.: }\newblock
  \emph{Subalgebras with converging star products in deformation quantization:
  An algebraic construction for CP{$^n$}}.
\newblock Journal of Mathematical Physics  \textbf{37}.12 (1996), 6311--6323.

\bibitem {cahen.gutt.rawnsley:QuantisationOfKaehlerManifolds3}
\textsc{Cahen, M., Gutt, S., Rawnsley, J.: }\newblock \emph{Quantization of
  K{\"a}hler manifolds. III}.
\newblock Letters in Mathematical Physics  \textbf{30}.4 (1994), 291--305.

\bibitem
  {dewilde.lecomte:ExistenceOfStarProductsAndOfFormalDeformationsOfThePoissonLieAlgebraOfArbitrarySymplecticManifolds}
\textsc{DeWilde, M., Lecomte, P. B.~A.: }\newblock \emph{Existence of
  Star-Products and of Formal Deformations of the Poisson Lie Algebra of
  Arbitrary Symplectic Manifolds}.
\newblock Letters in Mathematical Physics  \textbf{7} (1983), 487--496.

\bibitem
  {esposito.klein.schnitzer:AProofOfTsygansFormalityConjectureForHamiltonianActions}
\textsc{Esposito, C., de~Kleijn, N., Schnitzer, J.: }\newblock \emph{A proof of
  Tsygan's formality conjecture for Hamiltonian actions}.
\newblock arXiv:1812.00403   (2018).

\bibitem {esposito.schmitt.waldmann:OnlineComparisonContinuityWickStarProducts}
\textsc{Esposito, C., Schmitt, P., Waldmann, S.: }\newblock \emph{Comparison
  and continuity of Wick-type star products on certain coadjoint orbits}.
\newblock Forum Mathematicum \textbf{31}.5 (2019), 1203--1223.

\bibitem {esposito.stapor.waldmann:ConvergenceOfTheGuttStarProduct}
\textsc{Esposito, C., Stapor, P., Waldmann, S.: }\newblock \emph{Convergence of
  the Gutt Star Product}.
\newblock Journal of Lie Theory  \textbf{27} (2017), 579--622.

\bibitem {fedosov:ASimpleGeometricConstructionOfDeformationQuantization}
\textsc{Fedosov, B.~V.: }\newblock \emph{A Simple Geometrical Construction of
  Deformation Quantization}.
\newblock Journal of Differential Geometry  \textbf{40} (1994), 213--238.

\bibitem {hoermander:ComplexAnalysisInSeveralVariables}
\textsc{H{\"{o}}rmander, L.: }\newblock \emph{An Introduction to Complex
  Analysis in Several Variables}.
\newblock North Holland Mathematical Library, 3. edition, 1990.

\bibitem {kontsevich:DeformationQuantizationOfPoissonManifolds}
\textsc{Kontsevich, M.: }\newblock \emph{Deformation Quantization of {P}oisson
  manifolds}.
\newblock Letters in Mathematical Physics  \textbf{66} (2003), 157--216.

\bibitem
  {kraus.roth.schoetz.waldmann:OnlineConvergentStarProductOnPoincareDisc}
\textsc{Kraus, D., Roth, O., Sch{\"o}tz, M., Waldmann, S.: }\newblock \emph{A
  convergent star product on the Poincar\'e disc}.
\newblock Journal of Functional Analysis \textbf{277}.8 (2019), 2734--2771.

\bibitem {loeffler:FedosovDifferentialsAndCartanNumbers}
\textsc{L{\"o}ffler, J.: }\newblock \emph{Fedosov differentials and Catalan
  numbers}.
\newblock Journal of Physics A: Mathematical and Theoretical  \textbf{43}.29
  (2010), 299801.

\bibitem {natsume.nest.peter:StrictQuantizationOfSymplecticManifolds}
\textsc{Natsume, T., Nest, R., Peter, I.: }\newblock \emph{Strict Quantizations
  of Symplectic Manifolds}.
\newblock Letters in Mathematical Physics  \textbf{66} (2003), 73--89.

\bibitem {nest.tsygan:AlgebraicIndexTheorem}
\textsc{Nest, R., Tsygan, B.: }\newblock \emph{Algebraic Index Theorem}.
\newblock Communications in Mathematical Physics  \textbf{172} (1995), 223--262.

\bibitem
  {reichert.waldmann:ClassificationOfEquivariantStarproductsOnSymplecticManifolds}
\textsc{Reichert, T., Waldmann, S.: }\newblock \emph{Classification of
  Equivariant Star Products on Symplectic Manifolds}.
\newblock Letters in Mathematical Physics  \textbf{106} (2016), 675--692.

\bibitem {rieffel:DeformationQuantizationForActionsOfRd}
\textsc{Rieffel, M.~A.: }\newblock \emph{Deformation quantization for actions
  of $\RR^d$}.
\newblock Memoirs of the American Mathematical Society  \textbf{106}.506 (1993).

\bibitem {schmitt:StrictQuantizationOfCodjointOrbits}
\textsc{Schmitt, P.: }\newblock \emph{Strict quantization of coadjoint orbits}.
\newblock arXiv:1907.03185   (2019).

\bibitem
  {schoetz.waldmann:ConvergentStarProductsForProjectiveLimitsOfHilbertSpaces}
\textsc{Schötz, M., Waldmann, S.: }\newblock \emph{Convergent star products
  for projective limits of Hilbert spaces}.
\newblock Journal of Functional Analysis  \textbf{274}.5 (2018), 1381--1423.

\bibitem {stienon.xu:reductionOfGeneralizedComplexStructures}
\textsc{Sti{\'{e}}non, M., Xu, P.: }\newblock \emph{Reduction of generalized
  complex structures}.
\newblock Journal of Geometry and Physics  \textbf{58}.1 (2008), 105--121.

\bibitem {waldmann:ANuclearWeylAlgebra}
\textsc{Waldmann, S.: }\newblock \emph{A nuclear Weyl algebra}.
\newblock Journal of Geometry and Physics  \textbf{81} (2014), 10--46.

\end{thebibliography}
{
	\footnotesize
	\renewcommand{\arraystretch}{0.5}

}

\end{onehalfspace}
\end{document}